\documentclass{amsart}
\usepackage{amsfonts}
\usepackage{amsthm}
\usepackage{amsmath}
\usepackage{amscd}
\usepackage[utf8]{inputenc} 
\usepackage{t1enc}
\usepackage[mathscr]{eucal}
\usepackage{indentfirst}
\usepackage{tikz-cd}
\usepackage{graphicx}
\usepackage{graphics}
\usepackage{pict2e}
\usepackage{epic}
\numberwithin{equation}{section}
\usepackage[margin=2.9cm]{geometry}
\usepackage{epstopdf}

\usepackage{xcolor}

\usepackage[tablename=Table]{caption}
\usepackage{relsize}
\usepackage[utf8]{inputenc}
\usepackage{amsmath,amssymb}
\usepackage{amsthm}
\usepackage[all]{xy}
\usepackage{mathrsfs}
\usepackage{tikz-cd}

\newcommand{\ba}{\begin{array}}\newcommand{\ea}{\end{array}}

\usepackage{booktabs,graphicx,makecell,siunitx,eqparbox}

\usepackage{tikz}
\usepackage{tikz-cd}
\usetikzlibrary{%
  matrix,%
  calc,%
  arrows%
}

\usepackage{mathtools}
\usepackage[active]{srcltx}
\usepackage{hyperref,ifpdf}
\usepackage{graphicx}
\date{}

\DeclareMathOperator{\Sing}{Sing}

\DeclareMathOperator{\Ker}{Ker}
\DeclareMathOperator{\Ima}{Im}
\DeclareMathOperator{\rk}{rk}
\DeclareMathOperator{\Coker}{Coker}
\DeclareMathOperator{\Hom}{Hom}

\DeclareMathOperator{\codim}{codim}

\DeclareMathOperator{\Res}{Res}

\hypersetup{
	colorlinks=true,
	linkcolor=blue,  
	filecolor=magenta,      
	urlcolor=cyan,
}

\newtheorem{obs}{Remark}
\newtheorem{theorem}{Theorem}
\newtheorem{propo}{Proposition}
\newtheorem{lema}{Lemma}
\newtheorem{corolario}{Corollary}

\theoremstyle{definition}
\newtheorem{defin}{Definition}  

\newtheorem{example}{Example}

\begin{document}

\title[Dimension  two  holomorphic distributions]{Dimension  two  holomorphic distributions on  four-dimensional projective space}

\author[O. Calvo--Andrade, M. Corr\^ea, J. Fonseca--Quispe]{O. Calvo--Andrade, M. Corr\^ea,  J. Fonseca--Quispe}

\address{José Omegar Calvo--Andrade \\ Centro de Investigación en Matemáticas -- CIMAT \\ Guanajuato, México.} 

\email{jose.calvo@cimat.mx}

\address{Maurício Corrêa  \\ 
Universit\`a degli Studi di Bari, 
Via E. Orabona 4, I-70125, Bari, Italy
}
\email[M. Corr\^ea]{mauricio.barros@uniba.it,mauriciomatufmg@gmail.com }

\address{Julio  Fonseca-Quispe \\ Universidade Federal de Minas Gerais -- UFMG  \\ Belo Horizonte, Brazil.} 

\email[J. Fonseca]{julioleofq@gmail.com, julioleofq@ufmg.br}



\begin{abstract} 
 We study two-dimensional holomorphic distributions on $\mathbb{P}^4$. We classify dimension two distributions, of degree at most $2$,   with either locally free tangent sheaf or locally free  conormal sheaf   and whose singular scheme has pure dimension one.  We show that the corresponding sheaves are split. Next, we investigate the geometry of such distributions, studying from maximally non–integrable  to  integrable distributions.  In the 
maximally non–integrable case,   we   show that     the distribution  is either  of  Lorentzian type or  a   push-forward by a rational  map   of the Cartan prolongation of a singular contact  structure   on a weighted projective 3-fold. We   study distributions of dimension two in  $\mathbb{P}^4$  whose conormal  sheaves  are the  Horrocks--Mumford sheaves, describing the  numerical invariants of their singular schemes which are smooth and connected. Such distributions  are  maximally non–integrable,   uniquely determined by their singular schemes and  invariant by a  group $H_5 \rtimes SL(2,\mathbb{Z}_5) \subset Sp(4, \mathbb{Q})$, where $H_5$ is the Heisenberg group of level $5$.  We prove that 
the moduli spaces of  Horrocks–Mumford distributions  are irreducible quasi-projective varieties and we determine their dimensions. Finally, we observe that the space of codimension one distributions, of degree $d\geq 6$, on $\mathbb{P}^4$ has a family of degenerate flat holomorphic Riemannian metrics. Moreover, the degeneracy divisors of such metrics consist of codimension one distributions invariant by $H_5 \rtimes SL(2,\mathbb{Z}_5)$ and singular along a degenerate abelian surface with $(1,5)$-polarization and level-$5$-structure.
\end{abstract}

\maketitle

\tableofcontents

\section{Introduction}
In this work, we are interested in studying, in the algebro-geometric spirit of \cite{MOJ} and \cite{MSM}, singular holomorphic distributions of dimension $2$ on $\mathbb{P}^4$ whose associated sheaves are locally free and the singular schemes have pure dimension one.

The study of distributions and foliations has emerged from works of classical geometers such as Pfaff,  Jacobi, Grassmann, Frobenius,  Darboux, Poincar\'e  and Cartan, see  \cite{Forsyth}. 
Holomorphic distributions and Pfaff systems appear implicitly in the theory of reflexive sheaves and the study of associated degeneracy loci, especially, as we will see, in the context of locally free sheaves. Indeed, by an Ottaviani-Bertini-type theorem \cite[Theorem A.3]{MOJ} we have the following:

If $E$ is a reflexive sheaf on a polarized projective manifold $(X, \mathcal{O}_X(1))$, of rank $s< \dim(X)$. Then, there exist an integer $m>>0$, such that
\begin{itemize}
\item $E\otimes \mathcal{O}_X(-m)$ is the tangent sheaf of a distribution   of dimension $s$ and
\item   $E\otimes \mathcal{O}_X(-m)$ is the conormal sheaf of a distribution  of dimension $n-s$,
\end{itemize}
since there  is  an integer $m>>0$ such that $TX\otimes E^*\otimes\mathcal{O}_X(m)$ and $\Omega_X^1\otimes E^*\otimes \mathcal{O}_X(m)$ are globally generated.

Classical works of  algebro-geometers such as  Castelnuovo \cite{Cas}, Palatini \cite{Palatini} and  Fano \cite{Fano},  in their study of varieties given by the set of centers of complexes belonging to the linear systems of linear complexes    can be   interpreted in modern terms \cite{Ottaviani2, Faenzi-Fania} as  schemes   given by the degeneracy locus of a generic  injective  map
$$
\mathcal{O}_{\mathbb{P}^n}(-2)^{\oplus m}\to 
\Omega_{\mathbb{P}^n}^1,\quad m\leq n-1. 
$$
That is, as the singular locus of codimension $m$ distributions whose the conormal sheaf are isomorphic to the locally free sheaf $\mathcal{O}_{\mathbb{P}^n}(-2)^{\oplus m }$. 
We denote as usual by $\mathcal{M}^{st}(c_1,c_2,\dots, c_n)$ the moduli space of stable reflexive sheaves on $\mathbb{P}^n$ with Chern classes $c_i$, with $i=1,\dots,n$. In \cite{Okonek1} Okonek showed that 
$$
\mathcal{M}^{st}(-1,1,1,2)=\left\{\mbox{coker}(\phi);\phi: \mathcal{O}_{\mathbb{P}^4}(1)^{\oplus 2}\to T\mathbb{P}^4\right \}
$$
and $\mathcal{M}^{st}(-1,1,1,2)$ that is isomorphic to $\mathbb{G}(1,4)$ the grassmannian of lines in $\mathbb{P}^4$. From the distribution point of view the space
$\{\mbox{coker}(\phi);\phi: \mathcal{O}_{\mathbb{P}^4}(1)^{\oplus 2}\to T\mathbb{P}^4\}$ corresponds to the space of holomorphic foliations of degree zero in $\mathbb{P}^4$ and this space is given by
$\{\mbox{linear projections}\  \mathbb{P}^4 \dashrightarrow \mathbb{P}^2 \}$ which is naturally isomorphic to $\mathbb{G}(1,4)$.
It is easy to see that any rank two distribution on $\mathbb{P}^n$, with $n>3$, must be integrable, and it is given by a linear projection $\mathbb{P}^n \dashrightarrow \mathbb{P}^{n-2}$, see Proposition \ref{Distri-zero}. Okonek also proved in \cite{Okonek2} that 
$$
\mathcal{M}^{st}(-1,2,2,5)=\left\{\mbox{coker}(\phi);\phi: \mathcal{O}_{\mathbb{P}^4}(-2)^{\oplus 2}\to \Omega_{\mathbb{P}^4}^1\right \}.
$$
Therefore,  such moduli space has as elements dimension two distributions in $\mathbb{P}^4$, whose conormal sheaves are isomorphic to the locally free sheaf $\mathcal{O}_{\mathbb{P}^4}(-2)^{\oplus 2}$. Also, from the distribution
point of view, we can see that the space $\mathcal{M}^{st}(-1,2,2,5)$ is isomorphic to an open Zariski subset of the Grassmannian $\mathbb{G}(1,9)$, whose dimension is $16$, see remark \ref{Okonek2}.

Another important class of stable vector bundle that occurs as a tangent bundle of distributions is the so-called null-correlation bundle $N$ given by
$$
0\to N(1)\to T\mathbb{P}^{2r+1}\to \mathcal{O}_{\mathbb{P}^{2r+1}}(2)\to 0. 
$$
That is,  $N(1)$ is  the tangent bundle of a contact structure on $\mathbb{P}^{2r+1}$.

 In the context of  foliations by curves in $\mathbb{P}^3$ \cite{MSM},  we have an interesting example that can be found in the construction of  mathematical 4-instanton bundles appearing in the work of    Anghel,   Coanda and Manolache \cite{AnCM}.
  They showed that if   $C$ is  the disjoint union of five lines in $\mathbb{P}^3$ such that their union admits no 5-secant, 
then there exists an epimorphism
\begin{equation}\nonumber \label{eq-surj-5disj}
\Omega_{\mathbb{P}^3}^1 \stackrel{\varpi}{\longrightarrow} \mathcal{I}_C(2)\to 0. 
\end{equation}
Therefore  $E=\ker\varpi(3)$ (the conormal sheaf of the induced foliation) is an instanton bundle of charge $4$ such that $E(2)$ is globally generated.     The above foliation can be found in the classification of foliations by curves of degree three on $\mathbb{P}^3$ whose singular schemes have a pure dimension equal to one \cite{MSM}.   
On the other hand, Chang's characterization of    arithmetically Buchsbaum scheme curves on $\mathbb{P}^3$, with $h^1(I_Z(d - 1)) = 1$ being the only nonzero intermediate
cohomology for $H^i_{\bullet}(I_Z)$ \cite{Chang2}, says us that $Z$ is given by the  degeneracy locus of an injective map
\[
\mathcal{O}_{\mathbb{P}^3}(-d_1) \oplus  \mathcal{O}_{\mathbb{P}^3}(-d_2)\to \Omega_{\mathbb{P}^3}^1.
\]
These are foliations by curves which are given by a  global complete intersection of two codimension one distributions in $\mathbb{P}^3$, see \cite[Corollary 3]{MMR}. 

In  \cite[Table I, pg 104]{Chang}  appear  Pfaff systems in $\mathbb{P}^4$  whose singular schemes are very special algebraic varieties.  More precisely,   let  $Z$  be a   smooth non-general type Buchsbaum surface on $\mathbb{P}^4$ such that  $K_Z^2\notin\{  -2,-7\}$. Then, we have the following possibilities:   
\begin{enumerate}
\item   $Z$ is a singular locus of a Pfaff system of dimension $2$ and degree $2$ and it is a quintic elliptic scroll.   

\item $Z$ is a singular locus of a foliation by curves of degree $1$ and it is a projected Veronese surface. 

\item $Z$ is a singular locus of a Pfaff system of  dimension $2$ and degree $3$ and  it  is a $K3$ surface of genus $8$. 

\item $Z$ is a singular locus of a foliation by curves of degree $3$ and it is a $K3$ surface of genus $7$. 
\end{enumerate}

Those   Pfaff systems  of  dimension $2$,  degree $2$ and singular along  a quintic elliptic scroll also  occur in \cite[Remark 4.7, (i)]{DAPHR}   where the authors  study   Abelian and bielliptic surfaces on $\mathbb{P}^4$  and, more recently, Rubtsov in \cite{Rub}
 proves that these Pfaff systems represent generically symplectic structures on  $\mathbb{P}^4$   associated with Feigin--Odesskii--Sklyanin Poisson algebras.  We refer to the recent work of  Polishchuk \cite{ Polishchuk2} which shows a very interesting connection between such Poisson brackets, syzygies, and Cremona transformations.

In   \cite{MOJ}  the authors initiated a systematic study of    codimension one holomorphic distributions
on projective three-spaces, analyzing the properties of their singular schemes and tangent
sheaves. They showed that there exists a quasi-projective variety
that parameterizes isomorphism classes of distributions, on a complex projective manifold whose tangent sheaves have a fixed Hilbert polynomial.
In  \cite{GJM}  Galeano,   Jardim and Muniz provide a classification of distributions of degree $2$ in   $\mathbb{P}^3$. In \cite{MSM} the authors give the first step to the study of foliations by curves on the three-dimensional projective space with no isolated singularities. They provide a classification of such foliations by curves and degree $\leq 3$ proving that the foliations of degree $1$ or $2$    are contained in a pencil of planes or are Legendrian and are given by the complete intersection of two codimension one distributions, and that the conormal sheaf of a foliation by curves of degree $3$ with reduced singular scheme either splits as a sum of line bundles or is an instanton bundle.  For the classification of integrable distributions and low degree, we refer the reader to
\cite{Jou, LPT2,  CD2,  CLN,  CLPdeg3,M-A}.

 Our first result is the classification of the distributions of small degrees. 
 Firstly, we observe in  Proposition \ref{Distri-zero} that 
distribution of dimension $2$ and degree zero on $\mathbb{P}^n$ is either integrable and it is tangent to a linear rational projection $\mathbb{P}^n \dashrightarrow \mathbb{P}^{n-2}$ or   $\mathscr{F}$ is a contact structure on $\mathbb{P}^3$. Next, we investigate distributions of degrees $1$ and $2$ with associated sheaves being locally free and with singular schemes of pure dimension one describing the geometry of the corresponding first-derived distributions. 

Let $\mathscr{F}$ be a distribution of dimension $2$ on a complex 4-fold. Roughly speaking, the non-integrability condition of $\mathscr{F}$ is measured by its associated \textit{first derived distribution}, defined as follows. Suppose that $\mathscr{F}$ is not integrable, and define the 3-dimensional holomorphic distribution $\mathscr{F}^{[1]} := ( T\mathscr{F} + [T\mathscr{F},T\mathscr{F}] )^{**} \subset TX$ as the first derived distribution of $\mathscr{F}$.
A  maximally non-integrable distribution $\mathscr{F}$ of codimension  $2$ on a 4-fold $X$, called by \textit{Engel structure}, is a distribution such that  
$\mathscr{F}^{[1]}$ is not   integrable.  
Consider the O'Neill tensor
 \begin{alignat*}{2}
 \mathcal{T}(\mathscr{F}^{[1]}) :\wedge^{[2]}\mathscr{F}^{[1]}  &\longrightarrow&  TX/\mathscr{F}^{[1]} \\
 u\wedge  v &\longmapsto&  \pi([u,v]), 
\end{alignat*}
where $\pi: \mathscr{F}^{[1]}\to TX/\mathscr{F}^{[1]}$ denotes the projection
 and $\wedge^{[2]}\mathscr{F}^{[1]}:=(\wedge^{2}\mathscr{F}^{[1]})^{**}$.   
Then, there is a unique foliation by curves $  \mbox{Ker}(\mathcal{T}(\mathscr{F}^{[1]})^{**}:=\mathcal{L}(\mathscr{F}) \subset \mathscr{F}^{[1]} $, called \textit{ characteristic foliation} of $\mathscr{F}^{[1]}$, such that $[\mathcal{L}(\mathscr{F}^{[1]}),\mathscr{F}^{[1]}]\subset \mathscr{F}^{[1]}$.

 Our first result is the following. 
\begin{theorem}
Let $\mathscr{F}$ be a dimension two distribution on $\mathbb{P}^4$ of degree $d\in \{1,2\}$ with a locally free tangent sheaf $T\mathscr{F}$ and whose singular scheme has pure dimension one. Then $T\mathscr{F}$ splits as a sum of line bundles, and
\begin{enumerate}

    \item     $T\mathscr{F}=\mathcal{O}_{\mathbb{P}^4}(1)\oplus\mathcal{O}_{\mathbb{P}^4}$ and its singular scheme  is a  rational normal curve  of degree $4$.
     \item   $T\mathscr{F}=\mathcal{O}_{\mathbb{P}^4}\oplus\mathcal{O}_{\mathbb{P}^4}$, and its singular scheme is an arithmetically Cohen--Macaulay curve of degree $10$ and arithmetic genus $6$.
    
     \item    $T\mathscr{F}=\mathcal{O}_{\mathbb{P}^4}(1)\oplus\mathcal{O}_{\mathbb{P}^4}(-1)$, and its singular scheme is an arithmetically Cohen-Macaulay curve of degree $15$ and arithmetic genus $17$.

\end{enumerate}
\noindent If  $\mathscr{F}$ is integrable, then      it is either   a linear pull-back of a degree $d$  foliation by curve on $\mathbb{P}^3$ or  $T\mathscr{F}\simeq \mathfrak{g}\otimes \mathcal{O}_{\mathbb{P}^4}$,  where      either  $\mathfrak{g}$ is   an abelian Lie algebra of dimension $2$  or  $\mathfrak{g}\simeq \mathfrak{aff}(\mathbb{C})$.

\noindent If $\mathscr{F}$ is not integrable and $\mathscr{F}^{[1]}$ is integrable, denote $\mathfrak{f}^{[1]} := H^0(\mathscr{F}^{[1]})$ as the Lie algebra of holomorphic vector fields tangent to $\mathscr{F}^{[1]}$.  
Then: 
     \begin{enumerate}

  \item[$a)$] $\mathscr{F}^{[1]}$ has degree $\leq 3$. It is a linear pullback of a codimension one foliation $\mathscr{G}$ in $\mathbb{P}^3$ or $\mathbb{P}^2$
    
          \item[$b)$]   
          $\mathscr{F}^{[1]}$ 
has degree $3$ and 
$T\mathscr{F}^{[1]}\simeq \mathfrak{g}\otimes \mathcal{O}_{\mathbb{P}^4}$,  where $\mathfrak{g} $      is  a non-abelian   Lie algebra of dimension $3$. Moreover, one of the following holds:
   \begin{itemize}
    \item  if $ \dim 
 [\mathfrak{f}^{[1]},\mathfrak{f}^{[1]}]=1$, then $\mathfrak{f}^{[1]} \simeq \mathfrak{aff}(\mathbb{C})\oplus\mathbb{C}$; 
     \item  if $\dim [\mathfrak{f}^{[1]},\mathfrak{f}^{[1]}]=2$, then  $\mathfrak{f}^{[1]}\simeq \mathfrak{r}_{3,\lambda}(\mathbb{C}):=\{[v_1,v_2]=v_2; \ [v_1,v_3]=\lambda v_3;\ [v_2,v_3]=0 \}$,  where $\lambda \in \mathbb{C}^*$ with  $\lambda=-1$, or
     $0<|\lambda|<1$, or  $|\lambda|=1$ and $\mathfrak{Im}(\lambda)>0$;
       \item  if $\dim [\mathfrak{f}^{[1]},\mathfrak{f}^{[1]}]=3$ , then  $\mathfrak{f}^{[1]}\simeq \mathfrak{sl}(2,\mathbb{C})  $. 
    
    \end{itemize}
          
    \end{enumerate}
If $\mathscr{F}$ is an Engel structure.  Then: 
     \begin{enumerate}

         \item [$c)$]   $\mathscr{F}$    is  the  blow-down  of the Cartan prolongation of a singular contact  structure  of degree $0,1, 2$ or $3$ on $\mathbb{P}^3$.     
           
          \item [$d)$]   $\mathscr{F}$ has degree $2$ and     it  is  the  blow-down  of the Cartan prolongation of a singular contact  structure    on a weighted projective $3$-space.     
          
            \item [$e)$]   $\mathscr{F}$  has degree $2$ and  it  is of Lorentzian type. 
    \end{enumerate}
   
\end{theorem}

Next, it is natural to study holomorphic distributions of dimension $2$ whose tangent bundles are not split.  The only known non-decomposable vector bundles of rank $2$ are the so-called  Horrocks-Mumford bundles \cite{HM}.  The existence of these distributions can be established as follows:

H. Sumihiro showed in \cite{Su} that $E(a)$ is generated by global sections, for every $a\geq 1$. Since $\Omega_{\mathbb{P}^4}^1(2)$ is a globally generated sheaf,  then
$$E(a)\otimes \Omega_{\mathbb{P}^4}^1(2)\simeq \mathcal{H}om (E(-a-7),\Omega_{\mathbb{P}^4}^1)$$
are also globally generated sheaves for all $a\geq 1$. Then,  by a Bertini-type Theorem \cite{MOJ, Ottaviani}  we obtain dimension $2$ holomorphic distributions on $\mathbb{P}^4$ whose conormal sheaves are isomorphic to $E(-a-7)$   for all  $a\geq 1$.   We investigate such distributions, proving that they have interesting geometric properties. We determine the numerical invariants of their singular schemes, which are smooth and connected,  they are 
 maximally non–integrable,   uniquely determined by their singular schemes, and have many symmetries.

  In \cite{MSM} the authors have introduced a new invariant for distributions, called the Rao module, which appears in the classification of locally complete intersection codimension-two foliations in $\mathbb{P}^3$ which we recall as follows:
  
Let $\mathscr{F}$ be a holomorphic distribution of dimension two on $\mathbb{P}^4$ and consider the following graded module
$$R_{\mathscr{F}}:=H^1_*(\mathscr{I}_Z)=\bigoplus_{l\in\mathbb{Z}}H^1(\mathscr{I}_Z(l));$$
called the \textit{Rao module}   of $\mathscr{F}$. 
We prove the following result:
\begin{theorem} 
Let $\mathscr{F}_a$ be  a  dimension $2$  Horrocks-Mumford holomorphic distribution $E(-a-7) \to  \Omega_{\mathbb{P}^4}^1$. Then:

\begin{enumerate}
\item   $\mathscr{F}_a$   is a  degree  $2a+5$ 
 maximally non–integrable distribution, for all $a \geq 1$. 

\item  $Z_a=\Sing(\mathscr{F}_a)$ is a smooth and  connected     curve  with the following numerical invariants 
$$
\deg(Z_a)=4a^3+33a^2+77a+46,
$$
$$
p_a(Z_a)=9a^4+89a^3+\frac{553}{2}a^2+\frac{573}{2}a+45.
$$
Moreover, $Z_a$ is  never contained in a hypersurface of degree  $\leq 2a+5= \deg(\mathscr{F}_a)$.   
\item  For all $ a\geq4$, we have
$\dim_{\mathbb{C}} R_{\mathscr{F}_a}=401$. 

\item  If   $\mathscr{F}'$ is a dimension two distribution on  $\mathbb{P}^4$, with   degree  $2a+5$, such that $\Sing(\mathscr{F}_a) \subset \Sing(\mathscr{F}') $, then $ \mathscr{F}'=\mathscr{F}_a$.

\item $\mathscr{F}_a$ is invariant by a group  $\Gamma_{1,5}\simeq  H_5 \rtimes SL(2,\mathbb{Z}_5) \subset Sp(4, \mathbb{Q})$, where $H_5$ is the Heisenberg group of level $5$ generated by 
\begin{center}
   $\sigma: z_k \to z_{k-1}$  and    $\tau: z_k \to  \epsilon^{-k} z_{k}$, with  $k\in \mathbb{Z}_5$ and $\epsilon= e^{\frac{2\pi i}{5}}$.  
\end{center}

\end{enumerate}
\end{theorem}
This result provides a new family of Pfaff systems of rank $2$ and degree $d\geq 6$ that are invariant by a group that contains the Heisenberg group $H_5$.  In the study of Abelian and bielliptic surfaces   on $\mathbb{P}^4$  appears  a pencil of degree $2$ Pfaff systems invariant by $H_5$ and singular along $H_5$-invariant elliptic quintic scrolls (possibly degenerate); see \cite[Remark 4.7, (i)]{DAPHR}. Rubtsov in \cite{Rub} shows that such a family represents generically symplectic structures on  $\mathbb{P}^4$   which correspond to Feigin--Odesskii--Sklyanin Poisson algebras; see Example \ref{Pencil-distr}. A degeneration of  such a family  provides us   foliations of degree zero which are $H_5$-invariant,   Proposition  \ref{G_5-inv}.

We denote by ${\rm HM}\mathscr{D}ist(2a+5)$  
the moduli spaces of Horrocks-Munford Holomorphic  distributions of degree  $2a+5$. 
\begin{theorem} 
The moduli space $
    {\rm HM}\mathscr{D}ist(2a+5) $ of dimension $2$ Horrocks-Mumford holomorphic distributions, of degree $2a+5$, 
is an irreducible  quasi-projective variety of dimension
$$\frac{1}{3}a^4+7a^3+\frac{277}{6}a^2+\frac{199}{2}a+43$$
for all $a\geq 1$.
\end{theorem}

Let $\mathscr{H}_{d(a),g(a)}$ be the Hilbert scheme of smooth curves on $\mathbb{P}^4$ of degree $d(a)=4a^3+33a^2+77a+46$ and genus $g(a)=9a^4+89a^3+\frac{553}{2}a^2+\frac{573}{2}a+45$. 
The part (4) of the  Theorem 3 says us that  we have a $\Gamma_{1,5}$-equivariant    injective map
\begin{alignat*}{2}
   {\rm HM}\mathscr{D}ist(2a+5) &\longrightarrow& \mathscr{H}_{d(a),g(a)} \\
  [\mathscr{F}_a]&\longmapsto& \Sing(\mathscr{F}_a). 
\end{alignat*}

We also prove the following result. 

\begin{theorem}\label{Conformal--HM}
There exists a family $\{g_{\phi}\}_{\phi \in \mathcal{A}}$  of $\Gamma_{1,5}$-equivariant flat holomorphic  conformal  structure on the space of distributions of codimension one and degree $d\geq 6$, where $\mathcal{A} \subset  \mathbb{P}^M$ is a Zariski open  with 
$$
 M= \frac{1}{3}(d-5)^4+\frac{23}{3}(d-5)^3+\frac{343}{6}(d-5)^2+\frac{899}{6}(d-5)+74, 
$$
such that :    

\begin{enumerate}
   \item  there is a  rational map $\pi_{\phi}: \mathbb{P}H^0(\mathbb{P}^4, \Omega_{\mathbb{P}^4}^1(d+2)) \dashrightarrow \mathcal{H}_2/\Gamma_{1,5} $ with rational fibers, a rational  section  $s_{\phi}$ whose image 
consists of   codimension one distributions, of degree $d$, invariant by $\Gamma_{1,5}$ and  singular along to an abelian surface with $(1,5)$-polarization and level-$5$-structure.

    \item $g_{\phi}=\pi_{\phi}^*g_0$, where $g_0$ is the flat holomorphic conformal  structure of  $\mathcal{H}_2/\Gamma_{1,5}$ 
degenerating  along a hypersurface $\Delta_\phi$ of degree $10$ which is a cone over a  rational sextic curve in $\mathbb{P}^3$, and  $\Delta_\phi$ consists of codimension one distributions, of degree $d$,  invariant by $\Gamma_{1,5} $ and  singular along to  either: 
\begin{itemize}
    \item a  translation scroll associated with a normal elliptic quintic curve;   
     \item  or the tangent scroll of a normal elliptic quintic curve;
     \item   or a quintic elliptic scroll carrying a multiplicity-2 structure;
     
     \item or a union of five smooth quadric surfaces;
     
     \item   or a union of five planes with a multiplicity-2 structure.
    \end{itemize}
 
\end{enumerate}
\end{theorem}

\subsection*{Acknowledgments}
We would   like to thank Marcos Jardim, Alan Muniz and Israel Vainsencher for useful discussions. 
MC is partially supported by the Universit\`a degli Studi di Bari and by the
 PRIN 2022MWPMAB- "Interactions between Geometric Structures and Function Theories" and he is a member of INdAM-GNSAGA;
he was  partially supported by CNPq grant numbers 202374/2018-1, 400821/2016-8 and  Fapemig grant numbers  APQ-02674-21, APQ-00798-18,  APQ-00056-20. FQ  was partially supported by CAPES  and  CIMAT; he is grateful to  CIMAT   for its hospitality.

\section{Preliminaries}

\subsection{Basic definitions and Theorems}
 Let $X$ be a complex manifold. By a  sheaf we mean a coherent analytic sheaf of $\mathcal{O}_X$-modules. In this work, we will not distinguish   holomorphic vector bundles and locally free  sheaves.

Recall that the cohomology ring $H^{\ast}(\mathbb{P}^n,\mathbb{Z})=\mathbb{Z}[\mathbf{h}]/\mathbf{h}^{n+1}$, where $\mathbf{h}\in H^2(\mathbb{P}^n,\mathbb{Z})$ denotes the Poincar\'e dual of a hyperplane.
   Let $E$ be a sheaf on $\mathbb{P}^n$, we denote $E(k)=E\otimes\mathcal{O}_{\mathbb{P}^n}(1)^{\otimes k}$, $k\in\mathbb{Z}$, and $h^i(\mathbb{P}^n,E):=\dim_{\mathbb{C}}H^i(\mathbb{P}^n,E)$. Moreover, we will consider the \textit{Euler sequence} given by 
\begin{equation*}
0\rightarrow\mathcal{O}_{\mathbb{P}^n}\rightarrow\mathcal{O}_{\mathbb{P}^n}(1)^{\oplus (n+1)}\rightarrow T\mathbb{P}^n\rightarrow 0.  
\end{equation*}

The total Chern class of a vector bundle $r$ of rank $E$ over $\mathbb{P}^n$ will be denoted by $$c(E)=1+c_1(E)\cdot\textbf{h}+\cdots+c_r(E)\cdot\textbf{h}^{r},$$ where $c_i(E)\in H^{2i}(\mathbb{P}^n;\mathbb{Z})\simeq\mathbb{Z}$ will be considered mostly as integers. 
The total Chern class of vector bundles $E(k)$, $c(\wedge^pE)$ and $c(\mbox{Sym}^2(E))$ can  be calculated in terms of the Chern classes of $E$ by the following formulas:
\begin{itemize}
    \item let $E$ is a rank $2$ vector bundle. Then for all $k\in\mathbb{Z}$ we have:
\begin{equation*}
c\big(E(k)\big)=1+\big(c_1(E)+2k\big)\cdot\textbf{h}+\big(c_2(E)+k\cdot c_1(E)+k^2\big)\cdot\textbf{h}^2;   
\end{equation*} 

\item Let $E$ be a vector bundle of rank $3$. Then for all $k\in\mathbb{Z}$ we have:   
\begin{equation*}
c\big(E(k)\big)=1+\big(c_1(E)+3k\big)\cdot\textbf{h}+\big(c_2(E)+2k\cdot c_1(E)+3k^2\big)\cdot\textbf{h}^2 + \big(c_3(E)+k\cdot c_2(E)+k^2c_1(E)+ k^3\big)\cdot\textbf{h}^3;    
\end{equation*} 

\item if $E$ is a vector bundle of rank $4$ with Chern classes $c_1, c_2, c_3, c_4$, then:

\begin{equation*}
    \begin{split}
        c(\wedge^2E) & = 1 + 3c_1\textbf{h} + (3c_1^2+2c_2)\textbf{h}^{2} + (c_1^3+4c_1c_2)\textbf{h}^{3} + (2c_1^2c_2+c_2^2+c_1c_3-4c_4)\textbf{h}^{4},
    \end{split}
\end{equation*}
\begin{equation*}
    \begin{split}
        c(\wedge^3E) & = 1 + 3c_1\textbf{h} + (3c_1^2+c_2)\textbf{h}^{2} + (c_1^3+2c_1c_2-c_3)\textbf{h}^{3} + (c_1^2c_2-c_1c_3+c_4)\textbf{h}^{4};
    \end{split}
\end{equation*}

\item for a vector bundle $E$ of rank $2$ with Chern classes $c_1$ and $c_2$, the Chern class of symmetric power $\mbox{Sym}^2(E)$ is given by: 
\begin{equation*}
    \begin{split}
        c( \mbox{Sym}^2(E)) & = 1+ 3c_1\textbf{h} + (2c_1^2+4c_2)\textbf{h}^{2} + 4c_1c_2\textbf{h}^{3}.
    \end{split}
\end{equation*}

\end{itemize}
Let $E$ and  $F$ be vector bundles of ranks $4$ and $2$,  respectively, then
\begin{equation}\label{chernclass42}
\begin{split}
c_1(E\otimes F) & = 2e_1+4f_1, \\
c_2(E\otimes F) & = e_1^2+7e_1f_1+6f_1^2+2e_2+4f_2, \\
c_3(E\otimes F) & = 3e_1^2f_1 + 9e_1f1^2  + 4f_1^3  + 2e_1e_2 + 6e_2f_1 + 6e_1f_2 + 12f_1f_2 + 2e_3,\\
c_4(E\otimes F) & = 3e_1^2f_1^2  + 5e_1f_1^3  + f_1^4  + 5e_1e_2f_1 + 7e_2f_1^2  + 3e_1^2f_2 + 15e_1f_1f_2 + 12f_1^2f_2+\\
 & + e_2^2  + 2e_1e_3+ 5e_3f_1 + 2e_2f_2 + 6f_2^2 + 2e_4,
\end{split}
\end{equation}
where  $e_1,e_2,e_3,e_4$ and $f_1,f_2$, denotes the   Chern roots of $E$ and $F$, respectively. 

By Grothendieck-Riemann-Roch Theorem (\cite[page   157]{Maeda}) one gets:

\begin{corolario}\label{chernclassidealsheaf}
Let $Y$ be a non-singular curve and $i:Y\rightarrow\mathbb{P}^4$ a closed embedding. Then 
\[
\begin{array}{rcl}
c_3(\mathscr{I}_Y) & = & -2\deg(Y)\\
c_4(\mathscr{I}_Y) & = & 6-15\deg(Y)-6 p_a(Y)
\end{array}
\]
where $p_a(Y)$ denotes the arithmetic genus of $Y$.
\end{corolario}

\begin{theorem}[Bogomolov's inequality]
If $E$ is a stable rank $2$ holomorphic vector bundle on $\mathbb{P}^n$, then  $c_1(E)^2-4c_2(E)<0$.
\end{theorem}

Let $\mathcal{F}$ be a rank two torsion free sheaf on $\mathbb{P}^n$. We say that $\mathcal{F}$ is \emph{normalized} if it has first Chern class $c_1(\mathcal{F})=0,-1$.  For any rank two torsion free  sheaf $\mathcal{F}$ in $\mathbb{P}^n$ with first Chern class $c_1(\mathcal{F})=c$, we define its \textit{normalization} $\mathcal{F}_{\eta}$ as follows:
\begin{equation}\label{normalization}
\mathcal{F}_{\eta} : = \left\lbrace
\begin{array}{lccc}
\mathcal{F}\big(-\frac{c}{2}\big),  & \textrm{if c is even}\\
\\ \mathcal{F}\big(-\frac{c+1}{2}\big), & \textrm{if  c is odd}.
\end{array}
\right.
\end{equation}
It follows that $\mathcal{F}_{\eta}$ has first Chern class $-1$ if $c$ is odd and $0$ if $c$ is even. The following is a very useful  result: 

\begin{propo}\cite[ pages 84 and 87]{Oko} 
Let $\mathcal{F}$ be a rank two reflexive sheaf on $\mathbb{P}^n$. Then
\begin{enumerate}
    \item $\mathcal{F}$ is stable if and only if $H^0(\mathbb{P}^n, \mathcal{F}_{\eta})=0$.
    \item $\mathcal{F}$ is simple, i.e., $H^0(\mathbb{P}^n,\mathcal{H}om(\mathcal{F}, \mathcal{F}))=\mathbb{C}$.
\end{enumerate}
\end{propo}

We say that a holomorphic vector bundle on $\mathbb{P}^n$ of rank $r$ \textit{splits} when it can be represented as a direct sum of $r$ holomorphic line bundles.

\subsection{Degeneracy loci of maps of vector bundles}
Let $E$ and $F$ be holomorphic bundles on a complex manifold $X$ such that $\mathrm{rk}(E) = m$ and $\mathrm{rk}(F) = n$, and let $\varphi: E \to F$ be a homomorphism of vector bundles. Define
$$
D_k(\varphi) = \{ p \in X \mid \mathrm{rk} \, \varphi(p) \leq k \}.
$$

We will use the following result proved in \cite[page 16]{Ottaviani}.

\begin{theorem}[Bertini type]\label{Bertini1}
Let $E,\, F$ be vector bundles on a complex manifold $X$ such that $\rk (E)=m,\,\rk( F)=n$. Assume that the bundle $E^*\otimes F$ is generated by the global sections. Then, if $\varphi:E\to F$ is a generic morphism, then one of the following holds:

\begin{itemize}
    \item $D_k(\varphi)$ is empty.
    \item $D_k(\varphi)$ has the expected codimension $(m-k)(n-k)$ and $\Sing(D_k(\varphi))\subset D_{k-1}(\varphi)$.
\end{itemize}
In particular if $\dim X <(m - k + 1)(n - k + 1)$ then $D_k(\varphi)$ is empty or smooth when $\varphi$ is generic.
\end{theorem}

\begin{theorem}\label{conexidade}
Let $X$ be an irreducible complex projective variety of dimension $n$ and let $\varphi:E\to F$ be a homomorphism of vector bundles on X of ranks $e$ and $f$. Assume that the vector bundle $E^*\otimes F=\Hom(E,F)$ is ample. Then:

\begin{itemize}
    \item $D_k(\varphi)$ is non-empty if $n\geq (e-k)(f-k)$.
    \item $D_k(\varphi)$ is connected when $n>(e-k)(f-k)$.
\end{itemize}
\end{theorem}

\begin{proof}
Vide \cite[page 273]{FL}.
\end{proof}

\subsection{The Eagon-Northcott resolution}

\noindent Here we recall a particular case of the Eagon-Northcott resolution. Let \( G \) and \( F \) be complex vector bundles of ranks 4 and 2, respectively, over a complex manifold \( X \) of complex dimension \( n \geq 4 \). Let \( \phi: G \to F \) be a generically surjective morphism. We assume that the degeneracy locus \( Z \) is a scheme of pure codimension 3.
 The Eagon-Northcott resolution looks the following way.
\begin{equation}\label{ENR}
    0\to\det(G)\otimes \mbox{Sym}^2(F^*)\otimes\det(F^*)\to\bigwedge^3G\otimes F^* \otimes \det(F^*)\to\bigwedge^2G\otimes\det(F^*)\to\mathscr{I}_Z\to 0.
\end{equation}
For more details, see \cite[A2.6]{Ei}. 

\subsection{Holomorphic conformal structures}
 Let $X$  be a complex manifold. A non-degenerate tensor $g\in H^0(X, \mbox{Sym}^2\Omega_X^1\otimes \mathscr{L} )$, where $\mathscr{L}$ is a line bundle, is  called by   holomorphic conformal structure on $X$.

\section{Holomorphic distributions}

Let $X$ be a complex manifold of complex dimension $n=k+s$. A saturated codimension $k$ singular holomorphic distribution on $X$ is given by a short exact sequence of analytic coherent sheaves
\begin{equation}\label{distribution}
    \mathscr{F} : 0 \rightarrow T\mathscr{F}\xrightarrow{\varphi} TX\xrightarrow{\pi} N \mathscr{F}\rightarrow 0,
\end{equation}

The sheaf $N \mathscr{F}$ is the \textit{normal sheaf} of $\mathscr{F}$, is a non-trivial torsion-free sheaf of generic rank $k$ on $X$. The coherent sheaf $T \mathscr{F}$ is the \textit{tangent sheaf} of $\mathscr{F}$ of rank $s=n-k\geq 1$. Note that $T \mathscr{F}$ must be reflexive \cite[Proposition 1.1]{Har2}.

The \textit{singular scheme} of $\mathscr{F}$, denoted by $\Sing(\mathscr{F})$, is defined as follows. Taking the maximal exterior power of the dual morphism $\varphi^\vee:\Omega_X^1\to T\mathscr{F}^*$ we obtain a morphism $$\Omega_X^s\to\det(T\mathscr{F})^*$$ the image of such morphism is an ideal sheaf $\mathscr{I}_Z$ of a closed subscheme $Z_{\mathscr{F}}\subset X$ of codimension $\geq 2$, which is called the singular scheme of $\mathscr{F}$, twisted by $\det(T\mathscr{F})^*$.
As a subset of $X$  $\Sing(\mathscr{F})$  corresponds to the singular set of its normal sheaf. 

The space of the holomorphic distributions of codimension $k$ in $X$ may be identified by a class of sections 
\[
[\omega]\in\mathbb{P}H^0(X,\Omega_X^k\otimes \det ( N\mathscr{F})).
\]

The distribution is \textit{involutive} if $[T \mathscr{F}, T \mathscr{F}]\subset T\mathscr{F}$ and by the Frobenius theorem, the distribution is \textit{integrable}, i.e., tangent to a foliation.

We may consider a dual perspective: Consider the dual map $N\mathscr{F}^*\subset\Omega_{X}^1$, that we identify with the annihilator of the sheaf $T\mathscr{F}$. We get an exact sequence
\begin{equation}
    \mathscr{F}^{o} : 0 \rightarrow N\mathscr{F}^*\to\Omega_{X}^1\to\mathcal{Q}_{\mathscr{F}}\rightarrow 0.
\end{equation}
The sheaf $N\mathscr{F}^*$ has generic rank $k$ is reflexive and $\mathcal{Q}_{\mathscr{F}}$ is torsion free of rank $q$. With this formulation, the distribution is involutive if the following holds:

Consider the ideal $\mathcal{J}\subset\Omega_X^*$ generated by $N\mathscr{F}^*$, then the distribution $\mathscr{F}^{o}$ is involutive if $d\mathcal{J}\subset\mathcal{J}$, and then again the Frobenius Theorem implies that it is integrable, i.e., $N\mathscr{F}^*$ is the \textit{conormal sheaf} of a foliation.

The integrability condition can also be defined as follows. Let $U\subset X$ a open set and $\omega\in\Omega_X^k(U)$. For any $p\in U\setminus\Sing(\omega)$ there exist a neighborhood $V$ of $p$, $V\subset U$, and $1$-forms $\eta_1, \dots ,\eta_k\in\Omega_X^1(V)$ such that: 
$$\omega\mid_V=\eta_1\wedge\cdots\wedge\eta_k.$$ 
We say that $\omega$ satisfies the integrability condition if and only if,  $$d\eta_j\wedge\eta_1\wedge\cdots\wedge\eta_r=0, $$   for all $j=1, \dots ,k.$

For holomorphic  distributions on the projective space $\mathbb{P}^n$, we define the most important discrete invariant, the \textit{degree} of $\mathscr{F}$. 
The degree of the distribution $\mathscr{F}$, denoted by $d:=\deg(\mathscr{F})\geq 0$, is defined by
$$\deg(\mathscr{F})=c_1(N\mathscr{F})-\codim(\mathscr{F})-1.$$
Moreover, we also have that $c_1(T\mathscr{F})=\dim(\mathscr{F}) -\deg(\mathscr{F})$.

\section{Dimension two distributions on $\mathbb{P}^4$}

In this section we will  consider  dimension $2$ holomorphic distributions on $X=\mathbb{P}^4$ with either locally free tangent or  conormal sheaves  and  whose   singular schemes  have pure  dimension one. 

Let us recall the following result which classifies non-singular holomorphic distributions on complex projective spaces due to Glover, Homer, Stong \cite{GHS}.

\begin{theorem}[Glover--Homer--Stong]\label{non-sing}
A holomorphic distribution $\mathscr{F}$ on   $\mathbb{P}^n$ is non-singular if and only if $n$ is odd,  $\mathscr{F}$ has codimension one and degree zero.   
\end{theorem}
In \cite{GHS}, Glover, Homer, and Song do not directly mention the degree of the distribution. For zero-degree distributions, we refer to \cite[Prop. 2.1, p. 85]{Jou}.

Since any  distribution $\mathscr{F}$ of dimension $2$ and degree zero on $\mathbb{P}^n$, with $n>3$, is given in homogeneous coordinate  by a constant bi-vector $\sigma$ of rank 2, then clearly  $\sigma$ is decomposable and integrable, thus  $\mathscr{F}$ is tangent   to a linear rational projection
$\mathbb{P}^n \dashrightarrow \mathbb{P}^{n-2}$. If $n=3$, then
$\mathscr{F}$ either is a   pencil   of planes or it  is a contact structure on $\mathbb{P}^3$.
By this observation,  we have the following. 
\begin{propo}\label{Distri-zero}  
 Let $\mathscr{F}$ be a  distribution of dimension $2$ and degree zero on $\mathbb{P}^n$.  Then,  $\mathscr{F}$ either  is integrable and tangent to a linear rational projection $\mathbb{P}^n \dashrightarrow \mathbb{P}^{n-2}$ or is    contact structure on $\mathbb{P}^3$.    
\end{propo}

We recall that a contact structure $\mathscr{F}$ on $\mathbb{P}^{3}$ is given by 
$$
0\to T\mathscr{F}\simeq N(1)\to T\mathbb{P}^{3}\to \mathcal{O}_{\mathbb{P}^{3}}(2)\to 0. 
$$
where  $N$ is   a null-correlation bundle  on $\mathbb{P}^{3}$, see \cite[Example 5.2]{MOJ}.

\subsection{Holomorphic  distributions as subsheaves of the  Tangent Bundle}
If $\mathscr{F}$ is a distribution of dimension two,  
then    $\deg(\mathscr{F})=2-c_1(T {\mathscr{F}})\geq 0$ is the degree of $\mathscr{F}$ and $N {\mathscr{F}}$ is a torsion free  sheaf of rank $2$. Thus, $\mathscr{F}$ is given by the short sequence: 
\begin{equation}\label{tandist}
    \mathscr{F}: 0\rightarrow T {\mathscr{F}}\xrightarrow{\varphi}T\mathbb{P}^4\xrightarrow{\pi}N {\mathscr{F}}\rightarrow 0. 
\end{equation}
 We will suppose that $ T {\mathscr{F}}$  is locally free.

\subsection{Numerical invariants} We will determine relations between the Chern classes of the tangent sheaf and the numerical invariants of the singular scheme.

\begin{theorem}\label{TangChern}
Let $\mathscr{F}$ be a dimension two distribution on $\mathbb{P}^4$ given as in the exact sequence (\ref{tandist}), with locally free tangent sheaf $T {\mathscr{F}}$ and 
whose the singular scheme $Z$ has dimension one.  Then:
\begin{enumerate}
\item The Chern classes of the tangent sheaf are 

\begin{equation*}
\begin{split}
c_1(T {\mathscr{F}}) & = 2-d, \\
c_2(T {\mathscr{F}}) & = \frac{d^3-d^2+2d+2-\deg(Z)}{2d+1},
\end{split}
\end{equation*} 

\item The arithmetic genus $p_a(Z)$ of the singular scheme is given by 
$$p_a(Z)=\frac{-2d^6-6d^5-10d^4-10d^3+8d^2+12d+4+\deg(Z)(12d^3-2d^2-11d-6)+2(\deg(Z))^2}{8d^2+8d+2}.$$

\end{enumerate}
\end{theorem}

\begin{proof}
Let $\varphi:T {\mathscr{F}}\rightarrow T\mathbb{P}^4$ be the map that induces the distribution (\ref{tandist}).
From now on, we denote $\mbox{Sym}^2(T {\mathscr{F}})=S_2(T {\mathscr{F}})$.  
Considering the Eagon-Northcott resolution (\ref{ENR}), associated with the dual map $\varphi^\vee:\Omega_{\mathbb{P}^4}^1\rightarrow T {\mathscr{F}}^\vee$, and twisting by $\mathcal{O}_{\mathbb{P}^4}(d-2)$ we get:
    
\begin{equation}\label{ENtangdist}
    0\rightarrow S_2(T {\mathscr{F}})(-5)\rightarrow\Omega_{\mathbb{P}^4}^{3}\otimes T {\mathscr{F}}\rightarrow\Omega_{\mathbb{P}^4}^{2}\rightarrow\mathscr{I}_{Z}(d-2)\rightarrow 0.
\end{equation}

Let $c_1:=c_1(T {\mathscr{F}})$ and $c_2:=c_2(T {\mathscr{F}})$. Using Chern's class formulas, we have:
\begin{equation}\label{chernIdeal1}
\begin{split}
c(\mathscr{I}_{Z}(d-2)) & = c(S_2(T {\mathscr{F}})(-5))\cdot c(\Omega_{\mathbb{P}^4}^{2})\cdot c(\Omega_{\mathbb{P}^4}^{3}\otimes T {\mathscr{F}})^{-1} \\
 & = 1-c_1\textbf{h}+(2c_1^3  - 10c_1^2  - 4c_1c_2 + 20c_1 + 10c_2 - 20)\textbf{h}^3 +\\
 & +(- 5c_1^4  + 40c_1^3  + 16c_1^2c_2 - 125c_1^2  -85c_1c_2 - 6c_2^2  + 200c_1 + 135c_2 - 180)\textbf{h}^4,
\end{split}
\end{equation}
 On the other hand, 
\begin{equation*}
    \begin{split}
        c_1(\mathscr{I}_Z(d-2)) & = c_1(\mathscr{I}_Z)+(d-2)=d-2, \\
        c_2(\mathscr{I}_Z(d-2)) & = c_2(\mathscr{I}_Z)=0, \\
        c_3(\mathscr{I}_Z(d-2)) & = c_3(\mathscr{I}_Z)=-2\deg(Z), \\
        c_4(\mathscr{I}_Z(d-2)) & = c_4(\mathscr{I}_Z)-2(d-2)c_3(\mathscr{I}_Z)=c_4(\mathscr{I}_Z)+4(d-2)\deg(Z).
    \end{split}
\end{equation*}
Hence, since $c_1(\mathscr{I}_Z(d-2))=d-2$ then $c_1=2-d.$ Analogously, since $$c_3(\mathscr{I}_Z(d-2))=-2\deg(Z),$$ substituting the first Chern class in the equation (\ref{chernIdeal1}), implies:
$$c_2=\frac{d^3-d^2+2d+2-\deg(Z)}{2d+1}.$$
Finally, by Riemann-Roch Theorem:
\begin{equation*}
\begin{split}
\chi(\mathscr{I}_Z) & = 1+\frac{5}{4}c_3(\mathscr{I}_Z)-\frac{1}{6}c_4(\mathscr{I}_Z),
\end{split}
\end{equation*}
and since $c_3(\mathscr{I}_Z)=-2\deg(Z)$,  so  $c_4(\mathscr{I}_Z)=6-15\deg(Z)-6p_a(Z)$,
hence:
$$c_4(\mathscr{I}_Z(d-2))=6-15\deg(Z)-6p_a(Z)+4(d-2)\deg(Z).$$
Writing $c_4(\mathscr{I}_Z(d-2))$ in terms of the first and second Chern class, we have:
$$c_1^4  - 75c_1^3  + 12c_1^2c_2 + 1270c_1^2  - 225c_1c_2 + 6c_2^2  - 7450c_1 + 845c_2 + 14085$$ $$=6-15\deg(Z)-6p_a(Z)+4(d-2)\deg(Z),$$
and replacing $c_1$ and $c_2$ we have:
$$p_a(Z)=\frac{-2d^6-6d^5-10d^4-10d^3+8d^2+12d+4+\deg(Z)(12d^3-2d^2-11d-6)+2(\deg(Z))^2}{8d^2+8d+2}.$$
\end{proof}

\begin{obs} \label{sub-zero} \rm 
 Let $\mathscr{F}$ be a dimension two distribution on $\mathbb{P}^4$, such that $T {\mathscr{F}}$ is locally free.  If $\sigma \in h^0(T {\mathscr{F}}(-1))$ is  a non-trivial holomorphic  section, then  we have an induced foliation by curves of degree zero
 
\centerline{
\xymatrix{
\mathcal{O}_{\mathbb{P}^4}(1)  \ar[rr]^{\sigma} \ar@/^2pc/[rrrr]^{v} &&  T {\mathscr{F}}  \ar[rr]^{ } && T\mathbb{P}^4 .  \\
}
} 
\noindent This implies that $T {\mathscr{F}}$ is split. In fact,  since the singular scheme of  $v: \mathcal{O}_{\mathbb{P}^4}(1)\to  T\mathbb{P}^4$ is a single point the section  $\sigma : \mathcal{O}_{\mathbb{P}^4}(1)\to  T {\mathscr{F}} $ cannot has zeros, otherwise the codimension two scheme $\{\sigma=0\}$ would be contained in $ \{v=0\}$, a contradiction. Therefore, $\sigma : \mathcal{O}_{\mathbb{P}^4}(1) \to T \mathscr{F}$ is an injective map of vector bundles, and we have an exact sequence of vector bundles:
$$
0\to \mathcal{O}_{\mathbb{P}^4}(1)\to T \mathscr{F} \to  \mathcal{O}_{\mathbb{P}^4}(a) \to 0
$$
where $a=c_1(T \mathscr{F})-1$. But,
$$
H^1\left(\mathcal{O}_{\mathbb{P}^4}, \operatorname{Hom}\left(\frac{T\mathscr{F}}{\mathcal{O}_{\mathbb{P}^4}(1)}, \mathcal{O}_{\mathbb{P}^4}(1)\right)\right) \simeq H^1\left(\mathcal{O}_{\mathbb{P}^4}, \operatorname{Hom}(\mathcal{O}_{\mathbb{P}^4}(a), \mathcal{O}_{\mathbb{P}^4}(1))\right) \simeq H^1\left(\mathcal{O}_{\mathbb{P}^4}, \mathcal{O}_{\mathbb{P}^4}(1-a)\right) = 0
$$
and  by \cite[Proposition 2]{At}, the above sequence splits, so  
 $T\mathscr{F}\simeq \mathcal{O}_{\mathbb{P}^4}(1)\oplus \mathcal{O}_{\mathbb{P}^4}(c_1(T \mathscr{F})-1) $. 
\end{obs}

\subsection{Holomorphic  distributions as subsheaves of the  cotangent bundle}
We recall that $\deg(\mathscr{F})=-3-c_1(N {\mathscr{F}}^*)\geq 0$. The distribution  $\mathscr{F}$ is given by the short sequence: 
\begin{equation}\label{Conormdist}
    \mathscr{F}: 0\rightarrow N {\mathscr{F}}^*\xrightarrow{\varphi}\Omega_{\mathbb{P}^4}^{1}\rightarrow \mathcal{Q}_{\mathscr{F}}\rightarrow 0.
\end{equation}
We will suppose that $N {\mathscr{F}}^*$ is  locally free.

By using the same arguments as we have used in the proof of Theorem \ref{TangChern} we get the following:

\begin{theorem}\label{conormChern}
Let $\mathscr{F}$ be a dimension two distribution on $\mathbb{P}^4$ given as in the exact sequence (\ref{Conormdist}), with locally free conormal sheaf $N {\mathscr{F}}^*$  
and 
whose   singular scheme $Z$ has dimension one.  Then:
\begin{enumerate}
\item The Chern classes of the tangent sheaf are 
\begin{equation*}
    \begin{split}
    c_1(N {\mathscr{F}}^*) & = -3-d, \\
    c_2(N {\mathscr{F}}^*) & = \frac{d^3+4d^2+7d+2-\deg(Z)}{2d+1},
    \end{split}
\end{equation*}
\item The arithmetic genus $p_a(Z)$ of the singular scheme is given by
$$p_a(Z)=\frac{-2d^6-6d^5-10d^4-10d^3+8d^2+12d+4+\deg(Z)(12d^3-2d^2-11d-1)+2(\deg(Z))^2}{8d^2+8d+2}.$$
\end{enumerate}
\end{theorem}

\section{Distributions with locally free tangent sheaf of degree at most 2}
Degree zero distributions of dimension two are classified by the Proposition \ref{Distri-zero}. Therefore, in this section, we will consider holomorphic distributions of degree $1$ and $2$ on $\mathbb{P}^4$ with locally free tangent sheaf and  
whose  singular scheme  has dimension one.

\subsection{Classification of degree $1$ distributions}
In this case Theorem \ref{TangChern} becomes: 
\\
\begin{center}
$c_1(T {\mathscr{F}})=1 $, $c_2(T {\mathscr{F}})=\dfrac{4-\deg(Z)}{3}$, and $p_a(Z)=\dfrac{2(\deg(Z))^2-7\deg(Z)-4}{18}.$
\end{center}

\begin{propo}\label{lfree1}
If $\mathscr{F}$ is a dimension two distribution on $\mathbb{P}^4$ of degree $1$ with locally free tangent sheaf 
and  
 whose singular scheme has pure dimension one.  
Then  $T {\mathscr{F}}=\mathcal{O}_{\mathbb{P}^4}(1)\oplus\mathcal{O}_{\mathbb{P}^4}$ and  its singular scheme  a rational normal curve of degree $4$. 
\end{propo}

\begin{proof}
In this case,  the normalization of $T {\mathscr{F}}$ is also  $(T {\mathscr{F}})_{\eta}=T {\mathscr{F}}(-1)$. 
If $h^0((T {\mathscr{F}})_{\eta})\neq 0$, then by    Remark \ref{sub-zero}  we conclude that  $T {\mathscr{F}}$ is split and  we have that 
$T {\mathscr{F}}=\mathcal{O}_{\mathbb{P}^4}(1)\oplus\mathcal{O}_{\mathbb{P}^4}$, thus $c_2(T {\mathscr{F}})=0$. Furthermore, by Theorem \ref{TangChern} its singular scheme is a curve of degree $\deg(Z)=4$ and arithmetic genus $p_a(Z)=0$. That is, $Z$ is a rational normal curve of degree $4$. 
If $h^0((T {\mathscr{F}})_{\eta})=0$, then $T {\mathscr{F}}$ is a locally free stable sheaf with 
\begin{center}
    $c_1(T {\mathscr{F}}(-1))=-1$ and $c_2(T {\mathscr{F}}(-1))=\dfrac{4-\deg(Z)}{3}$.
\end{center}
 By Bogomolov's inequality, we have $\deg(Z)<\frac{13}{4}$, hence $\deg(Z)=0, 1, 2, 3$ and this implies that either  $c_2(T {\mathscr{F}}(-1))$ or $p_a(Z)$ are  not   integer numbers, a contradiction. 
This shows that the tangent sheaf $T {\mathscr{F}}$ can not be stable.
\end{proof}
\subsection{Classification of degree $2$ distributions}

In this case,  Theorem \ref{TangChern} becomes: 
\\
\begin{center}
$
c_1(T {\mathscr{F}})=0, \   c_2(T {\mathscr{F}})= \dfrac{10-\deg(Z)}{5}$ and $   p_a(Z)=\dfrac{(\deg(Z))^2+30\deg(Z)-250}{25}.$  
\end{center}

\begin{propo}\label{lfree2}
Let $\mathscr{F}$ be a dimension two distribution on $\mathbb{P}^4$ of degree $2$ with a locally free tangent sheaf $T {\mathscr{F}}$ whose singular scheme has pure dimension one. Then $T {\mathscr{F}}$ splits as a sum of line bundles, and

    \begin{enumerate}
        \item either $T {\mathscr{F}}=\mathcal{O}_{\mathbb{P}^4}(1)\oplus\mathcal{O}_{\mathbb{P}^4}(-1)$, and its singular scheme is a curve of degree $15$ and arithmetic genus $17$;
        \item or $T {\mathscr{F}}=\mathcal{O}_{\mathbb{P}^4}\oplus\mathcal{O}_{\mathbb{P}^4}$, and its singular scheme is a curve of degree $10$ and arithmetic genus $6$;
    \end{enumerate}

\end{propo}

\begin{proof}
The normalization of $T {\mathscr{F}}$ is $(T {\mathscr{F}})_{\eta}=T {\mathscr{F}}$.  If $h^0(T {\mathscr{F}}(-1))\neq 0$, then by    Remark \ref{sub-zero}  we conclude that  $T {\mathscr{F}}$ is split and we have that 
$T {\mathscr{F}}=\mathcal{O}_{\mathbb{P}^4}(1)\oplus\mathcal{O}_{\mathbb{P}^4}(-1)$. Since  $c_2(T {\mathscr{F}})=-1$,  by Theorem \ref{TangChern} its singular scheme is a curve of $\deg(Z)=15$ and $p_a(Z)=17$.
Now, if $h^0(T {\mathscr{F}}(-1))= 0$ then  $T {\mathscr{F}}$
is a  semistable vector bundle such that  
\begin{center}
$c_1(T {\mathscr{F}})=0$ and $c_2(T {\mathscr{F}})=\dfrac{10-\deg(Z)}{5}$.
\end{center}
By Bogomolov's inequality, $\deg(Z)\leq 10$.
If $\deg(Z)=0$, then $Z=\emptyset$ and $c_2(T {\mathscr{F}})=2$. This case does not occur by Theorem  \ref{non-sing}.
Since $c_2(T {\mathscr{F}})$ is an integer number, then $\deg(Z)=5,10$. If $\deg(Z)=10$, then  $c_2(T {\mathscr{F}})=0$,  thus by  \cite[Lemma 2.1]{GJ} we have $T {\mathscr{F}}\simeq \mathcal{O}_{\mathbb{P}^4}\oplus\mathcal{O}_{\mathbb{P}^4}$ and by Theorem \ref{TangChern} we have $\deg(Z)=10$ and $p_a(Z)=6$.
If $\deg(Z)=5$, then $c_2(T {\mathscr{F}})=1$. 
 But such a vector bundle with  $c_1(T {\mathscr{F}})=0$ and  $c_2(T {\mathscr{F}})=1$,   can not exist by 
Schwarzenberger's condition
$$
(S_4^2): c_2(c_2+1-3c_1-2c_1^2) \equiv 0(12),
$$
see \cite[Section 6.1]{Oko}. 
\end{proof}

\begin{example}(Linear pull-back of foliations by curves)
Let $\mathscr{G}$ be a foliation by curves, of degree $d$,  on $\mathbb{P}^3$   whose singular scheme is  $\mathrm{Sing}(\mathscr{G}):= R $, where    $R$ has dimension 0.  Consider the rational linear map
$\rho: \mathbb{P}^4  \dashrightarrow \mathbb{P}^3$ given by $$\rho(z_0:z_1:z_2:z_3,z_4)= (z_0:z_1:z_2:z_3).$$ Then the dimension two foliation $\mathscr{F}=\rho^*\mathscr{G}$  has degree $d$,  has    tangent sheaf  $T_\mathscr{F}=\mathcal{O}_{\mathbb{P}^4}(1)\oplus \mathcal{O}_{\mathbb{P}^4}(1-d)$ and it is singular along  the $\deg(R)$ lines     (counted with multiplicities)  passing through   $(0:0:0:0:1)$.
\end{example}

\section{Singular holomorphic Engel distributions  } \label{Engel}

 Let $\mathscr{F}$ be a non-integrable  holomorphic distribution  of  dimension $2$ on a complex manifold of dimension 4.
 Define $\mathscr{F}^{[1]}:=( T\mathscr{F}+ [T\mathscr{F},T\mathscr{F}])^{**}\subset TX$.  If $\mathscr{F}^{[1]}$ is not integrable and has dimension 3,
 we say that   $\mathscr{F}$ is  \textit{holomorphic singular   Engel distribution}. 
 Denote by $\pi: \mathscr{F}^{[1]} \to TX/\mathscr{F}^{[1]}$ the projection on the normal sheaf.  Consider the O'Neill tensor
 \begin{alignat*}{2}
 \mathcal{T}(\mathscr{F}^{[1]}) :\wedge^{[2]}\mathscr{F}^{[1]}  &\longrightarrow&  TX/\mathscr{F}^{[1]} \\
 u\wedge  v &\longmapsto&  \pi([u,v]).  
\end{alignat*}
Then, there is  a unique foliation by curves  $  \mbox{Ker}(\mathcal{T}(\mathscr{F}^{[1]})^{**}:=\mathcal{L}(\mathscr{F}) \subset \mathscr{F}^{[1]} $, called by \textit{characteristic foliation} of $\mathscr{F}^{[1]}$,  such that  $[\mathcal{L}(\mathscr{F}^{[1]}),\mathscr{F}^{[1]}]\subset \mathscr{F}^{[1]}$. We have the so-called \textit{derived flag}  associated with $\mathscr{F}$:
$$
  \mathcal{L}(\mathscr{F} )\subset  T\mathscr{F} \subset  \mathscr{F}^{[1]} \subset TX. 
$$
We say that $\mathscr{F}$ is a \textit{ locally free Engel distribution } if $T\mathscr{F}$ is so. See \cite{BCGGG,CMaz,PS, Vo} for more details on Engel distributions.

  \begin{example}\label{exe:null:deg1}
Consider the dimension $2$ distribution  $\mathscr{F}_0$  on $\mathbb{P}^4$ given by 
$$\omega=z_{0}^{2}dz_{4}\wedge
dz_{3}-z_{0}^{2}z_{2}dz_{4}\wedge
dz_{1}+(z_{1}z_{2}-z_{0}z_{3})dz_{4}\wedge dz_{0}- z_{0}z_{3}dz_{1}\wedge
dz_{3} 
+(z_{1}z_{3}-z_{0}z_{4})dz_{0}\wedge
dz_{3}-  
$$
$$
- (z_3^2-z_2z_4)dz_{0}\wedge dz_{1}.
$$
which is a distribution of degree $1$ and    $$\mathrm{Sing}( \mathscr{F})=\{z_0^2=z_2=z_3=0\}\cup \{z_0=z_1=z_3^2-z_2z_4=0\}   $$ has dimension 1.  The restriction of  $\omega $ to  $\mathbb{C}^4=\mathbb{P}^4-\{z_0=0\}$  is given by
$$   dz_{4}\wedge
dz_{3}- z_{2}dz_{4}\wedge
dz_{1} -  z_{3}dz_{1}\wedge
dz_{3} = (dz_4-z_3dz_1) \wedge (dz_3-z_2dz_1),
$$
which is the canonical Engel structure on $\mathbb{C}^4$.
In particular,  we have that  $$\mathscr{F}_0^{[1]}|_{\mathbb{C}^4}=\Ker ( dz_4-z_3dz_1)$$ 
 and $ \mathcal{L}(\mathscr{F})|_{\mathbb{C}^4}$ is generated by the vector field $\frac{\partial }{\partial z_2}$. In particular, the distribution $\mathscr{F}_0^{[1]}$ is the linear pull-back by the projection $(z_0:z_1:z_2:z_3:z_4)\to (z_0:z_1:z_3:z_4) $ of a contact structure on $\mathbb{P}^3$ with tangent bundle $N(1)$, where $N$ is the null-correlation bundle.
\end{example}

 \begin{defin}(Cartan prolongation)
 Let $\mathscr{F}$ be a singular contact structure on a threefold $X$ such that $T\mathscr{F}$ is locally free. Consider the 4-fold $\mathbb{P}(T\mathscr{F}^*)$ with the natural projection $\pi: Y:=\mathbb{P}(T\mathscr{F}^*) \to X $ and   the twisted  relative Euler sequence  
 $$
 0\to \mathcal{O}_X(-1)\to \pi^* T\mathscr{F}\to T_{Y|X}(-1)\to 0 
 $$
 The pull-back $\pi^*\mathscr{F}$ is a  codimension one non-integrable distribution, such that the  $T_{Y|X}\subset T(\pi^*\mathscr{F}) $.
 The kernel $\mathscr{D}$ of the composition $$T(\pi^*\mathscr{F})\to  \pi^* T\mathscr{F}\to T_{Y|X}(-1)$$ is an Engel structure called Cartan prolongation of $\mathscr{F}$. 
  \end{defin}
  
  \begin{obs}\rm
   Suppose that $T\mathscr{F}$ is locally free  and consider  an open set $U\subset X$ such that $\mathbb{P}(T\mathscr{F})_{|U}\simeq U\times \mathbb{P}^1$, is induced by  consider the 1-form $\beta$ which induces $\mathscr{F}$ and let $\omega_1,\omega_2$ be a local frame for $T\mathscr{F}^*_{|U}= \mathcal{O}_U^{\oplus 2}$ such that $\beta\wedge \omega_1\wedge \omega_2  \not\equiv 0$, then the prolongation   of $\mathscr{F}_{|U}$  is induced by the 2-form
 $$
 \beta\wedge (z_1\omega_1- z_0\omega_2).
 $$
  \end{obs}

  \begin{defin}\label{Lorentzian}(Lorentzian type structures)
  Let $\mathscr{F}$ be an Engel structure on a 4-fold $X$. Suppose that there is a  fibration  $f: X \dashrightarrow Y$  by  rational curves, over a projective manifold $Y$. 
  We say that $\mathscr{F}$ is of Lorentzian type if $T_{X/Y}\subset T \mathscr{F}$, i.e., the fibers of $f$ are tangent to $\mathscr{F}$, and the characteristic foliation $\mathcal{L}(\mathscr{F} )$ is induced by a section of $f^*TY\otimes   T_{X/Y}$. We say that the rational map  $f: X \dashrightarrow Y$ is   Lorentzian with  respect  to $\mathscr{F}$.
   \end{defin}
  
  \begin{obs}\rm \label{Lorentzian}
    Consider $X^{\circ }= X-\mbox{Sing}(\mathscr{F})$ and denote $\mathscr{F}^{\circ }=\mathscr{F}_{|X^{\circ }}$.   On an open set $U\simeq V\times \mathbb{P}^1\subset X^{\circ }$  such that $ T \mathscr{F}^{\circ }_{|U}\simeq  \mathcal{O}_U^{\oplus 2 }$, we have that $(T_{X/Y})_{|U}$ is induced   $v_2=\frac{\partial }{\partial z}$ and $\mathcal{L}(\mathscr{F} )_{|U}$ is  induced by a vector field  of type
  $v_1=\nu_2+z_0\nu_1+z_0^2\nu_0$,  where $\nu_0,\nu_1,\nu_2\in TV$.  Therefore, $\mathscr{F}^{[1]}_{|U}$ is generated by $v_1$, $v_2$, and $[v_1,v_2]=-2z_0\nu_1-\nu_2$. That is, $\mathscr{F}^{[1]}_{|U}$ is induced by the 1-form 
  $\omega =  z^2\omega_0 -2z\omega_1 + \omega_2 $, where $\omega_i(\nu_i)=0$, for all $i=0,1,2$. 
  \end{obs}

A particular example of dimension two distribution of Lorentzian type is the so-called Lorentzian tube on the tangent bundle of a threefold with a holomorphic conformal structure. In fact, consider a tensor $g\in H^0(Y, \mbox{Sym}^2\Omega_Y^1\otimes \mathscr{L} )$ inducing a holomorphic conformal structure on $Y$. 
We have a smooth divisor $X:=\{g=0\}\subset \mathbb{P}(\Omega_Y^1)$ and the restriction of the natural projection $\mathbb{P}(\Omega_Y^1)\to Y$ gives us a smooth conic bundle $f: X \to Y$.  
It follows from \cite[Proposition 4.5]{PS} that if  $\mathscr{F}$ is  a regular  Engel  structure  on a 4-fold $X$ such that
 $T\mathscr{F}\simeq   \mathcal{L}(\mathscr{F} )\oplus
 (T\mathscr{F}/\mathcal{L}(\mathscr{F} ))$ and    $ (T\mathscr{F}/\mathcal{L}(\mathscr{F}))^*$  is   not nef, then $\mathscr{F}$ is of Lorentzian type in our sense. We refer to \cite[Example 3.5 ]{PS} for more details.

Inspired by a Presas and Sol'a Conde result \cite[Proposition 4.5]{PS} we prove the following:

 \begin{propo}\label{Engelsplit}
 Let $\mathscr{F}$ be a locally free Engel structure on a 4-fold projective, whose singular scheme of the associate foliation $\mathcal{L}(\mathscr{F} )$ has codimension $\geq 3$. Suppose that $ \det(T\mathscr{F}/\mathcal{L}(\mathscr{F}))^*$
 is   not  pseudo-effective and     $ H^1(X,   \det(T\mathscr{F})^*\otimes \mathcal{L}(\mathscr{F})^{\otimes 2} )=0
$. Then $\mathscr{F}$ is of Lorentzian type.  
 \end{propo}
 \begin{proof}
Since the singular scheme of $\mathcal{L}(\mathscr{F} )$ has codimension $\geq 3$, then the section $\mathcal{L}(\mathscr{F}) \to T\mathscr{F}$ has no zeros. In particular,  the cokernel $T\mathscr{F}/\mathcal{L}(\mathscr{F} )$ is a line bundle. 
Since 
\begin{equation*}
\begin{split}
H^1(X,   \mbox{Hom}( (T\mathscr{F}/\mathcal{L}(\mathscr{F} )),\mathcal{L}(\mathscr{F} )) & \simeq  H^1(X,     \mathcal{L}(\mathscr{F})\otimes  (T\mathscr{F}/\mathcal{L}(\mathscr{F} ))^* ) \\
 & \simeq H^1(X,   \det(T\mathscr{F})^*\otimes \mathcal{L}(\mathscr{F})^{\otimes 2} )=0
\end{split}
\end{equation*}
then by \cite[Proposition 2]{At}  we have that 
 $T\mathscr{F}$  is isomorphic to $\mathcal{L}(\mathscr{F} )\oplus
( T\mathscr{F}/\mathcal{L}(\mathscr{F} )) \subset TX$.  On the other hand, $\det(T\mathscr{F}/\mathcal{L}(\mathscr{F}))^*$ is not pseudo-effective implies by Brunella's Theorem \cite{Brunella} that the induced foliation $ T\mathscr{F}/\mathcal{L}(\mathscr{F})\subset  T\mathscr{F}  \subset TX$ is a foliation by rational curves, i.e, such foliation is tangent to a fibration $f: X \dashrightarrow Y$
whose fibers are  rational curves, and since $\mathcal{L}(\mathscr{F})$ must be generically transversal to $f$ we conclude that  $\mathscr{F}$ is of Lorentzian type.   
 \end{proof}
  \begin{obs}\rm 
The hypothesis $  H^1(\mathbb{P}^4,   \det(T\mathscr{F})^*\otimes \mathcal{L}(\mathscr{F})^{\otimes 2} )=0$ on $\mathbb{P}^4$ of Proposition \ref{Engelsplit} is
satisfied.
  \end{obs}

\begin{theorem}
\label{Sing-Engel}
Let  $X$ be  a  projective 4-fold  such that $\mathrm{Pic}(X)\simeq \mathbb{Z}$. If  $T\mathscr{F}\subset TX$   is an Engel structure, then $ \mathscr{F}$ is singular. 
\end{theorem}
\begin{proof}
 We argue by contradiction and  assume that 
   $\mathscr{F}$ is regular. 
   This implies that the distributions  $\mathcal{L}(\mathscr{F})$ and  $\mathscr{F}^1$ are also regular.   In fact, for every point $p\in X$, take an open set $U$ such that $\mathscr{D}_{|U}$ is induced by vector fields $v,v_1$, where $v$ is tangent to $\mathcal{L}(\mathscr{F})_{|U}$.  Then $\mathrm{Sing}(\mathcal{L}(\mathscr{F})_{|U})=\{v =0\} \subset \{v\wedge v_1  =0\}= \mathrm{Sing}(\mathscr{F}_{|U})=\emptyset$.
Arguing
as above, the conormal bundle of $\mathscr{F}$ on $U$ is  generated by $\omega,\omega_1$, where $\omega$ induces $\mathscr{D}^1_{|U}$.
So,  $\mathrm{Sing}( \mathscr{F}^1)_{|U})=\{\omega =0\} \subset \{\omega\wedge \omega_1  =0\}= \mathrm{Sing}(\mathscr{F}_{|U})=\emptyset$.  Proposition \ref{Conde--Presas}  says us that  either   $    \mathcal{L}(\mathscr{F})^*$ or 
 $\det(T\mathscr{F}/\mathcal{L}(\mathscr{F}))$ are  not pseudo-effective.

 Suppose that $    \mathcal{L}(\mathscr{F})^*$ is not pseudoeffective, then $\mathcal{L}(\mathscr{F})\subset TX$ is ample since $\mathrm{Pic}(X)\simeq \mathbb{Z}$, thus by \cite{Wahl} we have that
$X$ isomorphic to $\mathbb{P}^4$,  and  Theorem \ref{non-sing}  implies that  $ \mathscr{F}$ is singular, resulting in a contradiction.  

Now, suppose that $\det(T\mathscr{F}/\mathcal{L}(\mathscr{F}))^*$  is 
  not pseudo-effective. 
    We have a diagram of exact sequences of vector bundles:
$$ \xymatrix{ 0 \ar[r]  & \dfrac{T\mathscr{F}}{\mathcal{L}(\mathscr{F})}\ar[r]  & \dfrac{TX}{\mathcal{L}(\mathscr{F})} \ar[r] & \dfrac{TX}{\mathcal{L}(\mathscr{F})} \ar[r]   & 0    }.$$
By taking the determinant we get
$\det(\mathscr{F}/\mathcal{L}(\mathscr{F}))\otimes \det( TX/ \mathscr{F}) \simeq  \det(TX/\mathcal{L}(\mathscr{F}))$. 
 Since $\mathcal{L}(\mathscr{F})$ is regular, by Baum-Bott Theorem \cite{BB, MF} we  obtain 
$$
0=c_1^4(TX/\mathcal{L}(\mathscr{F}))= c_1^4(\det( TX/\mathcal{L}(\mathscr{F})))
$$
Since $\mathrm{Pic}(X)\simeq \mathbb{Z}$, we have $  \det( TX/\mathcal{L}(\mathscr{F}))\simeq \mathcal{O}_X$.
 Therefore, 
$ \det( TX/ T\mathscr{F})^* \simeq  \det(T\mathscr{F}/\mathcal{L}(\mathscr{F}))$ is not  pseudo-effective by Demailly's Theorem \cite{Demailly}, since $\mathscr{F}$ is not integrable.  
That is  $\det(T\mathscr{F}/\mathcal{L}(\mathscr{F}))^*$ is  effective, since $\mathrm{Pic}(X)\simeq \mathbb{Z}$.  This is a  contradiction, since by   hypothesis
$\det(T\mathscr{F}/\mathcal{L}(\mathscr{F}))^*$  is not   effective. 
\end{proof}

 \begin{propo}\label{Null:contact}
Let $\mathscr{F}$ be a contact non-singular distribution on $\mathbb{P}^3$. Then, the prolongation of $\mathscr{F}$  is birational to $\mathscr{F}_0$. In particular,   $\mathbb{P}(T\mathscr{F}_0^*)$ 
is birational to the universal family
of lines contained in a quadric hypersurface in $\mathbb{P}^4$. 
\end{propo}

\begin{proof}
The contact non-singular distribution $\mathscr{F}$ has  tangent bundle  isomorphic to  $N(1)$, where is the null correlation bundle.  After a linear change of coordinates, we can assume 
that $\mathscr{F}$ and $\mathscr{F}_0^{[1]}$ are  induced by 
$\omega_1=z_0dz_1-z_1dz_0+ z_2dz_3-z_2dz_3$ and 
$\omega_0=z_3^2dz_2-z_1z_3dz_0+ (z_0z_1-z_2z_3)dz_3$, respectively. Consider the birational map  $f: \mathbb{P}^3  \dashrightarrow \mathbb{P}^3 $, given by $f(z_0,z_1,z_2,z_3 )=(z_0z_3/2,z_1z_2,-z_2z_3+z_0z_1/2,z_3^2)$. This map is the extension to $\mathbb{P}^3$ of the map in  \cite[pg. 5]{CD}.  Then $f^*\omega_1=z_3\omega_0$, showing that the prolongation of $\mathscr{F}_0^1$ is birational to the prolongation of $f^*\mathscr{F}$.  
Finally,  we can see from  \cite[pg 12]{PS} that   
$\mathbb{P}(T\mathscr{F}^*) =\mathbb{P}(N(1))$ is 
the universal family
of lines contained in a quadric hypersurface in $\mathbb{P}^4$.
\end{proof}

 Let $\beta \in H^0(X,\Omega_X^1\otimes \det(TX/\mathscr{F}^{[1]}))$ the twisted 1-form inducing $\mathscr{F}^{[1]}$. If the twisted 3-form 
 $\beta\wedge d\beta  \in H^0(X,\Omega_X^3\otimes \det(TX/\mathscr{F}^{[1]})^{\otimes 2})$  has no divisorial zeros, then it defines the characteristic foliation.  
 In this case, the determinant of the normal sheaf of $\mathcal{L}(\mathscr{F})$ is $\det(TX/\mathscr{F}^{[1]})^{\otimes 2})$. 
 
 \begin{propo}\label{Conde--Presas}
Let  $X$ be  a   K\"ahler  manifold   and    $\mathscr{F}\subset TX$ an Engel structure such that the determinant  of the normal sheaf of $\mathcal{L}(\mathscr{F})$ is isomorphic to  $\det(TX/\mathscr{F}^{[1]})^{\otimes 2})$. Then: \begin{itemize}
      \item[$a)$] $\mathcal{O}_X(-K_X)\simeq \det(TX/\mathscr{F}^{[1]})^{\otimes 2} \otimes \mathcal{L}(\mathscr{F})^* $;
       \item[$b)$]    either $\mathcal{L}(\mathscr{F})^*$ or $\det(T\mathscr{F}/\mathcal{L}(\mathscr{F}))^{*} $ are not pseudo-effective.  
\end{itemize}
\end{propo}
 \begin{proof}
Set $S:=  \mathrm{Sing}(\mathscr{F}) \cup \mathrm{Sing}(\mathscr{F}^{[1]}) \cup  \mathrm{Sing}( \mathcal{L}(\mathscr{F}))$ and consider $U:=X-S$. It follows from \cite[Remark 2.2]{PS} that $\mathcal{O}_X(-K_X)_{|U}\simeq (\det(TX/\mathscr{F}^{[1]})^{\otimes 2} \otimes \mathcal{L}(\mathscr{F})^*) _{|U}$ and $\det(\mathscr{F}^{[1]}/\mathcal{L}(\mathscr{F}))_{|U} \simeq  \det(TX/\mathscr{F}^{[1]})_{|U}$. 
By \cite[pg. 4]{PS}, we have an exact sequence  of vector bundles
$$
0\to T\mathscr{F}/ \mathcal{L}(\mathscr{F})_{|U} \to \mathscr{F}^{[1]}/\mathcal{L}(\mathscr{F})_{|U} \to   \mathcal{L}(\mathscr{F}) \otimes T\mathscr{F}/ \mathcal{L}(\mathscr{F})_{|U} \to 0
$$
and  by taking determinant we get 
$$
 \mathcal{L}(\mathscr{F})\otimes \det(T\mathscr{F}/ \mathcal{L}(\mathscr{F}))^{\otimes 2}_{|U} \simeq
 \det(\mathscr{F}^{[1]}/\mathcal{L}(\mathscr{F}))_{|U}\simeq
 \det(TX/\mathscr{F}^1)_{|U}. 
$$
Since $S$ has  codimension $\geq 2$, we conclude that
$\mathcal{O}_X(-K_X)\simeq \det(TX/\mathscr{F}^{[1]})^{\otimes 2} \otimes \mathcal{L}(\mathscr{F})^* $ and $
\det(TX/\mathscr{F}^{[1]}) \simeq \mathcal{L}(\mathscr{F})\otimes \det(T\mathscr{F}/ \mathcal{L}(\mathscr{F}))^{\otimes 2}.  
$ In particular, proving $a)$. Item  $b)$
follows from the same  reason used in  \cite[Lemma 4.2]{PS}  by applying Demailly's Theorem \cite{Demailly} and using  that $\mathscr{F}^{[1]}$ is non-integrable. 
 \end{proof}

\begin{propo}\label{Engel-conormal}
Let  $\mathscr{F}$  be a singular  non-integrable distribution of  dimension $2$  on $\mathbb{P}^4$   with locally free locally   conormal sheaf.  Suppose that $\mbox{Sing}(\mathscr{F})$   has pure dimension 1.  Then $\mathscr{F}$ is an Engel distribution. 
\end{propo}
\begin{proof}
Since $\mathscr{F}$ is not integrable, we can consider the associate distribution $\mathscr{F}^{[1]}$.
We argue by contradiction and  assume that    $\mathscr{F}^{[1]}$  is  integrable.  If $p\in \mbox{Sing}(\mathscr{F})$, consider an  open set $U$, with $p\in U$, such that $N {\mathscr{F}}^*|_{U}\simeq \mathcal{O}_U^{\oplus 2}$, where 
 $\mathscr{F}$ is  induced by a  decomposable  2-form 
 $
 \omega_1\wedge \omega_2,
 $
 where  $\mbox{Ker}(\omega_1)=  \mathscr{F}^{[1]}|_{U}$. But $\mbox{Sing}(\mathscr{F}^{[1]})|_{U}=\{\omega_1=0\}\subset \{ \omega_1\wedge \omega_2=0\}=\mbox{Sing}(\mathscr{F})|_{U}$.  The integrability of $\mathscr{F}^{[1]}$ implies that $\mbox{Sing}(\mathscr{F}^{[1]})$ has at least a dimension 2 component. A contradiction.   
\end{proof}

Now, we will give an example of a non-integrable dimension two and degree $1$  distribution on $\mathbb{P}^4$ which is not an Engel structure and the singular scheme has no pure dimension 1. 
\begin{example}
Consider the  dimension two non-integrable distribution  $\mathscr{F}$ of degree $1$ given by
$$
\omega=z_0^2dz_2\wedge dz_1+z_0z_3dz_2\wedge dz_4 -(z_0z_1+z_3z_4)dz_2\wedge dz_0-z_0z_2dz_0\wedge dz_1 - z_2z_3dz_0\wedge dz_4.
$$
The singular scheme of $\mathscr{F}$ is $\{z_0^2=z_3=0\}\cup \{z_0=z_2=z_4=0\}$. In $\mathbb{C}^4=\mathbb{P}^4-\{z_0=0\}$ we have that $\omega $ is given by
$$ dz_{2}\wedge
dz_{1}+z_{3}dz_{2}\wedge
dz_{4}  = dz_2 \wedge   (dz_1+z_3dz_4). 
$$
This shows us that the distribution $\mathscr{F}^{[1]} \subset T \mathbb{P}^4$ is a foliation of degree 0 induced by 
$z_2dz_0-z_0dz_2$, and in particular $\mathscr{F}$ is not an Engel structure.  
\end{example}

Let us give an example with the Engel structure of  Lorentzian type whose characteristic foliation  $\mathcal{L}(\mathscr{F})$ has degree $2$ and has no algebraic leaves.  In particular, $\mathscr{F}$ is  of  Lorentzian type and it is not  birational to a Cartan prolongation. 
\begin{example}
Consider the distribution $\mathscr{F}$ of dimension $2$ on  $\mathbb{P}^4$ of degree $2$ induced in affine coordinates  by 
$$
\sigma=(z_2^2-z_1^2)\frac{\partial }{\partial z_1}\wedge \frac{\partial }{\partial z_2}- (z_4^2-z_3z_1^2)\frac{\partial }{\partial z_2}\wedge \frac{\partial }{\partial z_3}-
 (1-z_4z_1^2)\frac{\partial }{\partial z_2}\wedge \frac{\partial }{\partial z_4}.
$$
We have that  $ \mathscr{F}^{[1]}$ is  a not integrable distribution of degree $3$ and is induced by the vector fields
$$
v_1=(z_2^2-z_1^2)\frac{\partial }{\partial z_1}+ (z_3^2-z_2z_1^2)\frac{\partial }{\partial z_2}+ (z_4^2-z_3z_1^2)\frac{\partial }{\partial z_3}+
 (1-x_4z_1^2)\frac{\partial }{\partial z_4},
$$
  $v_2=\frac{\partial }{\partial z_2}$,  and $v_3:=[v_2,v_1]= 2z_2 \frac{\partial }{\partial z_1}-z_1^2 \frac{\partial }{\partial z_2}$. That is,  $ \mathscr{F}$ is of Lorentzian type,  the linear projection  $f:\mathbb{P}^4 \dashrightarrow \mathbb{P}^3$ given by $f(z_0,z_1,z_2,z_3,z_4)=(z_0,z_1,z_3,z_4)$  is Lorentzian with respect to  $ \mathscr{F}$ 
 and the charactetistic foliation $\mathcal{L}(\mathscr{F})$ is the  Jouanolou's foliation induced by $v_1$ which has no algebraic leaves.  
\end{example}

\section{Distributions of degrees 1 and 2 with locally free conormal sheaf }

In this section, we will consider dimension two holomorphic distributions of degree $1$ and $2$ on $\mathbb{P}^4$ with a locally free conormal sheaf whose singular scheme has pure dimension one.

\subsection{Classification of degree $1$ distributions}

In this case Theorem \ref{conormChern} becomes:
\begin{center}
    $c_1(N {\mathscr{F}}^*) = -4$, $c_2(N {\mathscr{F}}^*)   = \dfrac{14-\deg(Z)}{3}$ and $p_a(Z)=\dfrac{(\deg(Z))^2-\deg(Z)-2}{9}.$
\end{center}

\begin{propo}\label{conormal-Engel}
If $\mathscr{F}$ is a dimension two distribution on $\mathbb{P}^4$ of degree $1$ with a locally free conormal sheaf
whose singular scheme has pure dimension one.
 Then  $N {\mathscr{F}}^*=\mathcal{O}_{\mathbb{P}^4}(-2)\oplus\mathcal{O}_{\mathbb{P}^4}(-2)$ and its singular scheme is a conic curve.
\end{propo}

\begin{proof}
The normalization of $N {\mathscr{F}}^*$ is $(N {\mathscr{F}}^*)_{\eta}=N {\mathscr{F}}^*(2)$. If $\sigma\in H^0((N {\mathscr{F}}^*)_{\eta})$ is a non-trivial section, then
\begin{center}
$0\leq\deg(\sigma=0)=c_2(N {\mathscr{F}}^*(2))=\dfrac{2-\deg(Z)}{3}$
\end{center}
which implies that $\deg(Z)\leq 2$. Therefore, $c_2(N {\mathscr{F}}^*(2))$ is an integer number if and only if $\deg(Z)=2$, i.e.,  $c_1(N {\mathscr{F}}^*(2))=c_2(N {\mathscr{F}}^*(2))=0$, and  \cite[Lemma 2.1]{GJ}  implies  that $N {\mathscr{F}}^*(2)\simeq \mathcal{O}_{\mathbb{P}^4}\oplus\mathcal{O}_{\mathbb{P}^4}$. Furthemore, by Theorem \ref{conormChern} we have $\deg(Z)=2$ and $p_a(Z)=0$.
If $h^0((N {\mathscr{F}}^*)_{\eta})=0$, then $N {\mathscr{F}}^*$ is a locally free stable sheaf with
\begin{center}
   $c_1(N {\mathscr{F}}^*(2))=0$ and $c_2(N {\mathscr{F}}^*(2))=\dfrac{2-\deg(Z)}{3}$. 
\end{center}
 By Bogomolov's inequality, we have $\deg(Z)< 2$ and this implies that $c_2(N {\mathscr{F}}^*)$  is   not   an integer number, a contradiction. So  the conormal sheaf $N {\mathscr{F}}^*$ is not stable, but only split.

\end{proof}

\begin{example}
Let $\omega_1,\omega_2 \in  H^0(\Omega_{\mathbb{P}^4}^{1}(2))$ be given by   
$\omega_1=(z_2-z_4)dz_0+z_3dz_1-z_0dz_2-z_1dz_3+z_0dz_4$ and $\omega_2=z_3dz_0+(2z_2-z_4)dz_1-2z_1dz_2-z_0dz_3+z_1dz_4.$ Consider the non-integrable distribution $\mathscr{F}$ induced by the 2-form 
\begin{equation*}
    \begin{split}
        \omega= \omega_1\wedge\omega_2 
        & = (2z_2^2-z_3^2-3z_2z_4+z_4^2)dz_0\wedge dz_1+(-2z_1z_2+z_0z_3+2z_1z_4)dz_0\wedge dz_2 +\\ 
        & +(-z_0z_2+z_1z_3+z_0z_4)dz_0\wedge dz_3+(z_1z_2-z_0z_3-z_1z_4)dz_0\wedge dz_4+ \\
        & + (2z_0z_2-2z_1z_3-z_0z_4)dz_1\wedge dz_2+(2z_1z_2-z_0z_3-z_1z_4)dz_1\wedge  dz_3 +\\
        & +(-2z_0z_2+z_1z_3+z_0z_4)dz_1\wedge dz_4 +(z_0^2-2z_1^2)dz_2\wedge dz_3+z_0z_1dz_2\wedge dz_4+
        \\
        & +(z_0^2-z_1^2)dz_3\wedge dz_4.
    \end{split}
\end{equation*}
The singular scheme of $\mathscr{F}$ is given by the conic  $$\Sing(\mathscr{F})=\{z_0=z_1=2z_2^2-z_3^2-3z_2z_4+z_4^2=0\}.$$
We have that the following holds
$$
 \begin{cases}     \omega_1\wedge\omega_2 \wedge d\omega_1=0,
\\
 \omega_1\wedge\omega_2 \wedge d\omega_2\neq 0,
 \\
\omega_2 \wedge d\omega_2 \neq 0.
\end{cases}
$$
This shows us that $\mathscr{F}$ is an Engel distribution with $\mathscr{F}^{[1]}=\mbox{Ker}(\omega_2)$. Then, by  \cite[Proposition 4.3]{ACMa}  there are a linear projection  $\phi:\mathbb{P}^4 \dashrightarrow  \mathbb{P}^3$  and a non-integrable distribution $\mathscr{G}$ of degree $1$ on $\mathbb{P}^3$  such that $\mathscr{F}^{[1]}=\rho^*\mathscr{G}$. Therefore,
$\mathscr{F}$ is the blow-down of the Cartan prolongation of $\mathscr{G}$. 

Now, consider the foliation $\mathscr{F}_\phi$ induced by   the rational map $\phi:\mathbb{P}^4 \dashrightarrow  \mathbb{P}(1,1,2) $, defined by $\phi(z_0,z_1,z_2,z_3,z_4)= (z_0,z_1,2z_2^2-z_3^2-3z_2z_4+z_4^2)$.
We have $\Sing(\mathscr{F})\subset \Sing(\mathscr{F}_\phi)$, which tells us that degree one distributions are not determined by their singular schemes. 
\end{example}

\begin{obs}\label{Okonek2} \rm 
It follows from  \cite[Theorem 1.5]{Okonek2} that 
$$
\mathcal{M}^{st}(-1,2,2,5)=\left\{\mbox{coker}(\phi);\phi: \mathcal{O}_{\mathbb{P}^4}(-2)^{\oplus 2}\to \Omega_{\mathbb{P}^4}^1\right \}.
$$
Now, consider the 10-dimensional complex vector space $W:=H^0(\Omega^1_{\mathbb{P}^4}(2))$. Since a generic  element of $$\left\{\mbox{coker}(\phi);\phi: \mathcal{O}_{\mathbb{P}^4}(-2)^{\oplus 2}\to \Omega_{\mathbb{P}^4}^1\right \} $$
is induced by a polynomial 2-form $\omega_1\wedge  \omega_2$, where $\omega_1,  \omega_2 \in W$ are generic 1-forms of degree 1, then the space  $\mathcal{M}^{st}(-1,2,2,5)$  is a Zariski open subset of the Grassmannian  $Gr(2,W)=\mathbb{G}(1,9)$, whose dimension is $16$.   
\end{obs}

\begin{example}(Intersection of pencil of hyperplanes) Consider the dimension two foliation $\mathscr{F}$ of degree $1$ given by
$$
\omega=(z_0dz_1-z_1dz_0) \wedge (z_2dz_3-z_3dz_2). 
$$
The leaves of $\mathscr{F}$ are planes which are intersections of hyperplanes of the pencils $\{\mu z_0+ \nu z_1=0\}_{[\mu : \nu]\in \mathbb{P}^1}$ and $\{\lambda z_2+ \rho z_3=0\}_{[\lambda : \rho]\in \mathbb{P}^1}$ and $$\mathrm{Sing}(\mathscr{F})=\{ z_0=z_1=0\}\cup \{ z_2=z_3=0\}$$ has codimension two. Moreover, it is clear that  $N {\mathscr{F}}^*=\mathcal{O}_{\mathbb{P}^4}(-2)\oplus\mathcal{O}_{\mathbb{P}^4}(-2)$ and  $T \mathscr{F}=\mathcal{O}_{\mathbb{P}^4}(1)\oplus \mathcal{O}_{\mathbb{P}^4}$,
 since  $i_{\frac{\partial }{\partial z_4}}\omega =0$. In fact,    $\mathscr{F}$ is the pull-back by the rational map $(z_0:z_1:z_2:z_3:z_4)\to (z_0:z_1:z_2:z_3) $ of the    foliation by curves, of degree 1,  induced by $\omega$.     
\end{example}

\subsection{Classification of degree $2$ distributions}

In this case Theorem \ref{conormChern} becomes:
\begin{center}
    $c_1(N {\mathscr{F}}^*) = -5$, $c_2(N {\mathscr{F}}^*)  = \dfrac{40-\deg(Z)}{5}$ and $p_a(Z)=\dfrac{2(\deg(Z))^2+65\deg(Z)-500}{50}.$
\end{center}

\begin{propo}
If $\mathscr{F}$ is a dimension two distribution on $\mathbb{P}^4$ of degree $2$ with a locally free conormal sheaf. 
 whose singular scheme  has pure dimension one. 
 Then $N {\mathscr{F}}^*=\mathcal{O}_{\mathbb{P}^4}(-2)\oplus\mathcal{O}_{\mathbb{P}^4}(-3)$ and its singular scheme is a curve of degree $10$ and arithmetic genus $7$.
\end{propo}

\begin{proof}
The normalization of $N {\mathscr{F}}^*$ is $(N {\mathscr{F}}^*)_{\eta}=N {\mathscr{F}}^*(2)$. If $\sigma\in H^0((N {\mathscr{F}}^*)_{\eta})\neq 0$, then this gives us a non-trivial section and we have a sequence

\centerline{
\xymatrix{
\mathcal{O}_{\mathbb{P}^4}(-2)  \ar[rr]^{\sigma} \ar@/^2pc/[rrrr]^{\omega} &&  N {\mathscr{F}}^*  \ar[rr]^{ } && \Omega_{\mathbb{P}^4}^1 .  \\
}
} 
\noindent

Then $(\sigma=0)\subset(\omega=0)$. By the  classification of degree zero distributions, either $ (\omega=0)$ is a point $\{p\}$  or a plane \cite[Proposition 4.3]{ACMa}. 
Note that this also implies that $\sigma$ has no divisorial zeros.  In the first case, $(\sigma=0)$ must be empty, otherwise  we would have that $(\sigma=0)\subset(\omega=0)=\{p\}$. Then $(\sigma=0)=\emptyset$ implies that $N {\mathscr{F}}^*(2)$ is  split. That is $N {\mathscr{F}}^*=\mathcal{O}_{\mathbb{P}^4}(-2)\oplus\mathcal{O}_{\mathbb{P}^4}(-3)$, thus $c_2(N {\mathscr{F}}^*)=6$. Furthermore, by Theorem \ref{conormChern} we have $\deg(Z)=10$ and $p_a(Z)=7$. In the second case,    either  $(\sigma=0)=\emptyset$   and  $N {\mathscr{F}}^*(2)$ is  split or  $(\sigma=0)$ is a plane curve.  But  if $(\sigma=0)$ is a curve,  the inequality
$$
0 <\deg(\sigma)=c_2(N {\mathscr{F}}^*(2))=\dfrac{10-\deg(Z)}{5}
$$
says us that $\deg(Z)=5$, since $Z$  is also    not empty and $c_2(N {\mathscr{F}}^*(2))$ is an integer number. Therefore, then $c_2(N {\mathscr{F}}^*(2))=1$ and $c_1(N {\mathscr{F}}^*(2))=-1$ and such    vector bundle  can not exist by Schwarzenberger's condition
$$(S_4^2):c_2(c_2+1-3c_1-2c_1^2)\equiv 0(12).$$
Then, $(\sigma=0)=\emptyset$   and  $N {\mathscr{F}}^*(2)$ is  split.  
 Now, suppose that  $h^0((N {\mathscr{F}}^*)_\eta)=0$, then $N {\mathscr{F}}^*$ is a locally free stable sheaf with
\begin{center}
$c_1(N {\mathscr{F}}^*(2))=-1$ and $c_2(N {\mathscr{F}}^*(2))=\dfrac{10-\deg(Z)}{5}$.
\end{center}
By Bogomolov's inequality, we have $0<\deg(Z)< 35/4$. 
In all   possible cases either $c_2(N {\mathscr{F}}^*)$ or $p_a(Z)$ are not integers numbers. Then $(N {\mathscr{F}}^*)_{\eta}$ can not be stable. 
\end{proof}

\begin{example}
 Consider the dimension two foliation $\mathscr{F}$ of degree $2$ given by
 $\omega=\omega_1\wedge \omega_2$, where 
$$w_1=-z_1dz_0+z_0dz_1-z_3dz_2+z_2dz_3$$
and 
$$w_2=(z_1z_0-z_4^2)dz_0+(z_2z_1-z_0^2)dz_1+(z_3z_2-z_1^2)dz_2+(z_4z_3-z_2^2)dz_3+(z_0z_4-z_3^2)dz_4.$$
Then, 
\begin{equation*}
    \begin{split}
       \omega & = \omega_1\wedge \omega_2 \\
        & = (z_0z_4^2 -z_1^2z_2)dz_0\wedge dz_1+(z_1^3+z_0z_1z_3-z_1z_2z_3-z_3z_4^2)dz_0\wedge dz_2+ \\
        & +(z_2z_4^2-z_0z_1z_2+z_1z_2^2-z_1z_3z_4)dz_0\wedge dz_3+(z_1z_3^2-z_0z_1z_4)dz_0\wedge dz_4+  (z_0^2z_4-z_0z_3^2)dz_1\wedge z_4 +\\
        & +( z_1z_2z_3-z_0z_1^2-z_0^2z_3+z_0z_2z_3)dz_1\wedge dz_2+(z_0^2z_2-z_0z_2^2-z_1z_2^2+z_0z_3z_4)dz_1\wedge dz_3  + \\
        & +(z_1^2z_2-z_3^2z_4)dz_2\wedge dz_3+(z_3^3-z_0z_3z_4)dz_2\wedge dz_4 +(z_0z_2z_4-z_2z_3^2)dz_3\wedge dz_4.
    \end{split}
\end{equation*}
The singular scheme of $\mathrm{Sing}(\mathscr{F}) =\{\omega=0\}$ is a curve of degree $10$ and arithmetic
genus $7$ which is singular at the points $(0:0:1:0:0)$ and $(1:0:1:0:0)$.  It follows from Proposition \ref{conormal-Engel} that $\mathscr{F}$ is an Engel distribution. In fact,  $\mathscr{F}$ is the blow-down of the Cartan prolongation of the canonical contact structure $-z_1dz_0+z_0dz_1-z_3dz_2+z_2dz_3$ on $\mathbb{P}^3$.   
\end{example}

The following example is the compactification to $\mathbb{P}^4$ of an example given by  Cerveau and Lins Neto in  \cite[pg 12]{CLN2}. 
\begin{example}(A non-integrable   distribution of degree $2$ with   isolated singularities)
Consider the degree $2$ distribution $\mathscr{F}$ induced by the 2-form
 \begin{equation*}
    \begin{split}
        \omega & = [z_4^3+(z_1z_2+z_3z_4)z_2+z_1^2z_3]dz_0\wedge dz_1 + [ z_3^3-(z_1z_2+z_3z_4)z_1+z_0^2z_4]dz_0\wedge dz_2- \\
        & - [z_1^3 -(z_1z_2-z_3z_4)z_4+z_3^2z_0]dz_0\wedge dz_3 - [z_2^3+(z_1z_2-z_3z_4)z_3 +z_4^2z_1]dz_0\wedge dz_4+ \\
        & + (z_1z_2+z_3z_4)z_0dz_1\wedge dz_2 + (z_1z_2-z_3z_4)z_0dz_3\wedge dz_4+\\
        & + z_1^2z_0dz_1\wedge dz_3 + z_4^2z_0dz_1\wedge dz_4 + z_3^2z_0dz_2\wedge dz_3 + z_2^2z_0dz_2\wedge dz_4.
    \end{split}
\end{equation*}
 The singular set of  $\mathscr{F}$ is the union of the points $(1:0:0:0:0)$,  $(0:1:0:-1:1)$ and 
   the  conic  given by   $\{z_0=z_2=z_1^2+z_4^2+z_3^2 -z_1z_3 -z_3z_4=0\}$. 
\end{example}

\subsection{Proof of Theorem 1} It follows from Propositions \ref{lfree1} and \ref{lfree2} that 
  $T \mathscr{F}$ is either isomorphic to $\mathcal{O}_{\mathbb{P}^4}(1)\oplus \mathcal{O}_{\mathbb{P}^4}(1-d)$ or
$\mathcal{O}_{\mathbb{P}^4} \oplus \mathcal{O}_{\mathbb{P}^4}$, where $d$ is the degree of $\mathscr{F}$.
Let  $v_i: \mathcal{O}_{\mathbb{P}^4}(a_i)\to \mathcal{O}_{\mathbb{P}^4}(a_1)\oplus \mathcal{O}_{\mathbb{P}^4}(a_2) \simeq  T  \mathscr{F}$  the  rational vector fields generating $T  \mathscr{F}$, where $a_1+a_2=2-d $. 
If $T \mathscr{F}=\mathcal{O}_{\mathbb{P}^4}(1)\oplus \mathcal{O}_{\mathbb{P}^4}(1-d)$, then $a_1=1$ and $v_1$ is represented, in homogeneous coordinates, by a constant vector field and $v_2$  is induced by a homogeneous polynomial vector field of degree 1 or 2. 
If  $T\mathscr{F}=\mathcal{O}_{\mathbb{P}^4}\oplus \mathcal{O}_{\mathbb{P}^4}$, then $\mathscr{F}$ has degree 2 and $v_1$ and $v_2$ are represented, in homogeneous coordinates, by linear vector fields.

Suppose that $\mathscr{F}$ is integrable and   $T\mathscr{F}\simeq \mathcal{O}_{\mathbb{P}^4}(1)\oplus \mathcal{O}_{\mathbb{P}^4}(1-d)$,  then the   foliation by curves induced by $v_1: \mathcal{O}_{\mathbb{P}^4}(1)\to   T  \mathscr{F}$, is given by a rational linear projection $\rho: \mathbb{P}^4  \dashrightarrow \mathbb{P}^3$ and $\mathscr{F}$ is the pull-back by $\rho$ of a codimension one foliation of degree $d$ on $\mathbb{P}^3$.    If    $T\mathscr{F}\simeq \mathcal{O}_{\mathbb{P}^4}\oplus \mathcal{O}_{\mathbb{P}^4}$, then 
 $T\mathscr{F}\simeq \mathfrak{g}\otimes \mathcal{O}_{\mathbb{P}^4}$,  where   either  $\mathfrak{g}$ is   an abelian Lie algebra of dimension $2$  or  $\mathfrak{g}\simeq \mathfrak{aff}(\mathbb{C})$.

Suppose that $\mathscr{F}$ is not integrable and that $\mathscr{F}^{[1]}$ is integrable:

if  $T \mathscr{F} \simeq \mathcal{O}_{\mathbb{P}^4}(1) \oplus \mathcal{O}_{\mathbb{P}^4}(1-d) \subset \mathscr{F}^{[1]}$,  then the 3-dimensional holomorphic foliation $\mathscr{F}^{[1]}$ is a linear pull-back of a codimension 1 foliation $\mathscr{G}$ in $\mathbb{P}^3$ or $\mathbb{P}^2$, and we are in case (a). Moreover, $\deg(\mathscr{F}^{[1]}) \leq \deg(v_1 \wedge v_2 \wedge [v_1, v_2]) \leq \deg(v_1)+\deg(v_2)+\deg([v_1, v_2])= 3$.

If $T \mathscr{F} \simeq \mathcal{O}_{\mathbb{P}^4} \oplus \mathcal{O}_{\mathbb{P}^4}$, consider  $v_1: \mathcal{O}_{\mathbb{P}^4} \to T \mathscr{F} $ and  $v_2: \mathcal{O}_{\mathbb{P}^4} \to T \mathscr{F}$, vector fields inducing $\mathscr{F}$ and $[v_1,v_2]$.    Then $\mathscr{F}^{[1]}$  is induced by $v_1$, $v_2$ and $v_3=[v_1,v_2]/f$, where $f$ is a polynomial of degree $m=0,1$ and $v_3$ is a vector field of degree $1-m$.
If $m=1$, then $\mathscr{F}^{[1]}$ has degree 
$ \leq 2$,   is a linear pull-back of a foliation of degree $\leq 2$, in $\mathbb{P}^3$, and we are in case (a).   
 If  $m=0$, then $T \mathscr{F} \simeq \mathcal{O}_{\mathbb{P}^4} \oplus \mathcal{O}_{\mathbb{P}^4}\subset \mathscr{F}^{[1]}\simeq \mathcal{O}_{\mathbb{P}^4} \oplus \mathcal{O}_{\mathbb{P}^4} \oplus \mathcal{O}_{\mathbb{P}^4}$, 
$\mathscr{F}^{[1]}$ has degree  $3$ and 
$\mathscr{F}^{[1]}\simeq \mathfrak{g}\otimes \mathcal{O}_{\mathbb{P}^4}$,  
 where $\mathfrak{g}=H^0(\mathscr{F}^{[1]}) := \mathfrak{f}^{[1]}$ is a non-abelian Lie algebra of dimension 3, since $\mathfrak{f}^{[1]}$ contains the non-abelian Lie algebra $H^0(T\mathscr{F})$ associated with the non-integrable distribution $\mathscr{F}$. 
  Thus,  the  rank of  $[\mathfrak{f}^{[1]},\mathfrak{f}^{[1]}] $  is $\geq 1$. 
Now, define 
$$  \mathfrak{r}_{3,\lambda}(\mathbb{C}):=\{[u_1,u_2]=u_2; \ [u_1,u_3]=\lambda u_3;\ [u_2,u_3]=0 \},$$  where $\lambda \in \mathbb{C}^*$ with  $\lambda\in \{1,-1\}$, or
     $0<|\lambda|<1$, or  $|z|=1$ and $\mathfrak{Im}(z)>0$. 
     By the classification of complex  non-abelian Lie algebra of dimension 3, \cite[Lemma 2]{BSt} and \cite[ Theorem 2.1]{FR},  we have that: 
     
  \noindent (a) if $\dim[\mathfrak{f}^{[1]}, \mathfrak{f}^{[1]}]=1$, then     $\mathfrak{g} $ is either  isomorphic to
       \begin{itemize}
    \item $ \mathfrak{n}_{3}(\mathbb{C})$ the $3$-dimensional Heisenberg Lie algebra,
    \item or   $  \mathfrak{aff}(\mathbb{C}) \oplus \mathbb{C}:=\{ [u_1,u_3]=0,\ [u_2,u_3]=0,\ [u_1,u_2]=u_2\}$;
  \end{itemize}

 \noindent(b) if $\dim[\mathfrak{f}^{[1]}, \mathfrak{f}^{[1]}]=2$, then     $\mathfrak{g} $ is either   isomorphic to 
       \begin{itemize}
    \item  $\mathfrak{r}_{3}(\mathbb{C}):=\{ [u_1,u_2]=u_2,\,[u_1,u_3]=u_2+u_3,\, [u_2,u_3]=0\}$,
    \item  or  $\mathfrak{r}_{3,\lambda}(\mathbb{C})$; 
  \end{itemize}
  
   \noindent (c) if $\dim[\mathfrak{f}^{[1]}, \mathfrak{f}^{[1]}]=3$, then     $\mathfrak{g} $ is  isomorphic $\mathfrak{sl}_{2}(\mathbb{C}).$

   It follows from \cite[    Th\'eor\`eme 6.6. and  Th\'eor\`eme 6.9]{CD2}
   that $\mathfrak{g} $  can not be $\mathfrak{r}_{3}(\mathbb{C})$ nor the Heisenberg Lie algebra.  Moreover, 
   $\mathfrak{f}^{[1]}$ can not be isomorphic to $\mathfrak{r}_{3,1}(\mathbb{C})$. 
   In fact, suppose that  $\mathfrak{f}^{[1]}\simeq \mathfrak{r}_{3,1}(\mathbb{C})$  and consider $v_1, v_2 \in \mathfrak{r}_{3,1}(\mathbb{C})$ vector fields  tangent to  $\mathscr{F}$.
   Then, we can write
   $v_1=a_1u_1+ a_2u_2+a_3u_3$ and  $v_2=b_1u_1+ b_2u_2+b_3u_3$, with $a_i,b_i\in \mathbb{C}$, for all $i=1,2,3$.  Thus, by using the relations of $\mathfrak{r}_{3,1}(\mathbb{C})$ we conclude that $[v_1,v_2]=(a_1b_2-a_2b_1)u_2+ \lambda(a_1b_3-a_3b_1)u_2$.
   Now, 
   $$v_1\wedge v_2 \wedge [v_1,v_2]=(\lambda-1)(a_1b_2-a_2b_1)(a_1b_2-a_2b_1)(u_1\wedge u_2\wedge u_3)$$ is not a zero holomorphic  3-vector  only if   $
   (\lambda-1)(a_1b_2-a_2b_1)(a_1b_2-a_2b_1)\neq 0$. In particular, $\lambda$ cannot be $1$.

Suppose that $\mathscr{F}$ is an Engel structure: 

if $T \mathscr{F} \simeq \mathcal{O}_{\mathbb{P}^4}(1)\oplus \mathcal{O}_{\mathbb{P}^4}(1-d)$.     Consider $v_1: \mathcal{O}_{\mathbb{P}^4}(1)\to T \mathscr{F} $ and $v_2: \mathcal{O}_{\mathbb{P}^4}(1-d)\to T \mathscr{F}$, vector fields that induce $\mathscr{F}$.
Then $\mathscr{F}^{[1]}$  is induced by $v_1$, $v_2$ and $v_3=[v_1,v_2]/f$, where $f$ is a polynomial of degree $0\leq m \leq d$ and $v_3$ is a vector field of degree $d-1-m$, since $v_1$ is a constant vector field.
 We have seen that $\mathscr{F}^{[1]}$ has the degree $0,1,2$ or $3$. 
If the characteristic foliation $\mathcal{L}(\mathscr{F})$ is $v_1: \mathcal{O}_{\mathbb{P}^4}(1)\to T \mathscr{F} $. Let $\alpha$ be the one-form inducing $\mathscr{F}^{[1]}$ and consider  
$\psi: \mathbb{P}^4  \dashrightarrow \mathbb{P}^4/ \mathcal{L}(\mathscr{F})\simeq \mathbb{P}^3$  the rational linear map whose lines are the leaves of $\mathcal{L}(\mathscr{F})$.
The conditions $\mathcal{L}(\mathscr{F})\subset \mathscr{F}^{[1]} $
and 
$[\mathcal{L}(\mathscr{F}),\mathscr{F}^{[1]}]\subset  \mathscr{F}^{[1]}$ are  equivalent to $i_{v_1} \alpha=  0$  and  
$ i_{v_1} d\alpha=  \mu \alpha$, for some $\mu\in \mathbb{C}^*$, since $v_1$ is a constant vector field.
This implies that there is a contact distribution $\mathscr{G}$ on $\mathbb{P}^3$ such that $\mathscr{F} = \rho^*\mathscr{G}$, whose tangent sheaf is isomorphic to $N(1)$, $\mathcal{O}_{\mathbb{P}^3} \oplus \mathcal{O}_{\mathbb{P}^3}(1)$, $\mathcal{O}_{\mathbb{P}^3}(-1) \oplus \mathcal{O}_{\mathbb{P}^3}(1)$, or $\mathcal{O}_{\mathbb{P}^3} \oplus \mathcal{O}_{\mathbb{P}^3}(-1)$, where $N$ is the null-correlation bundle if $\mathscr{G}$ has degree 0.  
Then $\mathscr{F}$ is  a   blow-down  of the Cartan prolongation of the singular contact  structure induced by $\alpha$. If the tangent bundle of $\mathscr{G}$ is $N(1)$, the result follows from Proposition \ref{Null:contact}, see also Example \ref{exe:null:deg1}.
For the other cases, consider the Cartan prolongation $\mathbb{P}\mathscr{D}$  of $\mathscr{D}$ on
$\pi:  \mathbb{P}(T\mathscr{F}^*)\simeq \mathbb{P}(\mathcal{O}_{\mathbb{P}^3}(-1)  \oplus  \mathcal{O}_{\mathbb{P}^3}(a-1)) \to \mathbb{P}^3$, where $a=1,2$.  
 We have the commutative diagram. 
\begin{center}
\begin{tikzcd}[row sep=large]
 \mathbb{P}(T\mathscr{F}^*)\simeq \mathbb{P}(\mathcal{O}_{\mathbb{P}^3}   \oplus  \mathcal{O}_{\mathbb{P}^3}(d)) \ar[rr, "  \rho"] \ar[d, " \pi "'] &&  \mathbb{P}^4  \ar[dll, bend left=20,dashed,  "\psi"] \\
  \mathbb{P}^3\simeq  \mathbb{P}^4/\mathcal{L}(\mathscr{D}) &&
\end{tikzcd}
\end{center}
where $ \rho: \mathbb{P}(T\mathscr{F}^*)\simeq \mathbb{P}(\mathcal{O}_{\mathbb{P}^3}   \oplus  \mathcal{O}_{\mathbb{P}^3}(d))\to \mathbb{P}^4$ is the   contraction of the divisor  $\mathbb{P}(\mathcal{O}_{\mathbb{P}^3}) \subset  \mathbb{P}(T\mathscr{F}^*)$.  
 Thus,  pushing forward the structures, we conclude that  $\rho_*\mathcal{L}(\mathbb{P}\mathscr{D})=\mathcal{L}\mathscr{D}$,  $\rho_* \mathbb{P}\mathscr{D}^1= \mathscr{D}^1$ and  
$\rho_*\mathbb{P}\mathscr{D}= \mathscr{D}$.
Now, if the characteristic foliation $\mathcal{L}(\mathscr{F})$ is $v_1: \mathcal{O}_{\mathbb{P}^4}(-1)\to T \mathscr{F} $
 then $\mathcal{L}(\mathscr{F})$ has degree 2. 
 In this case, we can suppose that
  $v_2: \mathcal{O}_{\mathbb{P}^4}(1)\to T \mathscr{F} $ is induced by the constant vector field $\frac{\partial}{\partial z_0}$. Thus 
  $v_1=\nu_2+z_0\nu_1+z_0^2\nu_0$, where $\nu_i$ is a  vector field of degree $i=0,1,2$,  such that $[v_2,\nu_i]=0$. Therefore, $\mathscr{F}^{[1]}$ is generated by $v_1$, $v_2$, and $[v_1,v_2]=-2z_0\nu_1-\nu_2$.
 So, $\mathscr{F}$ is of Lorentzian type; see Remark \ref{Lorentzian}.

Suppose that if $T \mathscr{F} \simeq \mathcal{O}_{\mathbb{P}^4}\oplus \mathcal{O}_{\mathbb{P}^4}$.  
Then the characteristic foliation is a non-trivial global section $\mathscr{G}: 0\to  \mathcal{O}_{\mathbb{P}^4}\to  T\mathscr{F}$ which corresponds to a degree-one foliation by curves which are tangent either to a   $\mathbb{C}^*$-action   or a $\mathbb{C}$-action. If it is induced by a $\mathbb{C}^*$-action,  since the singular scheme of the characteristic foliation has codimension $\geq 3$, then, after a biholomorphism of $\mathbb{P}^4 $, we have that the vector field
is of the type
\begin{center}
      $  v=a_0z_{0}\dfrac{\partial}{\partial z_0}  + a_1z_{1}\dfrac{\partial}{\partial z_1} +   a_2z_{2}\dfrac{\partial}{\partial z_2}+a_{3}z_{3}\dfrac{\partial}{\partial z_{3}}$, 
\end{center}
 where $(a_0,a_{1},a_{2},a_{3})\in (\mathbb{N})^{4}$, with $0\leq a_0 \leq a_1\leq a_2 < a_3$ and $a_0=0$ only if $a_1>0$.    If $a_0>0$,  we have a rational map $
 \psi: \mathbb{P}^4  \dashrightarrow   \mathbb{P}(a_0,a_{1},a_{2},a_{3})$, where     $\mathbb{P}(a_0,a_{1},a_{2},a_{3})$ is  a weighted projective  plane and  the leaves of $\mathcal{L}(\mathscr{F})$ are tangent to the fibers.  The conditions $\mathcal{L}(\mathscr{F})\subset \mathscr{F}^{[1]} $
and $[\mathcal{L}(\mathscr{F}),\mathscr{F}^{[1]}]\subset \mathscr{F}^{[1]}$ tell us that there is a non-integrable distribution $\mathscr{H}$ on $\mathbb{P}(a_0,a_{1},a_{2},a_{3})$ such that $\mathscr{F}^{[1]}=  \psi^*\mathscr{H}$ whose tangent sheaf is $[\psi_* (\mathscr{F}^{[1]}/\mathcal{L}(\mathscr{F}))]^{**}$. Then the Cartan prolongation of $\mathscr{H}$ is birational to $\mathscr{F}$.    The case where $a_0=0$ is similar. 
 
 Suppose that $\mathcal{L}(\mathscr{F})$  is induced by a $\mathbb{C}$-action. We observe that the other global vector field that generates $\mathscr{F}$ cannot also be tangent to a $\mathbb{C}$-action. In fact, if it is the case we have that associated derived Lie algebra $ [H^0(\mathscr{F}^{[1]}),H^0(\mathscr{F}^{[1]})] $ of linear vector fields with derived flag $H^0(T\mathscr{F})\subset [H^0(T\mathscr{F}),H^0(\mathscr{F}^{[1]})] \subset [H^0(\mathscr{F}^{[1]}),H^0(\mathscr{F}^{[1]})] $ is nilpotent and by Engel's Theorem \cite[Theorem 9.9]{FH} there is a non-zero element $p=(p_0,p_1,p_2,p_3,p_4)$ such that $v(p)=0$, for all $v\in [H^0(\mathscr{F}^{[1]}),H^0(\mathscr{F}^{[1]})]$. In particular $v_1(p)=v_2(p)=0$, where $v_1$ and $v_2$ are tangent to $T\mathscr{F}$.
This is a contradiction, since $T\mathscr{F}$ is locally free and $p$ would be  an embedded point of  the locally Cohen-Macaulay variety $\mbox{Sing}(\mathscr{F})$. 
 We see that the other vector field that generates $\mathscr{F}$ is tangent to a $\mathbb{C}^*$-action and this implies that $\mathscr{F}$ is of Lorentzian type. 
This proves the Theorem.

\begin{example}( $\mathscr{F}$ non-integrable and $\mathfrak{f}^{[1]}\simeq \mathfrak{aff}(\mathbb{C}) \oplus \mathbb{C}) $) 
Consider the non-integrable  distribution $\mathscr{F}$, of degree 2,  induced by a bi-vector $v_1\wedge v_2$, where 
$$
v_1=(z_0+z_1)\frac{\partial}{\partial z_0}+z_1\frac{\partial} {\partial
z_1}+ z_3\frac{\partial} {\partial
z_2}
$$
and 
$$
v_2= (\xi_0 z_0+\xi_1 z_1)\frac{\partial}{\partial z_0}+
(\xi_0-1)z_1 \frac{\partial}{\partial
z_1}+\xi_2z_3\frac{\partial}{\partial z_3}+((\xi_2+1)z_2+\xi_3z_3+\xi_4z_4)
\frac{\partial}{\partial z_2}+(\xi_5z_3+\xi_6z_4)
\frac{\partial}{\partial z_4},
$$
with  $(\xi_0,\xi_1,\xi_2, \xi_3, \xi_4,\xi_5,\xi_6)\in (\mathbb{C}^*)^7$. 
Then $v_3:=[v_1,v_2]= z_1 \frac{\partial}{\partial
z_0}+z_3 \frac{\partial}{\partial z_2}$,   $[v_1,v_3]=0$ and $[v_3,v_2]=v_3$. Then, the  foliation $\mathscr{F}^{[1]}$ is induced by $v_1\wedge v_2\wedge v_3$ and
$\mathfrak{f}^{[1]}\simeq \mathfrak{aff}(\mathbb{C}) \oplus \mathbb{C}$. These vector fields can be found in \cite{CD2}.  

\end{example}

\begin{example}( $\mathscr{F}$ non-integrable and $\mathfrak{f}^{[1]}\simeq \mathfrak{r}_{3,\lambda}(\mathbb{C})$) 
We have seen in the proof of the Theorem that  if   $\mathfrak{f}^{[1]}\simeq \mathfrak{r}_{3,\lambda}(\mathbb{C})$, with $\lambda\neq 1$,  and     
   $v_1=a_1u_1+ a_2u_2+a_3u_3$ 
   and  $v_2=b_1u_1+ b_2u_2+b_3u_3$
   are vector  inducing    $\mathscr{F}$, then  $[v_1,v_2]=(a_1b_2-a_2b_1)u_2+ \lambda(a_1b_3-a_3b_1)u_2$, 
 where    $
  (a_1b_2-a_2b_1) \neq 0$
   and 
    $
 (a_1b_3-a_3b_1)\neq 0$.

 The case  $ \mathfrak{f}^{[1]}$  when    is  isomorphic $\mathfrak{sl}_{2}(\mathbb{C})$ is similar,  but we must have  the conditions $(a_1b_2-a_2b_1) \neq 0$,   $(a_1b_3-a_3b_1)\neq 0$    and $
 (a_2b_3-a_3b_2)\neq 0$.  
\end{example}

\section{Horrocks-Mumford holomorphic distributions}

It is know that given cohomology class $(a,b)\in H^2(\mathbb{P}^4,\mathbb{Z})\oplus H^4(\mathbb{P}^4,\mathbb{Z})$, there exists a rank $2$ \textit{topological} complex $E$ with Chern class $c(E)=1+a\textbf{h}+b\textbf{h}^2$, if and only if the Schwarzenberger's condition
\begin{equation*}
    S_4^2: b\cdot(b+1-3a-2a^2)\equiv 0 \mod (12)
\end{equation*}
is satisfied. The problem lies in the existence of a holomorphic structure in a given topological vector bundle $E$.

In $1973$ G. Horrocks and D. Mumford \cite{HM} showed that there is a stable rank $2$ holomorphic vector bundle $E$ on $\mathbb{P}^4$, with Chern polynomial $c(E)=1+5\textbf{h}+10\textbf{h}^2$. This is essentially the only known indecomposable rank $2$ vector bundle on $\mathbb{P}^4$, \cite{Wolfram2}. Certainly, this is not entirely true, we can twist $E$ by a linear bundle or may pull $E$ back under a finite morphism $\pi:\mathbb{P}^4\to\mathbb{P}^4$. The normalization of the Horrocks-Mumford bundle is $E_\eta=E(-3)=F$ with Chern polynomial $c(F)=1-\textbf{h}+4\textbf{h}^2$, and $F$ is cohomology of a monad: 

\begin{equation}\label{monada}
0\rightarrow\mathcal{O}_{\mathbb{P}^4}(-1)^{\oplus 5}\xrightarrow{\alpha}\Omega_{\mathbb{P}^4}^{2}(2)^{\oplus 2}\xrightarrow{\beta}\mathcal{O}_{\mathbb{P}^4}^{\oplus 5}\rightarrow0.
\end{equation}

Recall that the polynomial $1-z+4z^2$ (or $1+5z+10z^2$) is irreducible over $\mathbb{Z}$, therefore $F$ (and hence $E$) does not split as a direct sum of line bundles. Furthermore, we can deduce from monad (\ref{monada}) that $F$ has no sections, and this shows that $F$ is in fact stable.

It follows from see \cite{Wolfram3} that the  minimal free resolution of $E$ is of type:
\begin{equation}\label{freeresolutionofE}
    0\rightarrow\mathcal{O}_{\mathbb{P}^4}(-5)^{\oplus 2}\rightarrow\mathcal{O}_{\mathbb{P}^4}(-3)^{\oplus 20}\rightarrow\mathcal{O}_{\mathbb{P}^4}(-2)^{\oplus 35}\rightarrow\mathcal{O}_{\mathbb{P}^4}(-1)^{\oplus 15}\oplus\mathcal{O}_{\mathbb{P}^4}^{\oplus 4}\rightarrow E\rightarrow 0,
\end{equation}

The dimensions of the cohomology groups of the Horrocks-Mumford bundle are described in Table (\ref{table}).

\begin{table}[h!]
\centering
\begin{tabular}{| c | c | c | c | c | c | } 
 \hline
\textbf{$k$} & \textbf{$h^0$} & \textbf{$h^1$} & \textbf{$h^2$} & \textbf{$h^3$} & \textbf{$h^4$}\\ [0.5ex] 
 \hline\hline
 $k\geq1$ & $\frac{((k+5)^2-1)((k+5)^2-24)}{12}$ & 0 & 0 & 0 & 0 \\
0 & 4 & 2 & 0 & 0 & 0 \\
-1 & 0 & 10 & 0 & 0 & 0 \\
-2 & 0 & 10 & 0 & 0 & 0 \\
-3 & 0 & 5 & 0 & 0 & 0 \\
-4 & 0 & 0 & 0 & 0 & 0 \\
-5 & 0 & 0 & 2 & 0 & 0 \\
-6 & 0 & 0 & 0 & 0 & 0 \\
-7 & 0 & 0 & 0 & 5 & 0 \\
-8 & 0 & 0 & 0 & 10 & 0 \\
-9 & 0 & 0 & 0 & 10 & 0 \\
-10 & 0 & 0 & 0 & 2 & 4 \\ 
$k\leq-11$ & 0 & 0 & 0 & 0 & $\frac{((k+5)^2-1)((k+5)^2-24)}{12}$ \\ [1ex] 
 \hline
\end{tabular}
\caption{Table for $h^i=\dim H^i(E(k))$}
\label{table}
\end{table}

Hence, by Hirzebruch--Riemann--Roch theorem:

\begin{equation}\label{eulerHM}
    \chi\big(E(k)\big)=\frac{1}{12}\big((k+5)^2-1\big)\big((k+5)^2-24\big), \,\,\ k\in\mathbb{Z}.
\end{equation}
For more details, see \cite[pag. 74]{HM}.

Let $\mathscr{M}_{\mathbb{P}^4}(-1,4)$ be the moduli space that describes the Horrocks-Mumford bundle. $H^1(\mathbb{P}^4,\mathcal{H}om(F,F))$  is isomorphic to the Zariski tangent space of $\mathscr{M}_{\mathbb{P}^4}(-1,4)$ in the point $F$ defined by $[F]$ and from $H^2(\mathbb{P}^4,\mathcal{H}om(F,F))=0$ follows $\mathscr{M}_{\mathbb{P}^4}(-1,4)$ is smooth in $[F]$.

The following results are provided in \cite[pag. 104]{Wolfram} and \cite[pag. 218]{Wolfram2}:

\begin{theorem}\label{Wolfram} Let $F$ be the Horrocks-Mumford bundle. Then:

\begin{itemize}
    \item $h^1(\mathbb{P}^4,\mathcal{H}om(F,F))=24$, $h^1(\mathbb{P}^4,\mathcal{H}om(F,F))=5$, $h^1(\mathbb{P}^4,\mathcal{H}om(F,F)(k))=0$ for $k\leq -2$ and $h^2(\mathbb{P}^4,\mathcal{H}om(F,F))=2.$
    \item $\mathscr{M}_{\mathbb{P}^4}(-1,4)$ is smooth in $[F]$ with dimension $24$.
\end{itemize}
\end{theorem}

\subsection{Vanishing Lemmas}

\noindent

The following Vanishing Lemmas are consequences of having considered  Euler's exact sequence twisted by a convenient sheaf and its long exact sequence of Cohomology. Using the Cohomologies table (\ref{table}) of the Horrocks-Mumford bundle and the free resolution of $E$, (\ref{freeresolutionofE}), we obtain:

\begin{lema}\label{Vlema1}
Let $E$ be the Horrocks-Mumford bundle. Then for $j=1,2,3$ and $k\in\mathbb{Z}$ we have:

\begin{enumerate}
    \item $h^0(\Omega_{\mathbb{P}^4}^{j}\otimes E(k))=0$ \,\,\,\ \textrm{for} \,\,\,\ $k\leq j$.
    \item $h^1(\Omega_{\mathbb{P}^4}^{j}\otimes E(k))=0$ \,\,\,\ \textrm{for} \,\,\,\ $k\leq j-4$  \,\,\,\,\,\,\ \textrm{or} \,\,\,\ $k\geq j+1$.
    \item $h^2(\Omega_{\mathbb{P}^4}^{j}\otimes E(k))=0$ \,\,\,\ \textrm{for} \,\,\,\ $k\leq j-6$  \,\,\,\,\,\,\ \textrm{or} \,\,\,\ $k\geq j-3$.
    \item $h^3(\Omega_{\mathbb{P}^4}^{j}\otimes E(k))=0$ \,\,\,\ \textrm{for} \,\,\,\ $k\leq j-10$ \,\,\,\ \textrm{or} \,\,\,\ $k\geq j-5$.
    \item $h^4(\Omega_{\mathbb{P}^4}^{j}\otimes E(k))=0$ \,\,\,\ \textrm{for} \,\,\,\  $k\geq j-9$.
\end{enumerate}

\end{lema}

\begin{proof}
Let us consider the Euler exact sequence: 
\begin{equation*}\label{Euler1}
    0\rightarrow\Omega_{\mathbb{P}^4}^{1}(1)\rightarrow\mathcal{O}_{\mathbb{P}^4}^{\oplus 5}\rightarrow\mathcal{O}_{\mathbb{P}^{4}}(1)\rightarrow 0.
\end{equation*}
    Then, twisting by $E(k-1)$ we have:
\begin{equation}\label{Euler1torcido}
    0\rightarrow\Omega_{\mathbb{P}^4}^{1}\otimes E(k)\rightarrow E(k-1)^{\oplus 5}\rightarrow E(k)\rightarrow 0.
\end{equation}
Now, taking a long exact sequence of cohomology, we get: 
$$0\rightarrow H^{0}(\mathbb{P}^4, \Omega_{\mathbb{P}^4}^{1}\otimes E(k))\rightarrow H^{0}(\mathbb{P}^4, E(k-1))^{\oplus 5}\rightarrow H^{0}(\mathbb{P}^4, E(k))\rightarrow\cdots$$
and by cohomology table \ref{table} of the Horrocks-Mumford bundle, we obtain:   
\begin{equation*}
    h^0(\mathbb{P}^4,E(k-1))=0 \,\,\ \textrm{for} \,\,\ k\leq 0,
\end{equation*}
then $h^{0}(\mathbb{P}^4, \Omega_{\mathbb{P}^4}^{1}\otimes E(k))=0$ for $k\leq 0$. By free resolution (\ref{freeresolutionofE}) we also see that $$h^{0}(\mathbb{P}^4, \Omega_{\mathbb{P}^4}^{1}\otimes E(1))=0,$$ thus: 
$$h^0(\mathbb{P}^4, \Omega_{\mathbb{P}^4}^{1}\otimes E(k))=0 \,\,\ \textrm{for} \,\,\,\ k\leq 1.$$
\end{proof}

\begin{lema}\label{Vlema3}
Let $E$ be the Horrocks-Mumford bundle and $k\in\mathbb{Z}$, then:
\begin{enumerate}
    \item $h^0(E\otimes E(k))=0$ \,\,\,\ \textrm{for} \,\,\,\ $k\leq -6$.
    \item $h^1(E\otimes E(k))=0$ \,\,\,\ \textrm{for} \,\,\,\ $k\leq -7$  \,\,\,\,\,\,\ \textrm{or} \,\,\,\ $k\geq 1$.
    \item $h^2(E\otimes E(k))=0$ \,\,\,\ \textrm{for} \,\,\,\ $k\leq -11$  \,\,\,\ \textrm{or} \,\,\,\ $k\geq -4$.
    \item $h^3(E\otimes E(k))=0$ \,\,\,\ \textrm{for} \,\,\,\ $k\leq -16$  \,\,\,\ \textrm{or} \,\,\,\ $k\geq -8$.
    \item $h^4(E\otimes E(k))=0$ \,\,\,\ \textrm{for} \,\,\,\ $k\geq -9$.
\end{enumerate}
\end{lema}

\begin{proof}

Let $\mathcal{K}=\Ker\beta$ of the monad (\ref{monada}) and consider the short exact sequence:
\begin{equation*}\label{seq2display}
     0\rightarrow\mathcal{O}_{\mathbb{P}^4}(-1)^{\oplus 5}\rightarrow\mathcal{K}\rightarrow E(-3)\rightarrow 0.
\end{equation*}
Twisting by $E(k+3)$ and passing to cohomology we have $h^0(E\otimes E(k))=h^0(\mathcal{K}\otimes E(k+3))$ for $k\leq -6$. 
Considering  the exact sequence of the monad (\ref{monada})
\begin{equation*}\label{seqdisplay1}
   0\rightarrow\mathcal{K}\rightarrow \Omega_{\mathbb{P}^4}^2(2)^{\oplus 2}\rightarrow\mathcal{O}_{\mathbb{P}^4}^{\oplus 5}\rightarrow 0,
\end{equation*}
  and  twisting by $E(k+3)$, passing to cohomology and by Lemma \ref{Vlema1} we have $h^0(\Omega_{\mathbb{P}^4}^2\otimes E(k+5))=0$ for $k\leq -3$, hence $h^0(\mathcal{K}\otimes E(k+3))=0$ for $k\leq -3$. Thus:
\begin{equation*}
    h^0(E\otimes E(k))=0 \,\,\,\ \textrm{for} \,\,\,\ k\leq -6.
\end{equation*}
\end{proof}
In addition, by using Serre's duality and Theorem \ref{Wolfram}, we build the table (\ref{table2}). 
\begin{table}[h!]
\centering
\begin{tabular}{| c | c | c | c | c | c | } 
 \hline
\textbf{$k$} & \textbf{$h^0$} & \textbf{$h^1$} & \textbf{$h^2$} & \textbf{$h^3$} & \textbf{$h^4$}\\ [0.5ex] 
 \hline\hline 
 $k\geq1$ & $\frac{k^4+30k^3+290k^2+975k+624}{6}$ & 0 & 0 & 0 & 0  \\
0 & 136 & 32 & 0 & 0 & 0  \\
-1 & 70 & 85 & 0 & 0 & 0  \\
-2 & 35 & 100 & 0 & 0 & 0  \\
-3 & 15 & 85 & 0 & 0 & 0  \\
-4 & 5 & 55 & 0 & 0 & 0 \\
-5 & 1 & 24 & 2 & 0 & 0  \\
-6 & 0 & 5 & 10 & 0 & 0  \\
-7 & 0 & 0 & 20 & 0 & 0  \\
-8 & 0 & 0 & 20 & 0 & 0  \\
-9 & 0 & 0 & 10 & 5 & 0  \\
-10 & 0 & 0 & 2 & 24 & 1  \\
-11 & 0 & 0 & 0 & 55 & 5  \\
-12 & 0 & 0 & 0 & 85 & 15  \\
-13 & 0 & 0 & 0 & 100 & 35  \\
-14 & 0 & 0 & 0 & 85 & 70  \\
-15 & 0 & 0 & 0 & 32 & 136  \\
$k\leq-16$ & 0 & 0 & 0 & 0 & $\frac{k^4+30k^3+290k^2+975k+624}{6}$  \\ [1ex] 
 \hline
\end{tabular}
\caption{Table  for  $h^i=\dim H^i(E\otimes E(k))$}
\label{table2}
\end{table}
Hence, by Hirzebruch- Riemann-Roch Theorem, we have:
\begin{equation}\label{eulerHM**HM}
    \chi\big(E\otimes E(k)\big)=\frac{1}{6}(k^4+30k^3+290k^2+975k+624), \,\,\ k\in\mathbb{Z}.
\end{equation}


\subsection{Horrocks-Mumford distributions as subsheaves of the  cotangent bundle}
In this section, we study dimension $2$ holomorphic distributions which are  induced by a Bertini-type Theorem \cite{MOJ, Ottaviani}. We determine the genus, degree and the Rao module of the singular scheme.  The techniques used are based on the study of the long exact cohomology sequence of (\ref{distCotangFam}) together with the knowledge of the cohomologies of the Horrocks-Mumford bundle.

We will use the next result proved in \cite[pag. 427]{Su}. 

\begin{propo}[H. Sumihiro - 1998]\label{Sumihiro}
    $E(1)$ is generated by global sections.
\end{propo}     

Thus, $E(a)$, for $a\geq 1$, is a vector bundle generated by global sections. 
Since $\Omega_{\mathbb{P}^4}^1(2)$ is globally generated, then by the Bertini type Theorem such that $\mathcal{G}=E(a)$, $a\geq 1$, and $L=\mathcal{O}_{\mathbb{P}^4}(2)$, we have
$$N {\mathscr{F}_a}^*=E(-a-7), \,\,\ (a\geq 1)$$ 
is the conormal sheaf of a dimension $2$ holomorphic distribution in $\mathbb{P}^4$:
\begin{equation}\label{distCotangFam}
    \mathscr{F}_a : 0\rightarrow E(-a-7)\xrightarrow{\varphi} \Omega_{\mathbb{P}^4}^{1}\rightarrow \mathcal{Q}_{\mathscr{F}_a}\rightarrow 0, \,\,\ (a\geq 1).
 \end{equation}
of degree $d_a:=\deg(\mathscr{F}_a)=2a+6$. 
In addition, since $E(a)\otimes \Omega_{\mathbb{P}^4}^1(2)=\mathcal{H}om(E(-a-7)),\Omega_{\mathbb{P}^4}^1)$ is globally generated,  then by Theorem \ref{Bertini1}  a  generic morphism $$\varphi:E(-a-7)\rightarrow \Omega_{\mathbb{P}^4}^1$$  induces a distributions whose singular scheme  $\Sing(\mathscr{F}_a):=\Sing(\varphi)$ is a  subvariety on $\mathbb{P}^4$ of  expected  pure  dimension  equal to one.

\subsection{Geometric properties of the singular scheme}
Since $$c(E(-a-7))  = 1+(-2a-9)\textbf{h}+(a^2+9a+24)\textbf{h}^2,$$ then as a direct consequence of Theorem \ref{conormChern} we determine the numerical invariants of the singular scheme.

\begin{propo}[Numerical invariants of the singular locus]
Let $Z_a=\Sing(\mathscr{F}_a)$ the singular scheme, then for $a\geq 1$ we have:
\begin{enumerate}
    \item $\deg(Z_a)=4a^3+39a^2+113a+92.$
    \\
    \item $p_a(Z_a)=9a^4+107a^3+\dfrac{847}{2}a^2+\dfrac{1261}{2}a+260.$
\end{enumerate}
\end{propo}

\begin{propo}\label{propo1}
The singular scheme $Z_a=\Sing(\mathscr{F}_a)$ is smooth and connected. 
\end{propo}

\begin{proof}
Theorem \ref{Bertini1} tells us that $Z_a$ is smooth, so it suffices to show that $Z_a$ is connected. In fact, since $Z_a$ has pure expected codimension $3$ then the ideal sheaf admits the Eagon-Northcott resolution (\ref{ENR}). Consider this complex associated with the morphism $\varphi^\vee:T\mathbb{P}^4\rightarrow E(a+2)$:
    \begin{equation}\label{ENcotangFam}
        0\rightarrow S_2(E(-a-7))(-2a-4)\rightarrow\Omega_{\mathbb{P}^4}^1\otimes E(-3a-11)\rightarrow\Omega_{\mathbb{P}^4}^2(-2a-4)\xrightarrow{\alpha}\mathscr{I}_{Z_a}\rightarrow 0.
    \end{equation}
breaking into short exact sequences and passing to cohomology 
\begin{equation}\label{seqc1}
    0\rightarrow S_2(E(-a-7))(-2a-4)\rightarrow\Omega_{\mathbb{P}^4}^1\otimes E(-3a-11)\rightarrow K\rightarrow 0,
\end{equation}
and
\begin{equation}\label{seqc2}
    0\rightarrow K\rightarrow\Omega_{\mathbb{P}^4}^2(-2a-4)\rightarrow\mathscr{I}_{Z_a}\rightarrow 0.
\end{equation}
where $K=\Ker \alpha$.
From exact sequence (\ref{seqc1}) passing to cohomology
$$\cdots\to H^i(S_2(E(-a-7))(-2a-4))\to H^i(\Omega_{\mathbb{P}^4}^1\otimes E(-3a-11))\to H^i(K)\to\cdots.$$
By Lemma \ref{Vlema1} and since $a\geq 1$ we have 
$$H^2(K)\simeq H^3(S_2(E(-a-7))(-2a-4)).$$
On the other hand, if $V$ is a vector space over $\mathbb{C}$ then $V\otimes_{\mathbb{C}}V=S_2(V)\oplus\bigwedge^2(V)$, i.e, every $2$-tensor may be written uniquely as a sum of a symmetric and an alternating tensor, therefore the second tensor power of a vector bundle decomposes as the direct sum of the symmetric and alternating squares as vector bundles. So, twisting by $\mathcal{O}_{\mathbb{P}^4}(-2a-4)$ we have:
$$E\otimes E(-4a-18)\simeq S_2(E(-a-7))(-2a-4)\oplus\mathcal{O}_{\mathbb{P}^4}(-4a-13),$$
thus
$$h^i(E\otimes E(-4a-18))= h^i(S_2(E(-a-7))(-2a-4))+ h^i(\mathcal{O}_{\mathbb{P}^4}(-4a-13)), \,\,\ i=0,\dots ,4.$$
Hence, since $a\geq 1$ then using Lemma \ref{Vlema3} and by Bott's formula \cite[pag. 4]{Oko} we have $$h^3(S_2(E(-a-7))(-2a-4))=h^2(K)=0.$$

From exact sequence (\ref{seqc2}) passing to cohomology 
$$\cdots\to H^i(\Omega_{\mathbb{P}^4}^2(-2a-4))\to H^i(\mathscr{I}_{Z_a})\to H^{i+1}(K)\to H^{i+1}(\Omega_{\mathbb{P}^4}^2(-2a-4))\to\cdots$$
Since $a\geq 1$ using Bott's formula, $$H^2(K)\simeq H^1(\mathscr{I}_{Z_a}),$$ hence $h^1(\mathscr{I}_{Z_a})=0$. From exact sequence
$$0\to\mathscr{I}_{Z_a}\to\mathcal{O}_{\mathbb{P}^4}\to\mathcal{O}_{Z_a}\to 0$$
passing to cohomology, we have:
$$0\to H^0(\mathscr{I}_{Z_a})\to H^0(\mathcal{O}_{\mathbb{P}^4})\to H^0(\mathcal{O}_{Z_a})\to H^1(\mathscr{I}_{Z_a})\to\cdots$$
since $h^i(\mathscr{I}_{Z_a})=0$ for $i=0,1$ hence $h^0(\mathcal{O}_{\mathbb{P}^4})=h^0(\mathcal{O}_{Z_a})=1$. Therefore, $Z_a$ is connected.

\end{proof}

Now, we will bound the degree of a hypersurface which can contain $\Sing(\mathscr{F}_a)$. 
\begin{propo}
$Z_a=\Sing(\mathscr{F}_a)$ is  never contained in a hypersurface of degree  $\leq 2a+5= \deg(\mathscr{F}_a)$.
\end{propo}
\begin{proof}
Consider the complex associated with a morphism $\varphi^\vee:T\mathbb{P}^4\rightarrow E(a+2)$ twisted by $\mathcal{O}_{\mathbb{P}^4}(d)$:
    \begin{equation}\label{ENcotangFamd}
        0\rightarrow S_2(E(-a-7))(d-2a-4)\rightarrow\Omega_{\mathbb{P}^4}^1\otimes E(d-3a-11)\rightarrow\Omega_{\mathbb{P}^4}^2(d-2a-4)\rightarrow \mathscr{I}_{Z_a}(d)\rightarrow 0
    \end{equation}
    with the respective    short exact sequences
$$
    0\rightarrow S_2(E(-a-7))(d-2a-4)\rightarrow\Omega_{\mathbb{P}^4}^{1}\otimes E(d-3a-11)\rightarrow K(d)\rightarrow 0,
$$
and
$$
    0\rightarrow K(d)\rightarrow\Omega_{\mathbb{P}^4}^{2}(d-2a-4)\rightarrow\mathscr{I}_{Z_a}(d)\rightarrow 0.
$$
Passing to cohomology we get that
$$\cdots\to H^0(\Omega_{\mathbb{P}^4}^2(d-2a-4))\to H^0(\mathscr{I}_{Z_a}(d))\to H^{1}(K(d))\to H^{1}(\Omega_{\mathbb{P}^4}^2(d-2a-4))\to\cdots$$
and 
$$\cdots\to
H^1(\Omega_{\mathbb{P}^4}^1\otimes E(d-3a-11)) \to H^1(K(d)) \to H^2(S_2(E(d-a-7))(d-2a-4))
\to\cdots.$$
On the one hand,  by Bott's formula, we have that $H^0(\Omega_{\mathbb{P}^4}^2(d-2a-4))=H^{1}(\Omega_{\mathbb{P}^4}^2(d-2a-4))=0$, if $d\leq 2a+5$. On the other hand, by Lemma \ref{Vlema1} we obtain that   $$H^1(\Omega_{\mathbb{P}^4}^1\otimes E(d-3a-11))= H^2(S_2(E(d-a-7))(d-2a-4))=0$$ if $d\leq 3a+6$. Then,
$$
H^0(\mathscr{I}_{Z_a}(d))\simeq  H^{1}(K(d))=0
$$
for all $d\leq 2a+5=\deg(\mathscr{F}_a)$. 
\end{proof}

\subsection{The Rao nodule of singular scheme.}
In  \cite{MSM} the authors have introduced a new invariant, called Rao module, which appears in the classification of locally complete intersection  codimension two distributions on $\mathbb{P}^3$.
In this section, Next, we will determine the  dimension of  Rao Module  of  Horrocks-Mumford distributions.

Let $\mathscr{F}$ be a holomorphic distribution on $\mathbb{P}^4$. Consider the following graded module
$$R_{\mathscr{F}}:=H^1_*(\mathscr{I}_Z)=\bigoplus_{l\in\mathbb{Z}}H^1(\mathscr{I}_Z(l));$$
called the \textit{Rao module} and
$$M_{\mathscr{F}}:=H^1_*(N {\mathscr{F}}^*)=\bigoplus_{l\in\mathbb{Z}}H^1(N {\mathscr{F}}^*(l));$$
called the \textit{first cohomology module} of $\mathscr{F}$ as a sub-sheaf of the conormal. 
Since $\mathscr{F}$ is locally free then $R_{\mathscr{F}}$ is finite-dimensional and $M_{\mathscr{F}}$ is always finite-dimensional.

\subsubsection{Rao Module dimension:}
Let $\mathscr{F}_a$ is a Horrocks-Mumford distribution
\begin{equation*}
    \mathscr{F}_a: 0\rightarrow E(-a-7)\xrightarrow{\varphi}\Omega_{\mathbb{P}^4}^1\rightarrow \mathcal{Q}_{\mathscr{F}_a}\rightarrow 0.
\end{equation*}
Considering the Eagon-Northcott complex associated to the morphism $\varphi^\vee:T\mathbb{P}^4\rightarrow E(a+2)$ and twisting by $\mathcal{O}_{\mathbb{P}^4}(q)$ we have:
\begin{equation}\label{ENdistCotang}
    0\rightarrow S_2(E(-a-7))(q-2a-4)\rightarrow\Omega_{\mathbb{P}^4}^1\otimes E(q-3a-11)\rightarrow\Omega_{\mathbb{P}^4}^2(q-2a-4)\rightarrow\mathscr{I}_{Z_a}(q)\rightarrow 0.
\end{equation}
In order to calculate $h^1(\mathscr{I}_{Z_a}(q))$ in (\ref{ENdistCotang}), for all $q\in\mathbb{Z}$ and $a\geq 1$, we  will prove some technical  Lemmas:

\begin{lema}\label{h1EEn02}
For all $q\in\mathbb{Z}$ and $a\geq 1$ we have:
\begin{enumerate}
    \item  If $\{q\neq 2a+4\}\cap\{q\leq 3a+2 \,\,\ \textrm{or} \,\,\ q\geq 3a+9\}\cap\{4a+2<q<4a+10\}$ then  $$h^1(\mathscr{I}_{Z_a}(q))=h^3(E\otimes E(q-4a-18))\neq 0.$$
    \item If $\{q\neq 2a+4\}\cap\{q\leq 3a+2 \,\,\ \textrm{or} \,\,\ q\geq 3a+9\}\cap\{q\leq 4a+2 \,\,\ \textrm{or} \,\,\ q\geq 4a+10\}$ then $$h^1(\mathscr{I}_{Z_a}(q))=h^3(E\otimes E(q-4a-18))=0.$$
\end{enumerate}
\end{lema}

\begin{proof}
We break the complex (\ref{ENdistCotang}) into exact sequences:
\begin{equation}\label{sqtdtang1Fam}
    0\rightarrow S_2(E(-a-7))(q-2a-4)\rightarrow\Omega_{\mathbb{P}^4}^{1}\otimes E(q-3a-11)\rightarrow K(q)\rightarrow 0,
\end{equation}
and
\begin{equation}\label{sqtdcotang2Fam}
    0\rightarrow K(q)\rightarrow\Omega_{\mathbb{P}^4}^{2}(q-2a-4)\rightarrow\mathscr{I}_{Z_a}(q)\rightarrow 0.
\end{equation}
Passing to cohomology we have the exact sequence (\ref{sqtdcotang2Fam}):
\begin{equation*}
    \cdots\rightarrow H^1(\Omega_{\mathbb{P}^4}^{2}(q-2a-4))\rightarrow H^1(\mathscr{I}_{Z_a}(q))\rightarrow H^2(K(q))\rightarrow H^2(\Omega_{\mathbb{P}^4}^{2}(q-2a-4))\rightarrow\cdots,
\end{equation*}
and by Bott's formula we have that  $h^1(\Omega_{\mathbb{P}^4}^{2}(q-2a-4))=0=h^2(\Omega_{\mathbb{P}^4}^{2}(q-2a-4))$ for $q\neq 2a+4$, then 
\begin{equation}\label{h1h2K2}
  h^1(\mathscr{I}_{Z_a}(q))=h^2(K(q)) \,\,\,\ \textrm{for} \,\,\,\ q\neq 2a+4 \,\,\ \textrm{and} \,\,\ a\geq 1.  
\end{equation}
Now, from
$$\cdots\to H^2(\Omega_{\mathbb{P}^4}^1\otimes E(q-3a-11))\to H^2(K(q))\to$$  $$\to H^3(S_2(E(-a-7))(q-2a-4))\to H^3(\Omega_{\mathbb{P}^4}^3\otimes E(q-3a-11))\to\cdots$$
and by Lemma \ref{Vlema1} we obtain 
$$h^2(K(q))=h^3(S_2(E(-a-7))(q-2a-4))$$
for all $q\in\mathbb{Z}$ such that $\{q\leq 3a+2 \,\,\ \textrm{or} \,\,\ q\geq 3a+9\}.$
Thus $$h^1(\mathscr{I}_{Z_a}(q))=h^3(S_2(E(-a-7))(q-2a-4))$$ for all $q\in\mathbb{Z}$ such that $\{q\neq 2a+4\}\cap\{q\leq 3a+2 \,\,\ \textrm{or} \,\,\ q\geq 3a+9\}$. By the decomposition of the  tensor product and Bott's formula, we have:
$$h^3(E\otimes E(q-4a-18))= h^3(S_2(E(-a-7))(q-2a-4)).$$
So $h^1(\mathscr{I}_{Z_a}(q))=h^3(E\otimes E(q-4a-18))$ for all $q\in\mathbb{Z}$ such that $\{q\neq 2a+4\}\cap\{q\leq 3a+2 \,\,\ \textrm{or} \,\,\ q\geq 3a+9\}.$
It follows form the  table \ref{table2} that $h^3(E\otimes E(q-4a-18)\neq 0$ for all $q\in\mathbb{Z}$ such that $4a+2<q<4a+10$. Thus
\begin{equation*}
    h^1(\mathscr{I}_{Z_a}(q))=h^3(E\otimes E(q-4a-18))\neq 0
\end{equation*}
for all $q\in\mathbb{Z}$ and $a\geq 1$ such that $$\{q\neq 2a+4\}\cap\{q\leq 3a+2 \,\,\ \textrm{or} \,\,\ q\geq 3a+9\}\cap\{4a+2<q<4a+10\}.$$
On the other hand, by table $\ref{table2}$ we have $h^3(E\otimes E(q-4a-18))=0$ for $q\in\mathbb{Z}$ such that $$\{q\leq 4a+2 \,\,\ \textrm{or} \,\,\ q\geq 4a+10\}.$$
Therefore, $h^1(\mathscr{I}_{Z_a}(q))=h^3(E\otimes E(q-4a-18))=0$ for all $q\in\mathbb{Z}$ such that $$\{q\neq 2a+4\}\cap\{q\leq 3a+2 \,\,\ \textrm{or} \,\,\ q\geq 3a+9\}\cap\{q\leq 4a+2 \,\,\ \textrm{or} \,\,\ q\geq 4a+10\}.$$
\end{proof}

\begin{lema}\label{h1C202}
For all $q\in\mathbb{Z}$ and $a\geq 1$   such that
$$\{q\leq 4a+2 \,\,\ \textrm{or} \,\,\ q\geq 4a+10\}\cap\{q\leq 3a+6 \,\,\ \textrm{or} \,\,\ q\geq 3a+13\}$$
we have that 
\begin{equation*} 
h^1(\mathscr{I}_{Z_a}(q))=h^1(\Omega_{\mathbb{P}^4}^2(q-2a-4))=0.  
\end{equation*}
\end{lema}

\begin{proof}
From
$$\cdots\to H^1(K(q))\to H^1(\Omega_{\mathbb{P}^4}^2(q-2a-4))\to H^1(\mathscr{I}_{Z_a}(q))\to H^2(K(q))\to\cdots$$
we have
$$h^1(\mathscr{I}_{Z_a}(q))=h^1(\Omega_{\mathbb{P}^4}^2(q-2a-4))$$ if and only if $h^1(K(q))=0=h^2(K(q))$, and from 
$$\cdots\to H^i(S_2(E(-a-7))(q-2a-4))\to H^i(\Omega_{\mathbb{P}^4}^{1}\otimes E(q-3a-11))\to H^i(K(q))\to\cdots$$
we have $$h^1(K(q))=0=h^2(K(q))$$ if and only if  $h^i(S_2(E(-a-7))(q-2a-4))=0$ for all $i=1,2,3$ and  $$h^i(\Omega_{\mathbb{P}^4}^1\otimes E(q-3a-11))=0$$ for all $i=1,2$.
It follows from table \ref{table2} that $$h^i(S_2(E(-a-7))(q-2a-4))=h^i(E\otimes E(q-4a-18))=0$$ for all $i=1,2,3$ and $q\in\mathbb{Z}$ such that  $\{q\leq 4a+2 \,\,\ \textrm{or} \,\,\ q\geq 4a+10\}$.
 By Lemma \ref{Vlema1} we have that $h^i(\Omega_{\mathbb{P}^4}^1\otimes E(q-3a-11))=0$ for all $i=1,2$ and $q\in\mathbb{Z}$ such that $\{q\leq 3a+6 \,\,\ \textrm{or} \,\,\ q\geq 3a+13\}$. 
Thus $h^1(\mathscr{I}_{Z_a}(q))=h^1(\Omega_{\mathbb{P}^4}^2(q-2a-4))$ for all $q\in\mathbb{Z}$ such that $$\{q\leq 4a+2 \,\,\ \textrm{or} \,\,\ q\geq 4a+10]\cap[q\leq 3a+6 \,\,\ \textrm{or} \,\,\ q\geq 3a+13\}.$$ 

Since Bott's formula we have $h^1(\Omega_{\mathbb{P}^4}^2(q-2a-4))=0$, then $h^1(\mathscr{I}_{Z_a}(q))=0$ for all $q\in\mathbb{Z}$ and $a\geq 1$ such that $\{q\leq 4a+2 \,\,\ \textrm{or} \,\,\ q\geq 4a+10\}\cap\{q\leq 3a+6 \,\,\ \textrm{or} \,\,\ q\geq 3a+13\}.$
\end{proof}

\begin{lema}\label{h1+2}
For all $q\in\mathbb{Z}$ and $a\geq 1$ such that $$\{q\neq 2a+4\}\cap\{q\leq 3a+2 \,\,\ \textrm{or} \,\,\ q\geq 3a+7\}\cap\{q\leq 4a+7 \,\,\ \textrm{or} \,\,\ q\geq 4a+14\}$$
we have that 
\begin{equation*}
    h^1(\mathscr{I}_{Z_a}(q))=h^2(\Omega_{\mathbb{P}^4}^1\otimes E(q-3a-11))+h^3(E\otimes E(q-4a-18)).
\end{equation*}

\end{lema}

\begin{proof}

By equation (\ref{h1h2K2}) we have $h^1(\mathscr{I}_{Z_a}(q))=h^2(K(q)) \,\,\,\ \textrm{for} \,\,\,\ q\neq 2a+4$.
Now, from
$$\cdots\to H^i(S_2(E(-a-7))(q-2a-4))\to H^i(\Omega_{\mathbb{P}^4}^1\otimes E(q-3a-11))\to H^i(K(q))\to\cdots$$
We note that $$h^2(K(q))=h^2(\Omega_{\mathbb{P}^4}^1\otimes E(q-3a-11))+h^3(S_2(E(-a-7)(q-2a-4)))$$ 
if and only if $h^2(S_2(E(-a-7))(q-2a-4))=0$ and $h^3(\Omega_{\mathbb{P}^4}^1\otimes E(q-3a-11))=0$.

On the one hand, by table \ref{table2} we have  that $$h^2(S_2(E(-a-7)(q-2a-4)))=h^2(E\otimes E(q-4a-18))=0$$ if and only if $q\leq 4a+7$ or $q\geq 4a+14$.

On the other hand, by Lemma \ref{Vlema1} we have $$h^3(\Omega_{\mathbb{P}^4}^1\otimes E(q-3a-11))=0$$ if and only if $q\leq 3a+2$ or $q\geq 3a+7$.
Thus, putting these last equations together, we have
$$h^1(\mathscr{I}_{Z_a}(q))=h^2(\Omega_{\mathbb{P}^4}^1\otimes E(q-3a-11))+h^3(S_2(E(-a-7)(q-2a-4)))$$
for all $q\in\mathbb{Z}$ and $a\geq 1$ such that $$\{q\neq 2a+4\}\cap\{q\leq 3a+2 \,\,\ \textrm{or} \,\,\ q\geq 3a+7\}\cap\{q\leq 4a+7 \,\,\ \textrm{or} \,\,\ q\geq 4a+14\}.$$
\end{proof}

\begin{theorem}
Let $\mathscr{F}_a$ is a dimension $2$ holomorphic distributions (\ref{distCotangFam}) on $\mathbb{P}^4$.  Then:
\begin{enumerate}
    \item $\dim_{\mathbb{C}} R_{\mathscr{F}_1}\geq 184$.
    \item $\dim_{\mathbb{C}} R_{\mathscr{F}_2}\geq 284$.
    \item $\dim_{\mathbb{C}} R_{\mathscr{F}_3}\geq 369$.
    \item $\dim_{\mathbb{C}} R_{\mathscr{F}_a}=401$,  for all $ a\geq4$.
\end{enumerate}

\end{theorem}

\begin{proof}

For $a=1$:

By Lemmas \ref{h1EEn02} and \ref{h1C202} we have $h^1(\mathscr{I}_{Z_1}(q))=0$ for all $q\in\mathbb{Z}$ such that $q\leq 6$ or $q\geq 14$.
 Lemma \ref{h1+2} implies that  $h^1(\mathscr{I}_{Z_1}(q))=h^2(\Omega_{\mathbb{P}^4}^1\otimes E(q-14))+h^3(E\otimes E(q-22))$ for $q= 10, 11$. Hence 
\begin{equation*}
\begin{split}
h^1(\mathscr{I}_{Z_1}(10)) & = h^2(\Omega_{\mathbb{P}^4}^1\otimes E(-4))+h^3(E\otimes E(-12)).
\end{split}
\end{equation*}

Now, by Lemma \ref{Vlema1} we have $h^i(\Omega_{\mathbb{P}}^1\otimes E(-4))=0$ for all $i=0,1,3,4$, thus 
$$\chi (\Omega_{\mathbb{P}^4}^1\otimes E(-4))=h^2(\Omega_{\mathbb{P}^4}^1\otimes E(-4)).$$ 
Since $c(\Omega_{\mathbb{P}^4}^1\otimes E(-4))=1 - 22\textbf{h} + 228\textbf{h}^2  - 1434\textbf{h}^3  + 5955\textbf{h}^4$ then by Hirzebruch-Riemann-Roch Theorem $\chi (\Omega_{\mathbb{P}^4}^1\otimes E(-4))=10$ we obtain that  $h^2(\Omega_{\mathbb{P}^4}^1\otimes E(-4))=10.$
By table \ref{table2} we have $h^3(E\otimes E(-12))=85$. Therefore $h^1(\mathscr{I}_{Z_1}(10))=95$. By using the same argument as  above, we get $h^1(\mathscr{I}_{Z_1}(11))=60$.
On the other hand, by Lemma \ref{h1EEn02} we have $h^1(\mathscr{I}_{Z_1}(q))=h^3(E\otimes E(q-22))\neq 0$ for all $q\in\mathbb{Z}$ such that $12\leq q<14$. So, by table \ref{table2} we have $h^1(\mathscr{I}_{Z_1}(12))=24$ and $h^1(\mathscr{I}_{Z_1}(13))=5$. Therefore
\begin{equation*}
\begin{split}
\dim_{\mathbb{C}} R_{\mathscr{F}_1} & = \sum_{l=7}^{13}h^1(\mathscr{I}_{Z_1}(l))\\
 & = h^1(\mathscr{I}_{Z_1}(7)) + h^1(\mathscr{I}_{Z_1}(8)) + h^1(\mathscr{I}_{Z_1}(9)) + 184.
\end{split}
\end{equation*}

For $a=2$:

By Lemmas \ref{h1EEn02} and \ref{h1C202} we have that $h^1(\mathscr{I}_{Z_2}(q))=0$ for all $q\in\mathbb{Z}$ such that either  $q\leq 10$ or $q\geq 18$.
By Lemma \ref{h1+2} we have $$h^1(\mathscr{I}_{Z_2}(q))=h^2(\Omega_{\mathbb{P}^4}^1\otimes E(q-17))+h^3(E\otimes E(q-26))$$ for $q=13, 14$. Hence by Lemma \ref{Vlema1} and table \ref{table2} we have $h^1(\mathscr{I}_{Z_2}(13))=110$. Similarly, we get $h^1(\mathscr{I}_{Z_2}(14))=90$.
In addition, by Lemma \ref{h1EEn02} we have $h^1(\mathscr{I}_{Z_2}(q))=h^3(E\otimes E(q-26))\neq 0$ for all $q\in\mathbb{Z}$ such that $15\leq q<18$. So, by table \ref{table2} we have $h^1(\mathscr{I}_{Z_2}(15))=55$, $h^1(\mathscr{I}_{Z_2}(16))=24$ and $h^1(\mathscr{I}_{Z_2}(17))=5$. Therefore
\begin{equation*}
\begin{split}
\dim_{\mathbb{C}} R_{\mathscr{F}_2} & = \sum_{l=11}^{17}h^1(\mathscr{I}_{Z_2}(l))\\
 & = h^1(\mathscr{I}_{Z_2}(11)) + h^1(\mathscr{I}_{Z_2}(12)) + 284.
\end{split}
\end{equation*}

For $a=3$:

By Lemma \ref{h1C202} we have $h^1(\mathscr{I}_{Z_3}(q))=0$ for all $q\in\mathbb{Z}$ such that $q\leq 14$ or $q\geq 22$.
By Lemma \ref{h1+2} we have $h^1(\mathscr{I}_{Z_3}(q))=h^2(\Omega_{\mathbb{P}^4}^1\otimes E(q-20))+h^3(E\otimes E(q-30))$ for $q=16, 17$. Hence $h^1(\mathscr{I}_{Z_3}(16))=95$ and $h^1(\mathscr{I}_{Z_3}(17))=110$.
 It follows from  Lemma \ref{h1EEn02}  that  $h^1(\mathscr{I}_{Z_3}(q))=h^3(E\otimes E(q-30))\neq 0$ for all $q\in\mathbb{Z}$ such that $18\leq q<22$. So, by the table \ref{table2} we get that  $h^1(\mathscr{I}_{Z_3}(18))=85$, $h^1(\mathscr{I}_{Z_3}(19))=55$, $h^1(\mathscr{I}_{Z_3}(20))=24$ and $h^1(\mathscr{I}_{Z_3}(21))=5$. Therefore
\begin{equation*}
\begin{split}
\dim_{\mathbb{C}} R_{\mathscr{F}_3} & = \sum_{l=15}^{21}h^1(\mathscr{I}_{Z_3}(l))\\
 & = h^1(\mathscr{I}_{Z_3}(15)) + 369.
\end{split}
\end{equation*}

For $a=4$:

By Lemma \ref{h1C202} we have $h^1(\mathscr{I}_{Z_4}(q))=0$ for all $q\in\mathbb{Z}$ such that $q\leq 18$ or $q\geq 26$.
By Lemma \ref{h1+2} we have $h^1(\mathscr{I}_{Z_4}(q))=h^2(\Omega_{\mathbb{P}^4}^1\otimes E(q-23))+h^3(E\otimes E(q-34))$ for $q=19,20$. Hence $h^1(\mathscr{I}_{Z_4}(19))=42$ and $h^1(\mathscr{I}_{Z_4}(20))=90$.
 Lemma \ref{h1EEn02}  says us that  $h^1(\mathscr{I}_{Z_4}(q))=h^3(E\otimes E(q-34))\neq 0$ for all $q\in\mathbb{Z}$ such that $21\leq q<26$. So, by table \ref{table2} we have: $h^1(\mathscr{I}_{Z_4}(21))=100$, $h^1(\mathscr{I}_{Z_4}(22))=85$, $h^1(\mathscr{I}_{Z_4}(23))=55$, $h^1(\mathscr{I}_{Z_4}(24))=24$ and $h^1(\mathscr{I}_{Z_4}(25))=5$. Therefore
\begin{equation*}
\begin{split}
\dim_{\mathbb{C}} R_{\mathscr{F}_4} & = \sum_{l=19}^{25}h^1(\mathscr{I}_{Z_4}(l))\\
 & = 401.
\end{split}
\end{equation*}

For $a=5$:

By Lemma \ref{h1C202} we have $h^1(\mathscr{I}_{Z_5}(q))=0$ for all $q\in\mathbb{Z}$ such that $q\leq 21$ or $q\geq 30$.
It follows from  Lemma \ref{h1+2}  that  $h^1(\mathscr{I}_{Z_5}(q))=h^2(\Omega_{\mathbb{P}^4}^1\otimes E(q-26))+h^3(E\otimes E(q-38))$ for $q=22,23$. Hence $h^1(\mathscr{I}_{Z_5}(22))=10$ and $h^1(\mathscr{I}_{Z_5}(23))=37$.
By Lemma \ref{h1EEn02} we have $h^1(\mathscr{I}_{Z_5}(q))=h^3(E\otimes E(q-38))\neq 0$ for all $q\in\mathbb{Z}$ such that $24\leq q<30$. So, by table \ref{table2} we have: $h^1(\mathscr{I}_{Z_5}(24))=85$, $h^1(\mathscr{I}_{Z_5}(25))=100$, $h^1(\mathscr{I}_{Z_5}(26))=85$, $h^1(\mathscr{I}_{Z_5}(27))=55$, $h^1(\mathscr{I}_{Z_5}(28))=24$ and $h^1(\mathscr{I}_{Z_5}(29))=5$. Therefore
\begin{equation*}
\begin{split}
\dim_{\mathbb{C}} R_{\mathscr{F}_5} & = \sum_{l=22}^{29}h^1(\mathscr{I}_{Z_5}(l))\\
 & = 401.
\end{split}
\end{equation*}

For $a=6$:

By Lemma \ref{h1C202} we have $h^1(\mathscr{I}_{Z_6}(q))=0$ for all $q\in\mathbb{Z}$ such that $q\leq 24$ or $q\geq 34$.
By Lemma \ref{h1+2} we have $h^1(\mathscr{I}_{Z_6}(q))=h^2(\Omega_{\mathbb{P}^4}^1\otimes E(q-29))$ for $q=25,26$. Hence $h^1(\mathscr{I}_{Z_6}(25))=10$ and $h^1(\mathscr{I}_{Z_6}(26))=5$.
By Lemma \ref{h1EEn02} we have $h^1(\mathscr{I}_{Z_6}(q))=h^3(E\otimes E(q-42))\neq 0$ for all $q\in\mathbb{Z}$ such that $27\leq q<34$. So, by table \ref{table2} we have: $h^1(\mathscr{I}_{Z_6}(27))=32$, $h^1(\mathscr{I}_{Z_6}(28))=85$, $h^1(\mathscr{I}_{Z_6}(29))=100$, $h^1(\mathscr{I}_{Z_6}(30))=85$, $h^1(\mathscr{I}_{Z_6}(31))=55$, $h^1(\mathscr{I}_{Z_6}(32))=24$ and $h^1(\mathscr{I}_{Z_6}(33))=5$. Therefore
\begin{equation*}
\begin{split}
\dim_{\mathbb{C}} R_{\mathscr{F}_6} & = \sum_{l=25}^{33}h^1(\mathscr{I}_{Z_6}(l))\\
 & = 401.
\end{split}
\end{equation*}

For $a=7$:

By Lemmas \ref{h1EEn02} and \ref{h1C202} we have $h^1(\mathscr{I}_{Z_7}(q))=0$ for all $q\in\mathbb{Z}$ such that $q\leq 27$ or $q=30$ or $q\geq 38$.
By Lemma \ref{h1+2} we have $h^1(\mathscr{I}_{Z_7}(q))=h^2(\Omega_{\mathbb{P}^4}^1\otimes E(q-32))$ for $q=28, 29$. Hence $h^1(\mathscr{I}_{Z_7}(28))=10$ and $h^1(\mathscr{I}_{Z_7}(29))=5$.
By Lemma \ref{h1EEn02} we have $h^1(\mathscr{I}_{Z_7}(q))=h^3(E\otimes E(q-46))\neq 0$ for all $q\in\mathbb{Z}$ such that $30<q<38$. So, by table \ref{table2} we have: $h^1(\mathscr{I}_{Z_7}(31))=32$, $h^1(\mathscr{I}_{Z_7}(32))=85$, $h^1(\mathscr{I}_{Z_7}(33))=100$, $h^1(\mathscr{I}_{Z_7}(34))=85$, $h^1(\mathscr{I}_{Z_7}(35))=55$, $h^1(\mathscr{I}_{Z_7}(36))=24$ and $h^1(\mathscr{I}_{Z_7}(37))=5$. Therefore
\begin{equation*}
\begin{split}
\dim_{\mathbb{C}} R_{\mathscr{F}_7} & = \sum_{l=28}^{37}h^1(\mathscr{I}_{Z_7}(l))\\
 & = 401.
\end{split}
\end{equation*}

For $a\geq 7$:

By Lemmas \ref{h1EEn02} and \ref{h1C202} we have $h^1(\mathscr{I}_{Z_a}(q))=0$ for all $q\in\mathbb{Z}$ such that $q\leq 3a+6$ or $3a+9\leq q\leq 4a+2$ or $q\geq 4a+10$.
By Lemma \ref{h1+2} we have $h^1(\mathscr{I}_{Z_a}(q))=h^2(\Omega_{\mathbb{P}^4}^1\otimes E(q-3a-11))$ for all $q\in\mathbb{Z}$ such that $\{q\neq 2a+4\}\cap\{q\leq 4a+2 \,\,\ \textrm{or} \,\,\ q\geq 4a+14\}$, in particular for $q=3a+7, 3a+8$. Hence $h^1(\mathscr{I}_{Z_a}(3a+7))=10$ and $h^1(\mathscr{I}_{Z_a}(3a+8))=5$.
By Lemma \ref{h1EEn02} we have $h^1(\mathscr{I}_{Z_a}(q))=h^3(E\otimes E(q-4a-18))\neq 0$ for all $q\in\mathbb{Z}$ such that $4a+2<q<4a+10$. So, by table \ref{table2} we have:

\noindent $h^1(\mathscr{I}_{Z_a}(4a+3))=32$, $h^1(\mathscr{I}_{Z_a}(4a+4))=85$, $h^1(\mathscr{I}_{Z_a}(4a+5))=100$,
\\
\noindent $h^1(\mathscr{I}_{Z_a}(4a+6))=85$, $h^1(\mathscr{I}_{Z_a}(4a+7))=55$, $h^1(\mathscr{I}_{Z_a}(4a+8))=24$  \\
\noindent and  $h^1(\mathscr{I}_{Z_a}(4a+9))=5$. Therefore
\begin{equation*}
\begin{split}
\dim_{\mathbb{C}} R_{\mathscr{F}_a} & = \sum_{l=3a+7}^{4a+9}h^1(\mathscr{I}_{Z_a}(l))\\
 & = \sum_{l=3a+7}^{3a+8}h^1(\mathscr{I}_{Z_a}(l))+0+\cdots + 0 + \sum_{l=4a+3}^{4a+9}h^1(\mathscr{I}_{Z_a}(l))\\
 & = 401.
\end{split}
\end{equation*}

\end{proof}

\subsection{Horrocks--Mumford distributions  are  determined by their  singular schemes}
We prove  that a Horrocks--Mumford distributions are determinate by its singular scheme.

\begin{theorem} \label{Determina}
 If   $\mathscr{F}'$ is a dimension two distribution on  $\mathbb{P}^4$, with   degree  $2a+5$, such that $\Sing(\mathscr{F}_a) \subset \Sing(\mathscr{F}') $, then $ \mathscr{F}'=\mathscr{F}_a  $.
\end{theorem}

\begin{proof}
We consider the dual map 
$\varphi^*:T\mathbb{P}^4\to E(a+2), \,\ a\geq 1.$
By \cite{CM}, the distribution $ \mathscr{F}_a$  will be determined by their singular scheme if 
$$H^i(\wedge^2\Omega_{\mathbb{P}^4}\otimes\wedge^{2+i}T\mathbb{P}^4\otimes S_i(E(-5)(-a-2))=0, \,\ i=1,2.$$

For $i=1$:
\\[10pt]
$H^1(\Omega_{\mathbb{P}^4}^2\otimes\wedge^3 T\mathbb{P}^4\otimes E(-a-7))$, and since $\wedge^3 T\mathbb{P}^4=\Omega_{\mathbb{P}^4}^1(5)$, then $$H^1(\Omega_{\mathbb{P}^4}^2\otimes\wedge^3 T\mathbb{P}^4\otimes E(-a-7))=H^1(\Omega_{\mathbb{P}^4}^2\otimes\Omega_{\mathbb{P}^4}^1\otimes E(-a-2)).$$
Twisting the Euler's sequence by $ \Omega_{\mathbb{P}^4}^2\otimes E(-a-3)$, we have 
$$0\to\Omega_{\mathbb{P}^4}^2\otimes\Omega_{\mathbb{P}^4}^1\otimes E(-a-2)\to[\Omega_{\mathbb{P}^4}^2\otimes E(-a-3)]^{\oplus 5}\to\Omega_{\mathbb{P}^4}^2\otimes E(-a-2)\to 0$$
Taking Cohomology and by Vanishing Lemma \ref{Vlema1}, we  obtain  $$H^1(\Omega_{\mathbb{P}^4}^2\otimes\Omega_{\mathbb{P}^4}^1\otimes E(-a-2))=H^0(\Omega_{\mathbb{P}^4}^2\otimes E(-a-2)), \,\ \forall \  a\geq -1.$$
But, by Vanishing Lemma \ref{Vlema1}, $H^0(\Omega_{\mathbb{P}^4}^2\otimes E(-a-2))=0$ for all  $a\geq -4$. Therefore $$H^1(\Omega_{\mathbb{P}^4}^2\otimes\Omega_{\mathbb{P}^4}^1\otimes E(-a-2))=0, \,\ \forall \  \ a\geq -1.$$

For $i=2$: we have that 
\\[10pt]
$H^2(\wedge^2\Omega_{\mathbb{P}^4}\otimes\wedge^{4}T\mathbb{P}^4\otimes S_2(E(-5)(-a-2))=H^2(\Omega_{\mathbb{P}^4}^2(5)\otimes S_2(E(-a-7)))$.
Twisting the Euler's sequence by $S_2(E(-a-7))(3)$  $$0\to\Omega_{\mathbb{P}^4}^2(2)\to\mathcal{O}_{\mathbb{P}^4}^{\oplus 10}\to\Omega_{\mathbb{P}^4}^1(2)\to 0$$ we have the exact sequence  
$$0\to\Omega_{\mathbb{P}^4}^2(5)\otimes S_2(E(-a-7))\to[S_2(E(-a-7))(3)]^{\oplus 10}\to\Omega_{\mathbb{P}^4}^1(5)\otimes S_2(E(-a-7))\to 0$$
Since $V\otimes V=S_2(V)\oplus\wedge^2V$, then taking $V=E(-a-7)$ and twisting by $ \mathcal{O}_{\mathbb{P}^4}(3)$, we have $$E\otimes E(-2a-11)=S_2(E(-a-7))(3)\oplus\mathcal{O}_{\mathbb{P}^4}(-2a-6).$$  By Bott's formula  $$H^1(E\otimes E(-2a-11))=H^1(S_2(E(-a-7))(3))$$
From the  Vanishing Lemma \ref{Vlema3} we conclude that  $h^1(E\otimes E(-2a-11))=0$ for all  $2a\geq -4$ or $2a\leq -12$. Thus $$H^1(S_2(E(-a-7))(3))=0 \,\,\ \text{ for all  } 2a\geq -4 \text{ or } 2a\leq -12 $$
Analogously, by Bott formula $H^2(E\otimes E(-2a-11))=H^2(S_2(E(-a-7))(3))$, and by Vanishing Lemma $h^2(S_2(E(-a-7))(3))=0$ for all  $2a\geq 0$ or $2a\leq -7$. Thus $$h^1(S_2(E(-a-7))(3))=0=h^2(S_2(E(-a-7))(3))$$ for all $2a\leq -12$ or $2a\geq 0$. Therefore $$H^2(\Omega_{\mathbb{P}^4}^2(5)\otimes S_2(E(-a-7)))=H^1(\Omega_{\mathbb{P}^4}^1(5)\otimes S_2(E(-a-7)))$$ for all  $2a\leq -12$ or $2a\geq 0$.
Now, twisting the Euler's sequence by $S_2(E(-a-7))(4)$, we have:
$$0\to\Omega_{\mathbb{P}^4}^1(5)\otimes S_2(E(-a-7))\to[S_2(E(-a-7))(4)]^{\oplus 5}\to S_2(E(-a-7))(5)\to 0.$$
Since, for all $a\geq 0$ we have   $$E\otimes E(-2a-10)=S_2(E(-a-7))(4)\oplus\mathcal{O}_{\mathbb{P}^4}(-2a-5).$$ 
Twisting by $\mathcal{O}_{\mathbb{P}^4}(4)$, taking cohomology and applying   Bott's formula, we obtain  $$H^0(E\otimes E(-2a-10))=H^0(S_2(E(-a-7))(4)).$$ By Vanishing Lemma we obtain that  $h^0(S_2(E(-a-7))(4))=0$ for all $2a\geq -4$. Similarly,  we  have  $H^1(E\otimes E(-2a-10))=H^1(S_2(E(-a-7))(4))$. Vanishing Lemma gives us that  $h^1(S_2(E(-a-7))(4))=0$ for all  $2a\geq -3$ or $2a\leq -11$. Thus $$h^0(S_2(E(-a-7))(4))=0=h^1(S_2(E(-a-7))(4))$$ for $a\geq 0$. Therefore $$H^1(\Omega_{\mathbb{P}^4}^1(5)\otimes S_2(E(-a-7)))=H^0(S_2(E(-a-7)(5)) , \text{ for } 2a\geq -3.$$
Similarly,  we can also conclude that  $$E\otimes E(-2a-14)=S_2(E(-a-7))\oplus\mathcal{O}_{\mathbb{P}^4}(-2a-9).$$ Twisting by $\mathcal{O}_{\mathbb{P}^4}(5)$ and by Bott's formula we have $$H^0(E\otimes E(-2a-9))=H^0(S_2(E(-a-7))(5)),$$ but by Vanishing Lemma $h^0(E\otimes E(-2a-9))=0$ for $2a\geq -3$, then $$h^1(\Omega_{\mathbb{P}^4}^1\otimes S_2(E(-a-7))(5))=0$$ for all  $a\geq 0$.
Since $H^2(\Omega_{\mathbb{P}^4}^2(5)\otimes S_2(E(-a-7)))=H^1(\Omega_{\mathbb{P}^4}^1(5)\otimes S_2(E(-a-7)))$ for all  $2a\leq -12$ or $2a\geq 0$. Then $$H^2(\Omega_{\mathbb{P}^4}^2(5)\otimes S_2(E(-a-7)))=0,$$
for all  $a\geq 0$.
\end{proof}

\subsection{Automorphisms of the Horrocks-Mumford distributions }

 \begin{propo}\label{Aut-HM}
The distribution $\mathscr{F}_a$ is invariant by a group  $\Gamma_{1,5}\simeq  H_5 \rtimes SL(2,\mathbb{Z}_5) \subset Sp(4, \mathbb{Q})$, where $H_5$ is the Heisenberg group of level $5$ generated by 
\begin{center}
   $\sigma: z_k \to z_{k-1}$  and    $\tau: z_k \to  \epsilon^{-k} z_{k}$, with  $k\in \mathbb{Z}_5$ and $\epsilon= e^{\frac{2\pi i}{5}}$.  
\end{center}
\end{propo}
\begin{proof}
The Horrocks--Mumford holomorphic distribution $\mathscr{F}_a$ is given by a morphism  $$\phi:E \otimes \mathcal{O}_{\mathbb{P}^4}(-a-7) \to  \Omega_{\mathbb{P}^4}^1.$$ Now, take  an element $\rho\in \Gamma_{1,5} \simeq  H_5 \rtimes SL(2,\mathbb{Z}_5) \subset Sp(4, \mathbb{Q})$. Consider the morphism
$$
\phi \circ [(\rho^{-1})^*\otimes Id]: (\rho^{-1})^*E \otimes \mathcal{O}_{\mathbb{P}^4}(-a-7) \to \Omega_{\mathbb{P}^4}^1 
$$
It follows from \cite[Satz (4.2.3)]{Wolfram} that    
$(\rho^{-1})^*E \simeq E$, which gives us a distribution 
$$  (\rho^{-1})^*\phi:=\phi \circ [(\rho^{-1})^*\otimes Id]:E \otimes \mathcal{O}_{\mathbb{P}^4}(-a-7) \to  \Omega_{\mathbb{P}^4}^1$$
such that $\Sing( (\rho^{-1})^*\phi)= \Sing( \phi)=\Sing(\mathscr{F}_a)$. But, by Theorem \ref{Determina} the distributions induced by $(\rho^{-1})^*\phi$ and $\phi$ coincide. 

\end{proof}

\begin{example}\label{Pencil-distr}(A family of Pfaff fields   singular on elliptic quintic scrolls). In \cite[Remark 4.7, (i)]{DAPHR}, the authors provide   a pencil  of Pfaff systems of rank $2$ and  degree $2$  which are invariant by a  group $G_5 \simeq H_5 \rtimes \mathbb{Z}_2$  given by
$$
\phi_{[\lambda:\mu]}: \Omega_{\mathbb{P}^4}^2 \to \mathscr{I}_{S_{[\lambda:\mu]}}\to 0,
$$
where $S_{[\lambda:\mu]}$ is a $G_5$-invariant  elliptic quintic scroll(possibly degenerate) in $\mathbb{P}^4$ given by 
$$
S_{[\lambda: \mu]}=\bigcap_{ i\in \mathbb{Z}_5} \{f_{[\lambda: \mu]}^i=0\}, 
$$
where 
$$f_{[\lambda: \mu]}^i=  \lambda^2\mu^2z_i^3+ \lambda^3\mu(z_{i+1}^2z_{i+3}+z_{i+2}z_{i+4}^2)
-\lambda\mu^3(z_{i+1}z_{i+2}^2+z_{i+3}^2z_{i+4})
-\lambda^4 z_{i}z_{i+1}z_{i+4}-\mu^4 z_{i}z_{i+2}z_{i+3}.$$

Rubtsov in \cite{Rub}  showed that the  corresponding   family of bivectors inducing $\phi_{[\lambda:\mu]}$ is 
$$
\rho_{[\lambda:\mu]}=\sum_{i\in \mathbb{Z}_5} 
(\lambda\mu z_i^2+  \mu^2z_{i+1}^2z_{i+4}-\lambda^2z_{i+2}z_{i+3}^2)( \mu\frac{\partial}{\partial z_{i+2}}\wedge \frac{\partial}{\partial z_{i+3}}+\lambda\frac{\partial}{\partial z_{i+1}}\wedge \frac{\partial}{\partial z_{i+4}})
$$
and that  such  a family  is  Poisson.
 That is, $[\rho_{[\lambda:\mu]},\rho_{[\lambda:\mu]}]=0$  for all $\lambda:\mu]\in \mathbb{P}^1$, where $[ \ ,   \ ]$
  denotes the Schouten bracket.

\end{example}
  By a simple computation, we observe that the Pfaff system
induced by  
$$
f\cdot \sum_{i\in \mathbb{Z}_5} 
( \mu\frac{\partial}{\partial z_{i+2}}\wedge \frac{\partial}{\partial z_{i+3}}+\lambda\frac{\partial}{\partial z_{i+1}}\wedge \frac{\partial}{\partial z_{i+4}}),
$$
where $f$ 
 is a polynomial of degree $2$,
$ [\lambda:\mu] =[1\pm  \epsilon^k\sqrt{5}:2] $, where  $\epsilon= e^{\frac{2\pi i}{5}}$  with  $k\in \mathbb{Z}_5$,  are decomposable bi-vectors, so   the associate symplectic foliations have dimension $2$ and degree zero, so by Proposition \ref{Distri-zero} the  leaves are planes passing  through a line. This observation  says  us the following:
\begin{propo}\label{G_5-inv}
Let $ [\lambda:\mu] =[1\pm  \epsilon^k\sqrt{5}:2] $, where  $\epsilon= e^{\frac{2\pi i}{5}}$  with  $k\in \mathbb{Z}_5$. Then  the bi-vectors
$$
\sum_{i\in \mathbb{Z}_5} 
( \mu\frac{\partial}{\partial z_{i+2}}\wedge \frac{\partial}{\partial z_{i+3}}+\lambda\frac{\partial}{\partial z_{i+1}}\wedge \frac{\partial}{\partial z_{i+4}})
$$
induce  foliations of degree zero which are invariant by the group $G_5 \simeq H_5 \rtimes \mathbb{Z}_2$. 
\end{propo}

\subsection{On the non-integrability of Horrocks--Mumford distributions}

In order to show the non-integrability of the Horrocks-Mumford distributions, we prove  the following result:

\begin{lema}\label{lema8}
Let $\mathscr{F}$ be a holomorphic foliation of codimension $k\geq 2$ on a complex manifold $X$, such that $\codim(\Sing(\mathscr{F}))\geq k+1$. If the conormal sheaf  $N {\mathscr{F}}^*$ is locally free and $\det(N {\mathscr{F}})$ is ample, then $\Sing_{k+1}(\mathscr{F})$ can not be irreducible.
\end{lema}

\begin{proof}

Suppose by contradiction that $\Sing_{k+1}(\mathscr{F}):=Z$ is irreducible and take $p\in Z$ be a generic point, i.e., $p$ is a point where $Z$ is smooth. Since the conormal sheaf $N {\mathscr{F}}^*$ is locally free,  then $\mathscr{F}$ is given by a locally decomposable holomorphic twisted and integrable $k$-form $\omega\in H^0(X,\Omega_X^k\otimes\det(N {\mathscr{F}}))$. Take a neighborhood $U$ of $p\in Z$  such that $(N {\mathscr{F}}^*)_{|U}\simeq \mathcal{O}_U^{\oplus k}$. Then 
there  exist holomorphic $1$-forms $\omega_1, \dots ,\omega_k\in H^0(U,\Omega_U^1)$ such that
$$\omega|_U=\omega_1\wedge\cdots\wedge\omega_k$$
and 
$$d\omega_i\wedge\omega_1\wedge\cdots\wedge\omega_k=0, \,\,\,\ \forall i=1,...,k.$$ 
Since $\codim(\Sing(\mathscr{F}))\geq 3$,  then by Malgrange's Theorem \cite{Malgrange1, Malgrange2},  there are  $f_1,...,f_k\in\mathcal{O}_n$ and $h\in\mathcal{O}_n^*$ such that 
$$\omega=h\cdot df_1\wedge\cdots\wedge df_k.$$
Hence $d\omega=dh\wedge df_1\wedge\cdots\wedge df_k$
$=\frac{dh}{h}\wedge\big(h.df_1\wedge\cdots\wedge df_k\big)=\theta\wedge\omega$, where $\theta=\frac{dh}{h}$ is the trace of the Bott connection, see \cite{MF}.
Now, consider $B_p$ a ball centered at $p$, of dimension $k+1$  sufficiently small and transversal to $Z$ in $p$. Then we can integrate the De Rham's class over an oriented $(k+1)$-sphere $L_p\subset B_p^*$ positively linked with $S(B_p)$. It follows from \cite{M-A} and \cite{MF} that:
$$\Res(\mathscr{F},c_1^{k+1},Z)=\frac{1}{(2\pi i)^{k+1}}\int_{L_p}\theta\wedge(d\theta)^k\cdot [Z],$$
is the Baum-Bott residue for $\mathscr{F}$ along $Z$ with respect to $c_1^{k+1}$ and  $[Z]\in H^{2k+2}(X,\mathbb{C})$ denotes  the fundamental class of the irreducible component $Z$ of $\Sing_{k+1}(\mathscr{F})$. 

On  the one hand,  since $h\in\mathcal{O}_n^*$ then $\theta=\frac{dh}{h}$ is a holomorphic $1$-form, hence $\int_{L_p}\theta\wedge(d\theta)^k=0$, so $\Res(\mathscr{F},c_1^{k+1},Z)=0$.  On the other hand, since $\mathscr{F}$ is a holomorphic foliation of codimension $k$ then by Baum-Bott formula  we have:
$$c_1^{k+1}(\det(N {\mathscr{F}}))= \Res(\mathscr{F},c_1^{k+1},Z)=0.$$
A contradiction, since $\det(N {\mathscr{F}})$ is ample. 

\end{proof}

With this in mind, we have the following result.

\begin{theorem}
Let $\mathscr{F}_a$ be  a Horrocks-Mumford distribution (\ref{distCotangFam}). Then $\mathscr{F}_a$ is  a  maximally non–integrable distribution, for $a\geq 1$.
That is, $\mathscr{F}_a$ is an Engel distribution. 
\end{theorem}

\begin{proof}
Suppose by contradiction that $\mathscr{F}_a$ is a  foliation.  By Proposition \ref{propo1} and Theorem \ref{Bertini1} we have that $\Sing(\mathscr{F}_a)$ is  irreducible  and has pure codimension $3$.  Since $$c_1(\det(N {\mathscr{F}_a}))=-c_1(\det(N {\mathscr{F}_a}^*))=2a+9>0,$$ for all $a\geq 1$, then $\det(N {\mathscr{F}})$ is ample sheaf. The result follows from    Lemma \ref{lema8}. Finally, since $\mathscr{F}_a$ is not integrable, has conormal sheaf locally free and singular locus of pure dimension, then by Proposition \ref{Engel-conormal} it is a
 Engel distribution. 

\end{proof}

\section{Moduli spaces of Horrocks-Mumford  distributions}

In this section, we study the moduli spaces  of   Horrocks-Mumford distributions. 

\subsection{Moduli spaces of holomorphic distributions}
In \cite{MOJ}, the authors described the moduli space of holomorphic distributions of codimension one on $\mathbb{P}^3$, in terms of Grothendieck's Quot-scheme for the tangent bundle, and determined under which  conditions these  varieties  are smooth, irreducible and they calculated  their  dimension. 

\begin{lema}[ \cite{MOJ} ]\label{dimMod}
Let $\mathscr{M}^{P,r,st}$ denote the open subset of $\mathscr{M}^{P}$ consisting of stable reflexive sheaves. Assume that the forgetful morphism $\varpi:\mathscr{D}ist^{P,st}\rightarrow\mathscr{M}^{P,r,st}$ is surjective, and that $\mathscr{M}^{P,r,st}$ is irreducible. If $\dim\Hom(F,TX)$ is constant for all $[F]\in\mathscr{M}^{P,r,st}$, then $\mathscr{D}ist^{P,st}$ is irreducible and 
\begin{equation*}
  \dim\mathscr{D}ist^{P,st}=\dim\mathscr{M}^{P,r,st}+\dim\Hom(F,TX)-1.    
\end{equation*}
\end{lema}

In \cite{MSM} the authors studied foliations by curves on $\mathbb{P}^3$ with locally free conormal sheaf and described their moduli spaces. We adapt this general theory to describe the moduli spaces of the Horrocks-Mumford distributions as subsheaves of the cotangent bundle.

\subsection{Moduli spaces of Horrocks-Mumford distributions }

Consider the Horrocks-Mumford  distribution:
\begin{equation*}
    \mathscr{F}_a:0\rightarrow E(-a-7)\rightarrow\Omega_{\mathbb{P}^4}^1\rightarrow \mathcal{Q}_{\mathscr{F}}\rightarrow 0 \,\,\,\, , \,\,\,\ a\geq 1.
\end{equation*}

Let $P=P_{\mathscr{F}_a}(t)=\chi(N {\mathscr{F}_a}^*(t))$ be the Hilbert polynomial of the stable bundle $N {\mathscr{F}_a}^*=E(-a-7)$, and we denoted by $d_a=\deg(\mathscr{F}_a)=2a+6$ and $c=c_2(N {\mathscr{F}_a}^*)=c_2(E(-a-7))$, for $a\geq 1$. Then:
\begin{equation*}
\begin{aligned}
P: ={} & P_{\mathscr{F}_a}(t) \\
     ={} & 2 + \frac{25}{12}(-d_a-3+2t)+\frac{35}{24}\big((-d_a-3+2t)^2-2(c-2at+t^2-9t)\big)+ \\
      & + \frac{5}{12}\big((-d_a-3+2t)^3-3(-d_a-3+2t)(c-2at+t^2-9t)\big) +\\
      & + \frac{1}{24}\big((-d_a-3+2t)^4-4(-d_a-3+2t)^2(c-2at+t^2-9t)+2(c-2at+t^2-9t)^2\big).
\end{aligned}
\end{equation*}
Now, we  denoted by
\begin{equation*}
     {\rm HM}\mathscr{D}ist(2a+6)
\end{equation*}
the moduli spaces of Horrocks--Munford holomorphic distribution of  degree $d_a=2a+6$. We have that $ {\rm HM}\mathscr{D}ist(2a+6)$ is a subspace of the moduli space of dimension two holomorphic distributions, of degree $2a+6$, on $\mathbb{P}^4$ whose conormal sheaves have 
 Hilbert polynomial equal to  $P$. Denote by 
\begin{equation*}
    \mathscr{M}_{\mathbb{P}^4}(-d_a-3,c)=\mathscr{M}_{\mathbb{P}^4}(-2a-9, a^2+9a+24)    
\end{equation*}
the moduli space of  the Horrocks-Mumford  bundles, with Chern classes $c_1=-d_a-3$ and  $c_2=c$.
Let us consider the forgetful morphism
\begin{alignat*}{2}
  \varpi_a:  {\rm HM}\mathscr{D}ist(2a+6)&\longrightarrow& \mathscr{M}_{\mathbb{P}^4}(-d_a-3,c) \\
  [\mathscr{F}_a]&\longmapsto&[E(-a-7)]. 
\end{alignat*}
Twisting by $\mathcal{O}_{\mathbb{P}^4}(a+4)$ we obtain the isomorphism $$\mathscr{M}_{\mathbb{P}^4}(-2a-9, a^2+9a+24)\simeq\mathscr{M}_{\mathbb{P}^4}(-1,4).$$ 
Hence by Theorem \ref{Wolfram} we have that $\mathscr{M}_{\mathbb{P}^4}(-d_a-3,c)$ is a nonsingular variety of dimension $24$. In addition, by Bertini type Theorem we have that  each $E(-a-7)$ is the conormal sheaf of a dimension two distribution $\mathscr{F}_a$, thus $\varpi_a$ is a surjective map. Therefore $\Ima(\varpi_a)=\mathscr{M}_{\mathbb{P}^4}(-d_a-3,c)$ is irreducible.

\begin{theorem}
The moduli space $  {\rm HM}\mathscr{D}ist(2a+6)$ of dimension two holomorphic distributions (\ref{distCotangFam}) is an irreducible quasi-projective variety of dimension
\begin{equation*}
    \frac{1}{3}a^4+\frac{23}{3}a^3+\frac{343}{6}a^2+\frac{899}{6}a+98,
\end{equation*}
for $a\geq 1$.
\end{theorem}

\begin{proof}
The fibers of $\varpi_a$ over a point $[E(-a-7)]\in\mathscr{M}_{\mathbb{P}^4}(-d_a-3,c)$ is the set $\mathscr{D}ist(E(-a-7))$ of all distributions whose conormal sheaf is Horrocks-Mumford:
\begin{equation*}
\mathscr{D}ist(E(-a-7)):=\{\varphi\in\mathbb{P}\Hom(E(-a-7),\Omega_{\mathbb{P}^4}^1); \,\,\ \ker\varphi=0 \,\,\ \textrm{and} \,\,\ \Coker\varphi \,\,\ \textrm{is torsion free}\}.    
\end{equation*}
That is an open subset of $\mathbb{P}\Hom(E(-a-7),\Omega_{\mathbb{P}^4}^1)$, see \cite[Section 2.3]{MOJ}. Hence
\begin{equation*}
    \dim\mathscr{D}ist(E(-a-7))=\dim\Hom(E(-a-7),\Omega_{\mathbb{P}^4}^1)-1.
\end{equation*}
We claim that $\dim\Hom(E(-a-7),\Omega_{\mathbb{P}^4}^1)=h^0(E(-a-7)\otimes \Omega_{\mathbb{P}^4}^1)$ is constant.
Indeed, twisting Euler's exact sequence by $E(a+2)$ bundle and taking cohomology, we have:
\begin{equation*}
    \cdots\rightarrow H^i(E(a+2)\otimes\Omega_{\mathbb{P}^4}^{1})\rightarrow H^i(E(a+1))^{\oplus 5}\rightarrow H^i(E(a+2))\rightarrow\cdots
\end{equation*}
By table \ref{table}, for $a\geq 1$, we have that $h^i(E(a+1))=0$ for $i=1,2,3,4$ and $h^i(E(a+2))=0$ for $i=2,3,4$, thus $h^i(E(a+2)\otimes \Omega_{\mathbb{P}^4}^{1})=0$ for $i=1,2,3,4.$
Therefore by (\ref{eulerHM}) we have: 
\begin{eqnarray*}
    h^0(E(a+2)\otimes\Omega_{\mathbb{P}^4}^1) &=& 5\chi(E(a+1))-\chi(E(a+2)) \\ 
    \\
      &=& \frac{1}{3}a^4+\frac{23}{3}a^3+\frac{343}{6}a^2+\frac{899}{6}a+75,
\end{eqnarray*}
for $a\geq 1$.
So, all the fiber in $\mathscr{D}ist(E(-a-7))$ has the same dimension. Since $\Ima\varpi_a=\mathscr{M}_{\mathbb{P}^4}(-d_a-3,c)$ is irreducible, then by fiber dimension theorem (see \cite[pag. 77]{Shafarevich}), we have that $  {\rm HM}\mathscr{D}ist(2a+6)$ is an irreducible variety.
Finally, by fiber dimension theorem, Lemma \ref{dimMod} and Theorem \ref{Wolfram}, we have:
\begin{equation*}
    \begin{split}
        \dim   {\rm HM}\mathscr{D}ist(2a+6) &= \dim\mathscr{M}_{\mathbb{P}^4}(-d_a-3,c)+\dim\Hom(E(-a-7),\Omega_{\mathbb{P}^4}^1)-1 \\
        \\
        &= 24 + h^0(E(a+2)\otimes\Omega_{\mathbb{P}^4}^1) -1 \\
        \\
        &= \frac{1}{3}a^4+\frac{23}{3}a^3+\frac{343}{6}a^2+\frac{899}{6}a+98,
    \end{split}
\end{equation*}
for $a\geq 1$.

\end{proof}

\subsection{Proof of Theorem \ref{Conformal--HM}}
  Consider the smooth quadric 3-fold   $\mathbb{Q}_3\in\mathbb{P}^{4}$   defined  as   
$$
\mathbb{Q}_3=\{ -2z_0z_{4}+\sum_iz_i^2=0\}.
$$
We have on $\mathbb{Q}_3$  a flat  holomorphic conformal structure given by 
$$
\omega_0 = -2dz_0\otimes dz_{4}   + \sum_{i=0}^4 dz_i\otimes dz_i \in  H^0 ( \mathbb{Q}_3, \mbox{Sym}^2\Omega^1_{\mathbb{Q}_3}(2)). 
$$
The Lie ball is the  manifold defined  by 
$
\mathbb{D}^{IV}_3 = \{ z\in\mathbb{C}^3: 2\sum\mid z_i\mid^2<1 +\mid\sum z_i\mid^2<2\}.
$
There exists an embedding
$
\Phi: \mathbb{D}^{IV}_3  \to  \mathbb{Q}_3
$
 given by $\Phi(z)= [ 1 : z : \frac{1}{2}\sum z_i^2 ]$. 
 The image $\Phi(\mathbb{D}^{IV}_3):=\mathcal{H}_2$ is called by the  Siegel upper half space  degree and $\omega:=(\omega_0)|_{\mathcal{H}_2}$ defines a flat holomorphic conformal structure on   $\mathcal{H}_2$. Moreover, $g$ 
  is invariant by a group  $\Gamma_{1,5}\simeq  H_5 \rtimes SL(2,\mathbb{Z}_5) \subset Sp(4, \mathbb{Q})$. In particular, this induces a flat holomorphic conormal structure, let us call by $g_0$,  on the quotient space $\mathcal{H}_2/\Gamma_{1,5}$.
There is an isomorphism $h: \mathcal{A}_{1,5}\to \mathcal{H}_2/\Gamma_{1,5} $   between  $\mathcal{A}_{1,5}$
   the  moduli space  of the abelian surfaces with $(1,5)$-polarization and level-5-structure is isomorphic to  $\mathcal{H}_2/\Gamma_{1,5}$, and in particular  $g_0:=h^*\omega$  is  a flat holomorphic conormal structure on $\mathcal{A}_{1,5}$.   Now,  take  $E$ a Horrocks-Mumford bundle on $\mathbb{P}^4$.  The space  $H^0(E)$ has dimension $4$ and the projective space  $\mathbb{P}(H^0(E))\simeq \mathbb{P}^3$ is birational to the 
 the moduli space    $\mathcal{A}_{1,5}$
   of the abelian surfaces with $(1,5)$-polarization and level-5-structure \cite{HM}. Denote  by $f: \mathbb{P}(H^0(E))
    \dashrightarrow \mathcal{A}_{1,5}\simeq \mathcal{H}_2/\Gamma_{1,5} $ the  corresponding $\Gamma_{1,5}$-equivariant birational map. For all $a\geq 1$, we consider a distribution  
$$
  \phi: E(-a-7)\rightarrow \Omega_{\mathbb{P}^4}^1, 
$$
where $ \phi\in \mathcal{A}\subset  \mathbb{P} H^0(\mbox{Hom}(E(-a-7),\Omega_{\mathbb{P}^4}^1))\simeq \mathbb{P}^M$  is a Zariski open  with 
$$
 M= \frac{1}{3}(d-5)^4+\frac{23}{3}(d-5)^3+\frac{343}{6}(d-5)^2+\frac{899}{6}(d-5)+74. 
$$
Twisting by  $\mathcal{O}_{\mathbb{P}^4}(a+7)$ and take cohomology we obtain a  four-dimensional linear subspace  
$H^0( \phi(E))  \subset H^0(\Omega_{\mathbb{P}^4}^1(d+2)),
$
with $d:=a+5$, generated by  twisted 1-forms  $\alpha_0:=\phi\circ s_0, \alpha_1:=\phi\circ s_1,\alpha_2:=\phi\circ s_2,\alpha_3:=\phi\circ s_3 \in  H^0(\Omega_{\mathbb{P}^4}^1(d+2))$, where  $ \{s_0,   s_1,  s_2, s_3\}$ is a base  for $H^0(E)$. Complete a base $\alpha_0, \alpha_1,\alpha_2,\alpha_3, \omega_1,\cdots , \omega_{N-4}$, where $N=\dim H^0(\Omega_{\mathbb{P}^4}^1(d+2))$,
and consider the following  linear projection  
$\psi_{ \phi}:\mathbb{P}(H^0(\Omega_{\mathbb{P}^4}^1(d+2)))  \dashrightarrow \mathbb{P}(H^0(E))$ defined by 
$$\pi_{ \phi}(\lambda_0\alpha_0+\lambda_1\alpha_1+\lambda_2\alpha_2+\lambda_3\alpha_3+\dots+ \lambda_{N-4} \omega_{N-4})=(\lambda_0:\lambda_1:\lambda_2:\lambda_3).$$
Then, by composing  these maps we get  the diagram 
\begin{center}
\begin{tikzcd}[row sep=large]
 \mathbb{P}(H^0(\Omega_{\mathbb{P}^4}^1(d+2)))  \ar[dd, " \psi_{ \phi} "]  \ar[ddrr, bend left=20,dashed,  " \pi_{ \phi}"]&&    \\
 &&    \\
 \mathbb{P}(H^0(\phi(E))) \simeq  \mathbb{P}(H^0(E))     \ar[r, dashed, "  f"]  \ar[uu, bend left=45,dashed,  " s_{ \phi}"]  & \mathcal{A}_{1,5}  \ar[r, "  h"]& \mathcal{H}_2/\Gamma_{1,5}
\end{tikzcd}
\end{center}
with the rational section $s_{ \phi}(\lambda_0:\lambda_1:\lambda_2:\lambda_3)=[\lambda_0\alpha_0+\lambda_1\alpha_1+\lambda_2\alpha_2+\lambda_3\alpha_3] \in  \mathbb{P}(H^0(\Omega_{\mathbb{P}^4}^1(d+2)))$. Then,    by construction the  image of   $s_{ \phi}$
   consists   of codimension one distributions, of degree $d$,
   induced by 1-foms $$\omega_{\lambda}:=
   \lambda_0 (\phi\circ s_0)+\lambda_1 (\phi\circ s_1)+\lambda_2 (\phi\circ s_2)+\lambda_3 (\phi\circ s_3).
  $$
   It follows from  Proposition \ref{Aut-HM} that $\phi$ is  invariant by $\Gamma_{1,5}$, then for all $  \lambda \in   \mathbb{P}(H^0(E))$ we have that  $\omega_{\lambda}$ is so, since $s_i$ is $\Gamma_{1,5}$-invariant for all $i=0,1,2,3$. Moreover, the codimension one  distribution induced by $\omega_{\lambda}$  is singular along to the  abelian surface $S_{\lambda}=\{\lambda_0   s_0+\lambda_1   s_1+\lambda_2   s_2+\lambda_3   s_3=0\}$(with $(1,5)$-polarization and level-$5$-structure), since 
   $ S_{\lambda}\subset \mbox{Sing}(\omega_{\lambda})$. 
   We have  the  $\Gamma_{1,5}$-equivariant   conformal structure  $g_{\phi}=\pi_{\phi}^*g_0$ which   is  clearly flat.  
 From \cite{Sato} we have that  $g_0$   degenerates along a  surface $\Delta_0$ of degree $10$ which is a cone over a  rational sextic curve in $C\subset  \mathbb{P}(H^0(E)) \simeq   \mathbb{P}(H^0(\phi(E))) \simeq \mathbb{P}^3$, and    $\Delta_0$ is the locus of  singular  abelian surfaces with $(1,5)$-polarization and level-$5$-structure, see  \cite{BMoore}.
In particular, $g_{\phi}=\pi_{\phi}^*g_0$ degenerate along the hypersurface  $\Delta_{\phi}:=\psi_{\phi}^{-1}(\Delta_0)$     of degree $10$ which is a cone over  $C\subset \mathbb{P}(H^0(E))$ 
 and each  distribution which belongs to $\Delta_{\phi}$ is singular along  to singular abelian surfaces with $(1,5)$-polarization and level-$5$-structure. 
Therefore,  from \cite[Theorem 0.1]{BHM} we conclude that all distribution induced by a 1-form $\omega\in \Delta_{\phi}$ vanish along either:  
 \begin{itemize}
    \item a  translation scroll associated to a normal elliptic quintic curve;   
     \item  or the tangent scroll of a normal elliptic quintic curve;
     \item   or a quintic elliptic scroll carrying a multiplicity-2 structure;
     
     \item or a union of five smooth quadric surfaces;
     
     \item   or a union of five planes with a multiplicity-2 structure.
    \end{itemize}


\begin{thebibliography}{99} 

\bibitem{AnCM}
C. Anghel, I. Coanda, N. Manolache,
 \textit{A property of five lines in  $\mathbb{P}^3$ and four generated 4-instantons}, Communications in Algebra, (2020).

\bibitem{CM} C. Araujo; M.  Corr\^ea, Jr. 
\textit{On degeneracy schemes of maps of vector bundles and applications to holomorphic foliations}. Math. Z., 276(1-2):505-515, 2014.


\bibitem{ACMa}
C. Araujo, M. Corr\^ea, A. Massarenti,  \textit{Codimension one Fano distributions on Fano manifolds}, Commun. Contemp. Math. 20 (2018), no. 5, 1750058.



\bibitem{At}
M. F. Atiyah, \textit{Complex analytic connections in fiber bundles}, Trans. AMS, 85 (1957), 181-
207.

\bibitem{DAPHR} A. Aure, W. Decker, S. Popescu, K. Hulek, K. Ranestad, \textit{  Syzygies of Abelian and Bielliptic Surfaces in $\mathbb{P}^4$} 
, Internat. J. Math., 4 (1993), 873–902.

\bibitem{AG}
A. Awane, M. Goze,\textit{  Pfaffian systems, k-symplectic systems}. Kluwer Acad. Pub., Dordrecht (2000).



\bibitem{BHM}
W. Barth,    K. Hulek,    and R. Moore,  \textit{   Degenerations of Horrocks-Mumford surfaces}. Math. Ann. 277 (1987),
735-755. 

\bibitem{BMoore}
W. Barth,  R. Moore, \textit{   Geometry in the space of Horrocks--Mumford surfaces},
Topology 28(2) (1989), 231-245.

\bibitem{BB}
P. Baum, R. Bott, \textit{ Singularities of holomorphic foliations}, J. Differential Geom. 7 (1972), 279-342.

\bibitem{BCGGG}
R. L Bryant, S. S Chern, R.B Gardner, H. L.  Goldschmidt, P. A.  Griffiths,  \textit{ Exterior
differential systems}, Mathematical Sciences Research Institute Publications, Vol. 18,
Springer-Verlag, New York, 1991.

\bibitem{Brunella}
M. Brunella,  \textit{ A positivity property for foliations on compact K\"ahler manifolds}. Internat.
J. Math. 17 (2006) 35-43.

\bibitem{BSt}
D. Burde, C.  Steinhoff,   \textit{ Classification of Orbit Closures of 4-Dimensional
Complex Lie Algebras}. Journal of Algebra 214-2 (1999), 729-739.







\bibitem{MOJ} O. Calvo--Andrade, M. Corr\^ea, M. Jardim. 
\textit{Codimension One Holomorphic Distributions on the Projective three space}. International Mathematics Research Notices. Vol.  2020, 23, pg.  9011-9074. 2020.




\bibitem{Cas}
 G. Castelnuovo, \textit{ Ricerche di geometria della retta nello spazio a quattro dimensioni}, Atti
R. Ist. Veneto Sc. 2 (1891), 855-90


\bibitem{CD}
D. Cerveau, J.  D\'eserti. \textit{ Birational maps preserving the contact structure on $\mathbb{P}_{\mathbb{C}}^3$}. J. Math. Soc. Japan 70 (2018), no. 2, 573-615.





\bibitem{CD2}
D. Cerveau, J.  D\'eserti. \textit{ Feuilletages et actions de groupes sur les espaces projectifs}, M\'em. Soc. Math.
Fr. (N.S.) (2005), no. 103, vi+124 pp. (2006).


\bibitem{CLN}
D. Cerveau, A. Lins Neto, \textit{ Irreducible components of the space of holomorphic foliations of
degree two in $\mathbb{CP}(n)$, $n\geq 3$}. Ann. of Math. 143 (1996), p. 577-612.

\bibitem{CLN2}
D. Cerveau and A. Lins Neto.\textit{ Codimension two holomorphic foliations}. J. Differential Geom., 113(3) :385-416, 2019.


\bibitem{Chang}
M. C. Chang,  \textit{Classification of Buchsbaum subvarieties of codimension 2 in projective space},
J. Reine Angew. Math. 401 (1989) 101-112. 

\bibitem{Chang2}
M. C. Chang, \textit{ Characterization of arithmetically Buchsbaum subschemes of codimension 2
in  $\mathbb{P}^n$}, J. Differential Geom. 31 (1990) 323-341.


\bibitem{Corr} M.  Corr\^ea, A. Muniz, 
\textit{Holomorphic foliations of degree two and arbitrary dimension}.  arXiv:2207.12880v3. 2023. 

\bibitem{M-A} M.  Corr\^ea, A. Fernández-Pérez.
\textit{Absolutely k-convex domains and holomorphic foliations on homogeneous manifolds}. J. Math. Soc. Japan \textbf{69} (2017), no. 3, 1235–1246. 

\bibitem{MMR} M. Corrêa, M. Jardim; R. V. Martins.
\textit{On the singular scheme of split foliations}. Indiana Univ. Math. J. \textbf{64} (2015), no. 5, 1359–1381.




\bibitem{MSM} M. Corr\^ea, M. Jardim, S. Marchesi.  
\textit{Classification of the invariants of foliations by curves of low degree on the three-dimensional projective space}.  Rev. Mat. Iberoam. 39 (2023), no. 5, pp. 1641-1680.  




\bibitem{MF} M. Corr\^ea; F. Lourenço. 
\textit{Determination of Baum-Bott residues of higher codimensional foliations}. Asian J. Math. \textbf{23} (2019), no. 3, 527–538.


\bibitem{CMaz} M. Corr\^ea; L. G. Maza, \textit{Engel Theorem through singularities}. Bulletin des Sciences Math\'ematiques,  140, (6),  675-686. 2016

\bibitem{CLPdeg3}
R. C. da Costa, R. Lizarbe, and J. V. Pereira, \textit{ Codimension one foliations of degree three on projective space}s, Bulletin des Sciences Math\'ematiques. Volume 174, February 2022, 103092. 






\bibitem{Wolfram} W. Decker. 
\textit{Das Horrocks-Mumford-Bündel und das Modul-Schema für stabile 2-Vektorbündel über $\mathbb{P}^4$ mit $c_1=-1,c_2=4$}. (German) [The Horrocks-Mumford bundle and the moduli scheme for stable rank 2 vector bundles over $\mathbb{P}^4$ with $c_1=-1,c_2=4$] Math. Z. \textbf{188} (1984), no. 1, 101–110.

\bibitem{Wolfram2} W. Decker; F. Schreyer. 
\textit{On the uniqueness of the Horrocks-Mumford bundle}. Math. Ann. \textbf{273} (1986), no. 3, 415–443. 

\bibitem{Wolfram3} W. Decker. 
\textit{Monads and cohomology modules of rank $2$ vector bundles}, Compositio Math. \textbf{76}
(1990), 7–17.


\bibitem{Demailly}
J.-P. Demailly,\textit{ On the Frobenius integrability of certain holomorphic p-forms}. Complex
geometry (G\"ottingen, 2000), 93-98, Springer, Berlin, 2002.








\bibitem{Ei} D. Eisenbud. 
\textit{Commutative algebra. With a view toward algebraic geometry}. Graduate Texts in Mathematics, \textbf{150}. Springer-Verlag, New York, 1995. xvi+785 pp.




\bibitem{Faenzi-Fania}
 D. Faenzi, M.L. Fania,  \textit{Skew-symmetric matrices and Palatini scrolls}, Math. Ann. 347
(2010), 859-88


\bibitem{Fano} G. Fano.
\textit{Reti di complessi lineari dello spazio $S_5$ aventi una rigata assegnata di rettecentri}. Rend. Reale Accad. Naz. Lincei, no. XI (1930), 227-23.


\bibitem{FR} E.A. Fern\'andez-Culma, N. Rojas, \textit{ On the classification of 3-dimensional complex Hom-Lie algebras},
Journal of Pure and Applied Algebra,
Volume 227, Issue 5, May 2023. 




\bibitem{Forsyth}
A. Forsyth. \textit{Theory of differential equations}. Cambridge University Press, Cambridge, 2012.

\bibitem{FH} 
W. Fulton and J. Harris, \textit{ Representation Theory: A first course}, Springer-Verlag New York Inc., 1991.

\bibitem{FL} W. Fulton, R. Lazarsfeld. 
\textit{On the connectedness of degeneracy loci and special divisors}. Acta Math. \textbf{146}, 3-4 (1981) 271-283.

\bibitem{GJM} 
H.  Galeano, M. Jardim,  A. Muniz
\textit{ Codimension one distributions of degree 2 on the three-dimensional projective space.} J. Pure Appl. Algebra 226, No. 2,  32 p. (2022). 


\bibitem{GHS}
H. H. Glover, W. D. Homer, R. E. Stong, \textit{ Splitting the tangent bundle of projective
space}, Indiana Univ. Math. J. 31 (1982), no. 2, 161-166. 
  

\bibitem{GJ}
D. Guimar\~aes, M. Jardim
\textit{Moduli spaces of quasitrivial sheaves on the three dimensional projective space}, 	arXiv:2108.01506, 2021.
 



\bibitem{Har2} R. Hartshorne. 
\textit{Stable Reflexive Sheaves}.  Math. Ann. \textbf{254} (1980), no. 2, 121–176.

   

\bibitem{HPe}
A.  Horing and T.  Peternell. \textit{Algebraic integrability of foliations with numerically trivial
canonical bundle}. Invent. Math., 216(2):395-419, 2019. 

\bibitem{HM} G. Horrocks; D. Mumford.
\textit{A rank $2$ vector bundle on $\mathbb{P}^4$ with $15,000$ symmetries}. Topology \textbf{12} (1973), 63–81.  14F05 

\bibitem{Jou}
J-P. Jouanoulou, \textit{Equations de Pfaff alg\'ebriques}. LNM 708, Springer Verlag, Berlin, (1979).

 
 
 



\bibitem{LPT2}
F. Loray, J. Vit\'orio Pereira, and F. Touzet, \textit{ Foliations with trivial canonical bundle on Fano 3-folds}.
Math. Nachr. 286(8-9), 921-940 (2013).

\bibitem{Maeda} H. Maeda. 
\textit{Construction of vector bundles and reflexive sheaves}. Tokyo J. Math. \textbf{13} (1990), no. 1, 153–162.
14J60 (14F05)

\bibitem{Malgrange1} B. Malgrange. 
\textit{Frobenius avec singularités. I. Codimension un.}
(French) Inst. Hautes Études Sci. Publ. Math. No. \textbf{46} (1976), 163–173.

\bibitem{Malgrange2} B. Malgrange. 
\textit{Frobenius avec singularités. II. Le cas général}. (French) Invent. Math. \textbf{39} (1977), no. 1, 67–89.


 




\bibitem{Oko} Ch. Okonek; M. Schneider; H. Spindler. 
\textit{Vector Bundles on Complex Projective spaces.} Progress in Math. \textbf{3}, Birkha\"user 1978.

\bibitem{Okonek1}
C. Okonek, \textit{Reflexive Garben auf $\mathbb{P}^4$},  Math. Ann. 260, 211-237. (1982).

\bibitem{Okonek2}
C. Okonek, \textit{ Moduli reflexiver Garben und Flachen vom kleinem Grad in $\mathbb{P}^4$}. Math. Z. 184, 549-572, (1983). 

\bibitem{Ottaviani} G. Ottaviani. 
\textit{Variet\`a proiettive de codimension picola}. Quaderni INDAM, Aracne, Roma, (1995).

\bibitem{Ottaviani2}
G. Ottaviani, \textit{ On 3-folds in  $\mathbb{P}^5$ which are scrolls}, Ann. Scuola Norm. Sup. Pisa Cl. Sci. (4)
19, no. 3, 451-471 (1992).

\bibitem{Palatini}
 F. Palatini, \textit{ Sui sistemi lineari di complessi lineari di rette nello spazio a cinque dimensioni}, Atti del R. Ist. Veneto di Scienze, Lettere e Arti, no. 60 (1901), 371-38

\bibitem{Polishchuk}
 A. Polishchuk,  {\it  Algebraic geometry of Poisson brackets. Algebraic geometry}, 7. J. Math. Sci.
84 (1997), no. 5,  1413-1444.


 \bibitem{Polishchuk2}
  A. Polishchuk, {\it Feigin-Odesskii brackets, syzygies, and Cremona transformations}, Journal of Geometry and Physics. Volume 178, 2022.  

 

\bibitem{Rub} V. Rubtsov, {\it Quadro-cubic Cremona transformations and Feigin--Odesskii--Sklyanin algebras with 5 generators}, in: Recent Developments in Integrable
Systems and Related Topics of Mathematical Physics, Kezenoi-Am, Russia, 2016, Springer, 2018, pp. 75-106.

  

 


\bibitem{Sato}
T. Sato, 
\textit{The flat holomorphic conformal structure on the Horrocks-Mumford orbifold}  Math. Nachr.   Vol. 163,  297-304   1993. 

\bibitem{Shafarevich} I. R. Shafarevich.
\textit{Basic algebraic geometry. 1. Varieties in projective space}. Second edition. Translated from the 1988 Russian edition and with notes by Miles Reid. Springer-Verlag, Berlin, 1994. xx+303 pp. 


\bibitem{PS}
L. Sol\'a Conde, F. Presas, \emph{Holomorphic Engel Structures. }
Rev. Mat. Complut. (2014) 27: 327.

\bibitem{Su} H. Sumihiro. 
\textit{Determinantal varieties associated to rank two vector bundles on projective spaces and splitting theorems}. (English summary)  Hiroshima Math. J. \textbf{29} (1999), no. 2, 371–434.  14M12 (14J60)  

 


\bibitem{Vo}
T. Vogel, \emph{ Existence of Engel structures}
Ann. of Math. (2) 169 (2009), no. 1, 79-137.

\bibitem{Wahl}
J. M. Wahl,  \emph{  A cohomological characterization of $\mathbb{P}^n$}, Invent. Math. 72 (1983), no. 2, 315-322

 

\end{thebibliography}
\end{document}